\documentclass[11pt,a4paper]{article}

\usepackage[T1]{fontenc}
\usepackage{lmodern}
\usepackage[utf8]{inputenc}

\usepackage[left=2.45cm, top=2.45cm,bottom=2.45cm,right=2.45cm]{geometry}
\usepackage{amsmath}
\usepackage{amssymb,amsthm,graphicx,xcolor,tikz}

\usepackage[british]{babel}
\usepackage{bbm}

\usepackage[
	bookmarks=true,
	bookmarksnumbered=true,
	bookmarksopen=true,
	unicode=true,
	pdftoolbar=true,
	pdfmenubar=true,
	pdffitwindow=false,
	pdfstartview={FitH},
	pdftitle={},
	pdfauthor={},
	pdfsubject={},
	pdfcreator={},
	pdfproducer={},
	pdfkeywords={},
	pdfnewwindow=true,
	colorlinks=true,
	linkcolor=black,
	citecolor=black,
	filecolor=black,
	urlcolor=black]{hyperref}
	
\makeatletter
	\def\@cite#1#2{[\textbf{#1}\if@tempswa , #2\fi]}	
	\def\@biblabel#1{[#1]}								
\makeatother

\numberwithin{equation}{section}
\numberwithin{figure}{section}

\newtheorem {theorem}{Theorem}[section]

\newtheorem {lemma}[theorem]{Lemma}

\theoremstyle{definition}
\newtheorem{definition}[theorem]{Definition}
\newtheorem {remark}[theorem]{Remark}
\newtheorem {example}[theorem]{Example}

\renewcommand{\H}{\mathbb{H}}

\newcommand{\R}{\mathbb{R}}
\renewcommand{\S}{\mathbb{S}}
\newcommand{\dS}{\mathrm{d}\mathbb{S}}

\newcommand{\cF}{\mathcal{F}}

\newcommand{\cC}{\mathcal{C}}

\newcommand{\by}{\mathbf{y}}
\newcommand{\bx}{\mathbf{x}}

\newcommand{\bz}{\mathbf{z}}
\newcommand{\bu}{\mathbf{u}}
\newcommand{\bn}{\mathbf{n}}
\newcommand{\be}{\mathbf{e}}
\newcommand{\bv}{\mathbf{v}}
\newcommand{\bw}{\mathbf{w}}
\newcommand{\bo}{\mathbf{o}}

\DeclareMathOperator{\Vol}{Vol}
\DeclareMathOperator{\rad}{rad}
\DeclareMathOperator{\roll}{roll}

\DeclareMathOperator{\interior}{int}
\DeclareMathOperator{\closure}{cl}

\DeclareMathOperator{\bd}{bd}

\DeclareMathOperator{\as}{as}
\DeclareMathOperator{\Sp}{Sp}

\DeclareMathOperator{\mcd}{mcd}

\newcommand{\dint}{\mathrm{d}}
   
\renewcommand{\phi}{\varphi}

\begin{document}

\title{\bfseries Floating Bodies and Duality \\ in Spaces of Constant Curvature}

\author{%
    Florian Besau\footnotemark[1]%
    \and Elisabeth~M.~Werner\footnotemark[2]%
}

\date{}
\renewcommand{\thefootnote}{\fnsymbol{footnote}}
\footnotetext[1]{%
    Technische Universität Wien, Austria. Email: florian.besau@tuwien.ac.at
}

\footnotetext[2]{%
    Case Western Reserve University, USA. Email: elisabeth.werner@case.edu
}

\maketitle

\begin{abstract}\noindent
    We investigate a natural analog to Lutwak's $p$-affine surface area in $d$-dimensional spherical, hyperbolic and de Sitter space. In particular, we show that these curvature measures appear naturally as the volume derivative of floating bodies of non-Euclidean convex bodies conjugated by duality, such as spherical, hyperbolic and de Sitter convex bodies.
    
    We provide a unifying framework by establishing a real-analytic version of this relation controlled by the constant curvature of the $d$-dimensional real space form. These new curvature measures relate in two distinctly different ways to curvature measures on Euclidean space, one of which is Lutwak's centro-affine invariant $p$-affine surface area, and the other is related to a rigid-motion invariant curvature measure that appears naturally as the volume derivative of Schneider's mean-width separation body.

    \smallskip\noindent
    \textbf{Keywords.} affine surface area, $L_p$-affine surface area, floating body, weighted floating body, spherical convex body, hyperbolic convex body, de Sitter convex body, illumination body, separation body

    \smallskip\noindent
    \textbf{MSC 2020.} Primary  52A20; Secondary 28A75, 52A55, 53A35.
\end{abstract}

\section{Introduction}\label{sec:introduction}

The \emph{floating body} and the \emph{affine surface area} of a convex body are closely tied constructions in affine geometry. The concept of affine floating bodies dates back to Dupin \cite{Dupin:1814, Dupin:1822} at the start of the 19th century and can be seen, together with the introduction of the affine normal by Transon \cite{Transon:1841} in 1841, as the first results in affine differential geometry. Much later, at the start of the 20th century, Tzitzéica introduced affine spheres and subsequently more and more geometers investigated affine invariant properties of curves and surfaces in the spirit of Klein's Erlangen program. See, for example, the introduction in the monograph \cite{LSZH:2015} and the references therein for a brief history on the development of affine differential geometry in the 20th century.

\smallskip
The name ``floating body'' or ``floating surface'' is coined by the fact, that by Archimedes' principle a body floating in a liquid has always the same ratio of volume above and below the water surface independent of its orientation.
The floating bodies associated with a given convex body generate a one-parameter family, controlled by the ratio between the above and below part. This family of convex bodies converge from the inside to the original body in an affine covariant way.
Taking the volume derivative with respect to the family of floating bodies gives rise to Blaschke's affine surface area. This relation between the floating body and the affine surface area was first established by Blaschke in 1923 for smooth convex bodies in dimension 2 and 3 and generalized by Leichtweiss \cite{LW1:1986} to general dimensions under some curvature conditions. Finally, in work of Schütt and Werner \cite{SW:1990}, it was established that the volume derivative of the \emph{convex floating body} exists for all convex bodies and gives rise to the affine surface area for general convex bodies. Other extensions were obtained by Leichtweiss \cite{LW2:1986,LW:1989} and by a completely different method by Lutwak \cite{Lutwak:1991} who extended an approach by Petty \cite{Petty:1974}. It later turned out that all these extensions give the same notion, proved in \cite{DH:1995, Schutt:1993} respectively. Blaschke's affine surface area was characterized as essentially the only upper-semicontinuous equi-affine invariant valuation on all convex bodies by Ludwig and Reitzner \cite{LudRei:1999}. See \cite{LM:2023} for a recent survey on affine invariant valuations.
The floating body is also connected to Ulam's problem on convex bodies floating in equilibrium, with recently has seen major progress due to Ryabogin \cite{Ryabogin:2022,Ryabogin:2023} and Florentin et al. \cite{FSWZ:2022}.

\smallskip
We showed in \cite{BW:2016, BW:2018} that the convex floating body of a convex body naturally extends to all spaces of constant curvature, that is, real space forms. We established a real analytic analog of the affine surface area by considering the volume derivative of our intrinsic notion of convex floating body in the space form. Applications of this non-Euclidean floating body and its associated curvature measure were obtained in \cite{BLW:2018,BRT:2021,BT:2020}, where the asymptotic behavior of random volume approximation of non-Euclidean convex bodies was investigated. These results also align with a new interest in stochastic geometry to explore spaces of constant curvatures, see, for example, \cite{BHRS:2017, BGRSTW:2023, BHT:2023, GKT:2022, HHT:2021, HOOT:2023, KRT:2022, OT:2023}. Also see \cite{BS:2020,BS:2022,BS:2023,Fodor:2021} for more recent results on convex bodies in spaces of constant curvature.

\smallskip
In this article we aim to combine the floating body construction with the natural duality on non-Euclidean convex bodies in spaces of constant curvature to investigate the volume derivative of the family of floating bodies \emph{conjugated} by duality. If the curvature is zero, i.e., if we consider a Euclidean vector space of dimension $d\geq 2$, this volume derivative was studied by Meyer and Werner \cite{MW:2000}, where it was shown that it gives rise to $L_p$-affine surface area \cite{Lutwak:1996} for $p=-d/(d+2)$, a centro-affine invariant curvature measure. For a recent excellent presentation of how centro-affine differential geometry is at play in the Brunn--Minkowski theory we refer to \cite{Milman:2023}.

\smallskip
Our explorations in this paper lead to a natural non-Euclidean generalization of $L_p$-affine surface area for $p=-d/(d+2)$ in any $d$ dimensional real space form. We put additional emphasis on the spherical and hyperbolic case. We expect that these new curvature measures in spherical, hyperbolic and de Sitter spaces will lead to new insights, in particular in view of isoperimetric inequalities and applications in stochastic geometry.

\subsection{Statement of principal results}

For $\diamond\in\{s,h\}$, the floating area $\Omega_1^\diamond(K)$ was defined in \cite{BW:2016,BW:2018} as an analog of Blaschke's affine surface area for spherical and hyperbolic convex bodies. Indeed, for $\diamond =s$, if $K$ is a spherical convex body of the $d$-dimensional unit sphere $\S^d$, or, for $\diamond=h$, if $K$ is a hyperbolic convex body in the $d$-dimensional hyperbolic space $\H^d$, then the floating area is defined by
\begin{equation}\label{eqn:floating_area}
    \Omega_1^\diamond(K) := \int_{\bd K} H_{d-1}^\diamond(K,\bu)^{\frac{1}{d+1}}\, \Vol_{\bd K}^\diamond(\dint \bu) \in [0,+\infty),
\end{equation}
where $H_{d-1}^\diamond(K,\bu)$ is the intrinsic generalized Gauss--Kronecker curvature of $\bd K$ and $\Vol_{\bd K}^\diamond$ is the intrinsic surface area measure of the boundary $\bd K$. Here by ``intrinsic'' we refer to the metric structure imposed on $\bd K$ as an embedded hypersurface of the spherical, respectively hyperbolic, space.

For our first result we briefly recall the construction of floating bodies of a convex body. Following the definition in \cite{SW:1990}, for $\delta>0$, we define the floating body $K_\delta$ as the intersection of all closed half-spaces $H^+$, such that the complementary closed half-spaces $H^-$ intersect $K$ in a set of volume no more than $\delta$, that is,
\begin{equation*}
    K_\delta := \bigcap \{H^- : \Vol_d(K\cap H^+) \leq \delta\} \subseteq K.
\end{equation*}
Schütt and Werner \cite{SW:1990} showed that one may define the affine surface area $\as_1(K)$ for general convex bodies as the volume derivative of the floating body $K_\delta$ in the following way
\begin{equation}\label{eqn:limit_affine_surface_area}
    a_d^{\frac{2}{d+1}} \lim_{\delta\to 0^+} \frac{ \Vol_d(K)-\Vol_d(K_\delta)}{\delta^{\frac{2}{d+1}}} = \int_{\bd K} \kappa_o(\bx)^{\frac{1}{d+1}} \, C_K(\dint \bx) =: \as_1(K),
\end{equation}
where the constant $a_d$ is defined by \eqref{eqn:ad}, $\kappa_o(K,\cdot)$ is a centro-affine invariant curvature function and $C_K$ is the centro-affine invariant cone-volume measure on $\bd K$, see Section \ref{sec:back_polar}. In \cite{BW:2016,BW:2018} we motivated \eqref{eqn:floating_area} by establishing a spherical and hyperbolic analog of \eqref{eqn:limit_affine_surface_area}, which yields that
\begin{equation}\label{eqn:limitFloat}
    a_d^{\frac{2}{d+1}} \lim_{\delta\to 0^+} \frac{\Vol_d^\diamond(K \setminus \mathcal{F}_\delta^{\diamond} K)}{\delta^{\frac{2}{d+1}}} = \Omega_1^{\diamond}(K),
\end{equation}
if $K\subset \S^d$ is a spherical convex body, or $K\subset \H^d$ is a hyperbolic convex body and, for $\diamond\in\{s,h\}$. Here $\mathcal{F}_\delta^\diamond$ denotes the corresponding intrinsic notion of floating body, see Section~\ref{sec:spherical} respectively Section~\ref{sec:hyperbolic}. 

In this article we show that \eqref{eqn:limitFloat} also holds true in the de Sitter space: for closed sets $L\subset \dS^d_1$ that have a space-like boundary, are proper convex with respect to space-like geodesic arcs between points in $L$, and future-directed complete--we call these sets \emph{proper future-directed de Sitter convex bodies}. These de Sitter convex bodies are in a one-to-one correspondence with hyperbolic convex bodies via the hyperbolic duality mapping, see Section~\ref{sec:hyperbolic}.

Note that the homogeneous space $\dS^d_1$ is an oriented Lorentz manifold and carries a natural volume measure. Although a future-directed de Sitter convex body $L$ is unbounded and therefore has infinite volume in general, for any past-directed closed half-space $H^-\subset \dS^d_1$ the closed caps $L\cap H^-$ are either bounded or empty and therefore have finite volume. As a consequence the floating body $\mathcal{F}^h_\delta L\subseteq L$ again yields a family of de Sitter convex bodies, that approach $L$ from the inside as $\delta\to 0^+$. Moreover, $L\setminus \mathcal{F}_\delta^h L$ is bounded for all $\delta>0$, and therefore has finite volume. This allows us to establish in Theorem \ref{thm:hyperbolic_float_limit} that \eqref{eqn:limitFloat} holds true also for positiv-time directed, or past-directed, de Sitter convex bodies.

\begin{figure}[t]
    \centering
    \begin{tikzpicture}
        \node at (0,0) {\includegraphics[width=7cm]{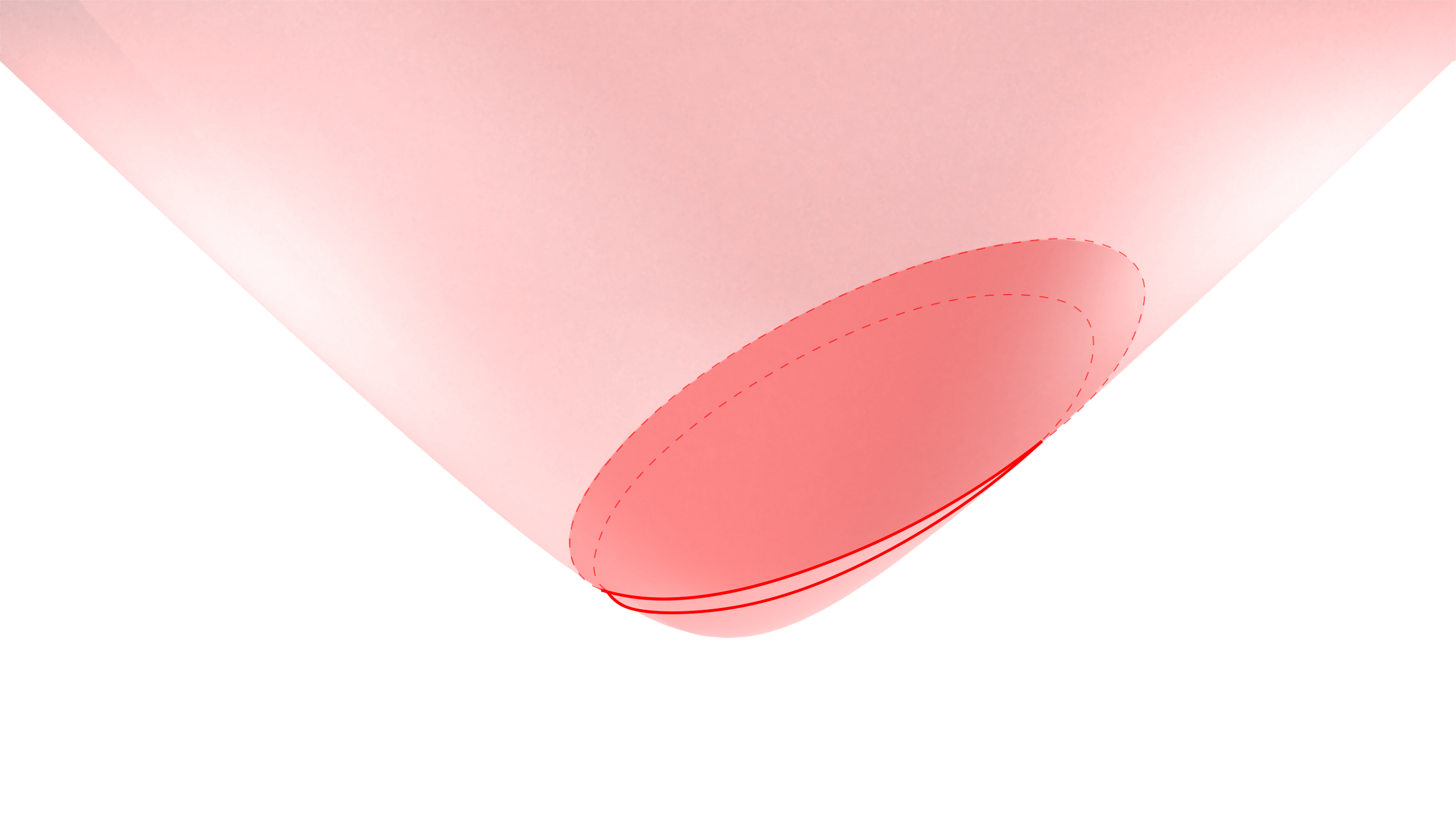}};
        \node at (0.7,1.) {$\mathcal{F}^{h,*}_\delta K$};
        \node at (0.4,-0.3) {$K$};
        \node at (-2, 1) {$\H^2$};
    \end{tikzpicture}
    \begin{tikzpicture}
        \node at (0,0) {\includegraphics[width=7cm]{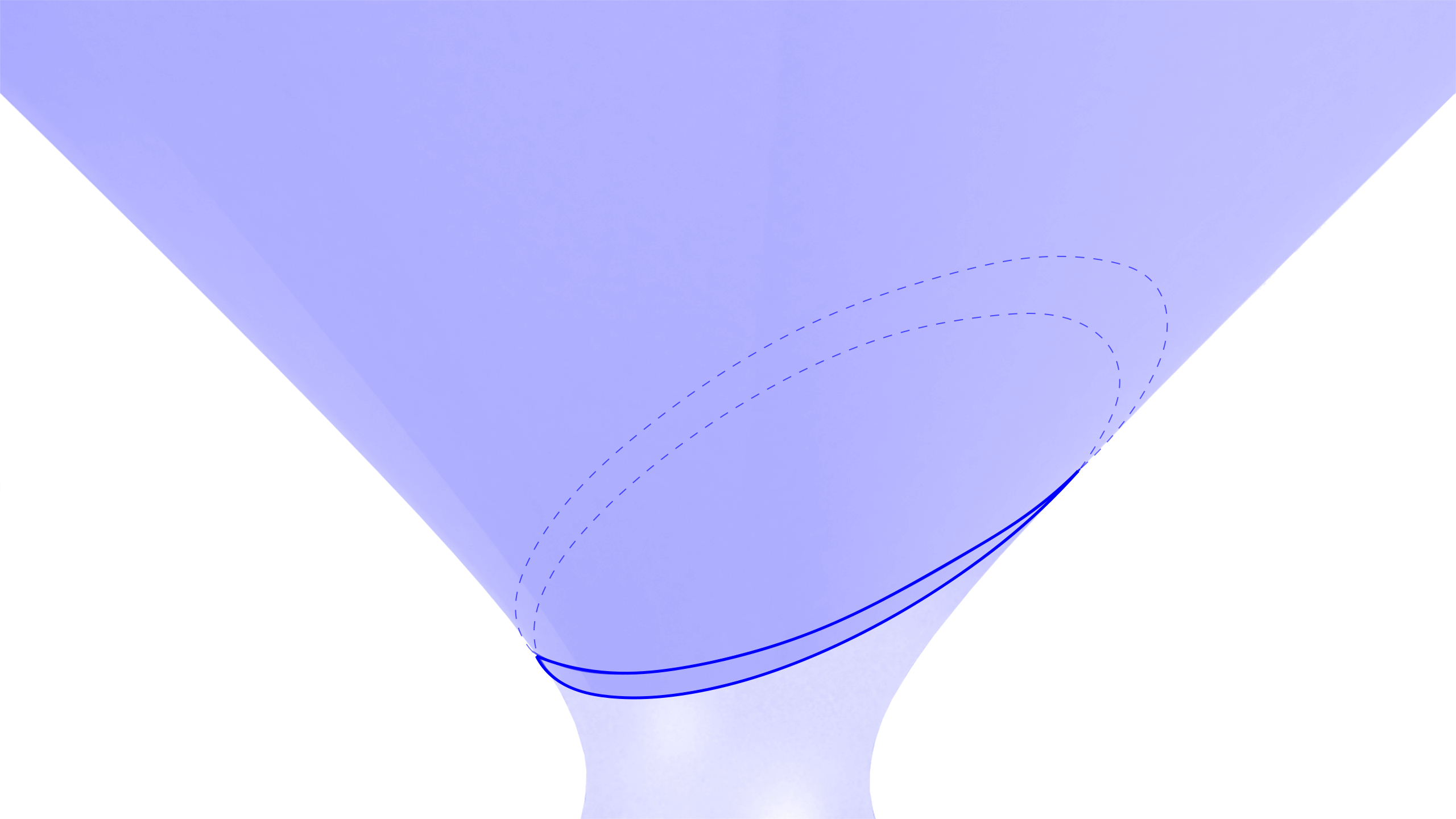}};
        \node at (0.4,-0.75) {$\mathcal{F}^h_\delta K^*$};
        \node at (0.4,-1.45) {$K^*$};
        \node at (-1.2, -1.7) {$\dS_1^2$};
    \end{tikzpicture}
    \caption{Hyperboloid model of $\H^2$ (left) and $\dS^2_1$ (right) in the Lorentz--Minkowski space $\R^{2,1}$. $K\subset \H^2$ is a hyperbolic convex body and $\mathcal{F}_\delta^{*,h} K \supseteq K$ is its dual floating body. $K^*\subset \dS^2_1$ is the hyperbolic dual of $K$ and $\mathcal{F}^{h}_\delta K^* = \left(\mathcal{F}_\delta^{*,h} K\right)^* \subseteq K^*$ is its floating body. $K^*$ and $\mathcal{F}^h_\delta K^*$ are proper future-directed de Sitter convex bodies and $K^*\setminus \mathcal{F}^h_\delta K^*$ is bounded in $\dS^2_1$ and has finite hyperbolic volume.}
    \label{fig:hyperboloid_float}
\end{figure}
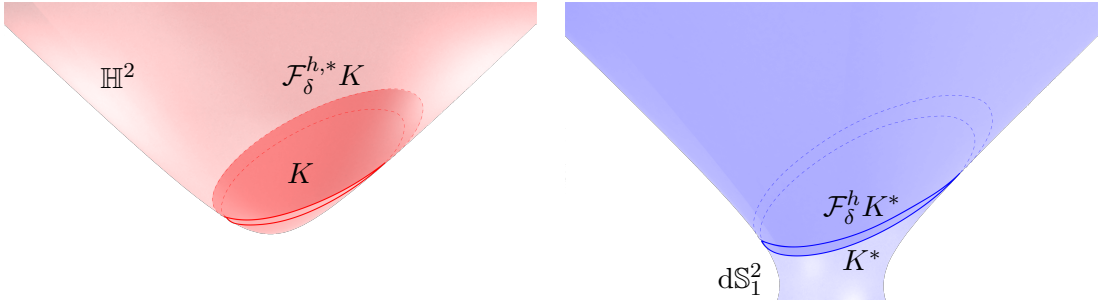

\begin{figure}[t]
    \centering
    \begin{tikzpicture}
        \node at (0,0) {\includegraphics[width=7cm]{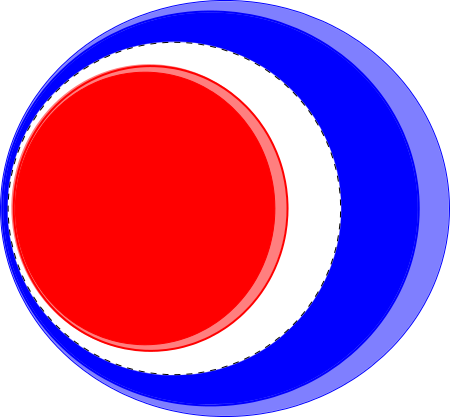}};
        \node at (-1.5,0) {$K$};
        \node at (3.8,1) {$K^*$};
        \node[rotate=60] at (2.3,-1) {$\mathcal{F}^h_{\delta} K^*$};
        \node[rotate=60] at (1.1,-1) {$\mathcal{F}^{h,*}_\delta K$};
        \node at (1.25,1.15) {$\S^1$};
    \end{tikzpicture}
    \caption{Projective model of a hyperbolic floating body conjugated by duality $\mathcal{F}^{h,*}_{\delta} K\supseteq K$ in $\H^2$ and its hyperbolic dual $\mathcal{F}^h_{\delta} K^*\subseteq K^*$ in the associated projective model of de Sitter space $\dS_1^2$.}
    \label{fig:projective_float}
\end{figure}
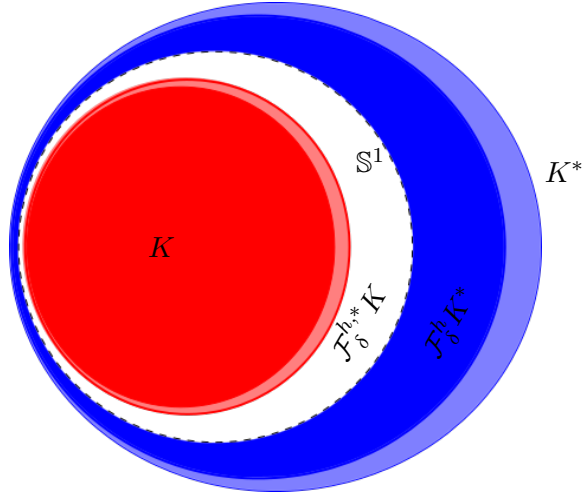

\medskip
If a convex body $K\subset \R^d$ contains the origin in the interior, then the \emph{polar body} $K^\circ$ is defined by 
\begin{equation*}
    K^\circ = \{\by \in\R^d: \bx\cdot \by \leq 1 \text{ for all $\bx\in K$}\}.
\end{equation*}
Meyer and Werner \cite[Thm.\ 8]{MW:2000} showed that volume derivative of the floating body conjugate by polarity gives rise to another centro-affine invariant curvature measure. Indeed, if $K\subset \R^d$ is a convex body that contains the origin in the interior and is of class $\cC^2_+$, that is, the boundary $\bd K$ is a twice differentiable embedded hypersurface of $\R^d$ and the Gauss--Kronecker curvature is strictly positive in every boundary point, then
\begin{equation}\label{eqn:limit_polar_affine_surface_area}
    a_d^{\frac{2}{d+1}} \lim_{\delta\to 0^+} \frac{\Vol_d((K^\circ_\delta)^\circ) - \Vol_d(K)}{\delta^{\frac{2}{d+1}}} = \int_{\bd K} \kappa_o(\bx)^{-\frac{1}{d+1}} \, C_K(\dint \bx) 
    =: \as_{-d/(d+2)}(K).
\end{equation}
Here $\as_p$, for $p>-d$, refers to a centro-affine invariant family of $p$-affine surface areas, see Section~\ref{sec:p_affine_surface}. 

In this article we establish a weighted extension of \eqref{eqn:limit_polar_affine_surface_area}, see Theorem \ref{thm:main_weighted}, and apply it to derive non-Euclidean analogs of \eqref{eqn:limit_polar_affine_surface_area} by conjugating the floating body $\mathcal{F}^\diamond_\delta K$ with the natural duality in the spherical and hyperbolic setting. See Section~\ref{sec:hyperbolic} and Figures~\ref{fig:hyperboloid_float} and \ref{fig:projective_float} for illustrations of the hyperbolic setting in the hyperboloid and projective model.

\begin{theorem}\label{thm:limit_polar_affine_surface_area}
    Let $K$ be a spherical, hyperbolic, or proper future-directed de Sitter convex body of class $\cC^2_+$. Then
    \begin{equation*}
        a_d^{\frac{2}{d+1}} \lim_{\delta\to 0^+} \frac{\Vol_d^\diamond(\mathcal{F}_\delta^{\diamond, *} K \setminus K)}{\delta^{\frac{2}{d+1}}}
        = \int_{\bd K} H_{d-1}^{\diamond}(K,\bu)^{-\frac{1}{d+1}} \, \Vol_{\bd K}^\diamond(\dint \bu) =: \Omega_{-d/(d+2)}^\diamond(K),
    \end{equation*}
    where, for $\diamond\in\{s,h\}$, $\mathcal{F}_\delta^{\diamond, *} K$ is the intrinsic floating body conjugate by duality, that is, $\mathcal{F}_\delta^{\diamond,*} := (\mathcal{F}_\delta^{\diamond} K^*)^*$, and the constant $a_d$ is defined by \eqref{eqn:ad}.
\end{theorem}

Basic properties of $\as_{-d/(d+2)}$, obtained in \cite{Hug1:1996,Ludwig:2001,Ludwig:2010, Schutt:1993}, are also established for these new non-Euclidean curvature measures $\Omega_{-d/(d+2)}^\diamond$ in the following.
\begin{theorem}\label{thm:dual_formula_non-euclidean}
    Let $\diamond \in \{s,h\}$. Then
    \begin{enumerate}
     \item[i)] $\Omega_{-d/(d+2)}^\diamond$ is an intrinsic valuation on spherical, respectively hyperbolic and de Sitter convex bodies, that is, if $K,L$ are convex bodies such that $K\cup L$ is also a convex body then
     \begin{equation*}
        \Omega_{-d/(d+2)}^\diamond(K)+\Omega_{-d/(d+2)}^\diamond(L) = \Omega_{-d/(d+2)}^\diamond(K\cup L) + \Omega_{-d/(d+2)}^\diamond(K\cap L),
     \end{equation*}
     and if $g$ is an isometry of the space, then $\Omega_{-d/(d+2)}^\diamond(g(K)) = \Omega_{-d/(d+2)}^\diamond(K)$.
     
     \item[ii)] $\Omega_{-d/(d+2)}^\diamond$ is lower semi-continuous with respect to the Hausdorff distance.
     
     \item[iii)] If $K$ is of class $\cC^2_+$, then
     \begin{equation*}
        \Omega_{-d/(d+2)}^\diamond(K) = \int_{\bd K^*} H_{d-1}^\diamond(K^*,\bu)^{\frac{d+2}{d+1}} \, \Vol_{\bd K^*}^\diamond(\dint \bu).
     \end{equation*}
    \end{enumerate}
\end{theorem}

\subsection{Real-analytic extensions and limits obtained by scaling the curvature}

We obtain an extension of Theorem \ref{thm:limit_polar_affine_surface_area} for $\lambda \in \R$ in all real-space forms $\Sp^d(\lambda)$, and de Sitter space forms $\Sp^d_1(\lambda)$ for $\lambda <0$, of constant curvature $\lambda$, namely Theorem \ref{thm:main_lambda}, which gives rise to the curvature measure
\begin{equation*}
    O_{-d/(d+2)}^\lambda(K) = \int_{\bd K} H_{d-1}^\lambda(K,\bu)^{-\frac{1}{d+1}} \, \Vol_{\bd K}^\lambda(\dint \bu),
\end{equation*}
where $H_{d-1}^\lambda(K,u)$ is the generalized intrinsic Gauss--Kronecker curvature of $\bd K$ at $u$ and $\Vol_{\bd K}^\lambda$ is the intrinsic $(d-1)$-dimensional Hausdorff measure restricted to $\bd K$.

In Section \ref{sec:real_analytic} we fix $\overline{K}\subset \R^d$ in the projective model of $\Sp^d(\lambda)$ for all $\lambda\in\R$ and rescale $\delta$ by $\lambda$. Then we may take the limit $\lambda\to 0$ and observe that the polar floating body $\mathcal{F}_{\delta_\lambda}^{\lambda,*} \overline{K}$ converges to the $V_1$-illumination body \cite{Werner:1994}, or $V_1$-separation body \cite{Schneider:2020}, of $\overline{K}$. Here the $V_1$-illumination body is defined by
\begin{equation*}
    \mathcal{I}_{\delta}^{V_1}(\overline{K}) = \{\bx \in \R^d : \Delta_1([\overline{K},\bx],\overline{K}) \leq \delta\},
\end{equation*}
where $[L,\bx] = \mathrm{conv}(L,\{\bx\})$ and $\Delta_1(L,L') = V_1(L)+V_1(L')-2V_1(L\cap L')$ is the intrinsic volume deviation between convex bodies $L$ and $L'$ in $\R^d$ with respect to the intrinsic volume $V_1$, see e.g.\ \cite{BH:2023, BHK:2021}.
Note that the intrinsic volume $V_1$ is up to a dimensional constant the same as the mean width.

\renewcommand*{\thefootnote}{\arabic{footnote}}

Thus, as an immediate corollary to our extension of Theorem \ref{thm:limit_polar_affine_surface_area} for all $\lambda\in\R$, that is Theorem \ref{thm:main_lambda}, we derive, for $\lambda\to 0$, the following\footnote{Independent of our investigations, equation \eqref{eqn:olaf} also appeared previously without proof in a preprint version on arXiv of \cite{Schneider:2020} and was communicated to R.\ Schneider by O.\ Mordhorst in a private correspondence.}
\begin{theorem}\label{thm:V1_illumination_body}
    Let $K\subset\R^d$ be a convex body of class $\cC^2_+$. Then
    \begin{equation}\label{eqn:olaf}
        \frac{1}{c_d} \lim_{\delta\to 0^+} \frac{\Vol_d(\mathcal{I}_\delta^{V_1} K) - \Vol_d(K)}{\delta^{\frac{2}{d+1}}} 
        = \int_{\bd K} H_{d-1}(K,\bx)^{-\frac{1}{d+1}}\, \Vol_{\bd K}(\dint \bx) =: O_{-d/(d+2)}(K),
    \end{equation}
    where $c_d = \frac{1}{2}(d+1)^{2/(d+1)}$. Here $H_{d-1}(K,\cdot)$ is the generalized Gauss--Kronecker curvature and $\Vol_{\bd K}$ is the $(d-1)$-dimensional Hausdorff measure on $\bd K$.
\end{theorem}
Another proof of Theorem \ref{thm:V1_illumination_body} is presented in Section \ref{sec:weighted_polar_volume} using the polar volume of the weighted floating body in $\R^d$.

The duality $^*$ is a natural mapping between convex bodies in $\Sp^d(\lambda)$ and $\Sp^d(1/\lambda)$ for $\lambda \geq 0$, respectively $\Sp^d_1(1/\lambda)$ for $\lambda <0$. However, under suitable conditions we may rescale $K^*\subset \Sp^d(1/\lambda)$ with respect to a fixed point $\be\in\Sp^d(\lambda)$ to obtain $K^{\be}\subset \Sp^d(\lambda)$. 
We investigate this $\be$-polarity from convex bodies in $\Sp^d(\lambda)$ to $\Sp^d(\lambda)$ relative to a fixed point $\be$ in Section \ref{sec:star_polarity}. 
In Theorem \ref{thm:main_analytic} we establish, that, for a fixed point $\be$ in the interior of a convex body $K\subset \Sp^d(\lambda)$, the volume derivative of the floating body conjugated by $\be$-polarity gives rise to a curvature measure $\Omega^{\lambda,\be}_{-d/(d+2)}(K)$. If we identify $\Sp^d(\lambda)$ with a Euclidean model such that $\be$ is the origin $\bo$ and $\overline{K}\subset\R^d$ is the Euclidean convex body associated with $K$, then 
\begin{equation*}
    \lim_{\lambda\to 0^+} \Omega^{\lambda,\be}_{-d/(d+2)}(\overline{K}) = \as_{-d/(d+2)}(\overline{K}).
\end{equation*}
Thus $\Omega^{\lambda,\be}_{-d/(d+2)}$ can be seen as a real analytic extension of $\as_{-d/(d+2)}$. We further note that in the spherical setting $\lambda= 1$, respectively in the hyperbolic setting $\lambda=-1$, the $\be$-polarity is the same as the spherical, respectively hyperbolic, duality mapping $^*$ and therefore $\Omega^{\lambda,\be}_{-d/(d+2)}$ agrees with $\Omega^{\diamond}_{-d/(d+2)}$ for $\diamond \in \{s,h\}$ and does not depend on $\be$.

\medskip
Notably we observe, that while the non-Euclidean floating body and the floating area $\Omega_1^{\diamond}$ give rise to a real-analytic extension of the centro-affine and rigid-motion invariant curvature measure $\as_1$, our investigations for the polar floating body and $\Omega_{-d/(d+2)}^\diamond$ give rise to two distinct space limits from the non-Euclidean setting back to the Euclidean space. In the first case we connect $\Omega_{-d/(d+2)}^{\diamond}$ using the duality $^*$ to the rigid-motion invariant curvature measure $O_{-d/(d+2)}$ and in the second case, by fixing a point $\be$, we connect $\Omega_{-d/(d+2)}^\diamond$ using $\be$-polarity to the centro-affine invariant curvature measure $\as_{-d/(d+2)}$. So remarkably, both $\Omega_1^{\diamond}$ as well as $\Omega_{-d/(d+2)}^{\diamond}$ connect to centro-affine and rigid-motion invariant curvature measures in Euclidean space, see Figure \ref{graph:rels}.

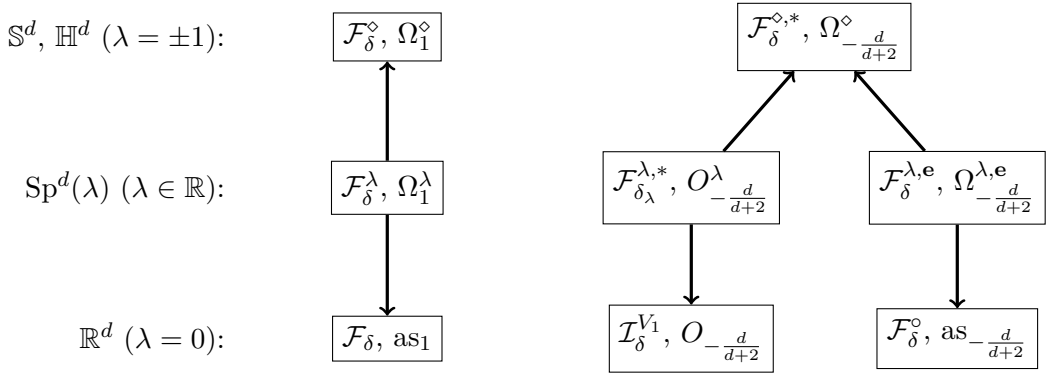
\begin{figure}[t]
    \centering
    \begin{tikzpicture}

            \node[left] at (-2,0) {$\S^d$, $\H^d$ ($\lambda=\pm 1$):};
            \node[left] at (-2,-2) {$\Sp^d(\lambda)$ ($\lambda \in \R$):};
            \node[left] at (-2,-4) {$\R^d$ ($\lambda=0$):};
        
            \node[shape=rectangle, draw=black] (A1) at (0,0) {$\mathcal{F}^{\diamond}_\delta$, $\Omega_1^\diamond$};

            \node[shape=rectangle, draw=black] (A2) at (5.75,0) {$\mathcal{F}^{\diamond,*}_\delta$, $\Omega_{-\frac{d}{d+2}}^\diamond$};
            
            \node[shape=rectangle, draw=black] (B2) at (4,-2) {$\mathcal{F}^{\lambda,*}_{\delta_\lambda}$, $O_{-\frac{d}{d+2}}^\lambda$};
            \node[shape=rectangle, draw=black] (B3) at (7.5,-2) {$\mathcal{F}^{\lambda,\be}_\delta$, $\Omega_{-\frac{d}{d+2}}^{\lambda,\be}$};
            
            \node[shape=rectangle, draw=black] (C2) at (4,-4) {$\mathcal{I}^{V_1}_\delta$, $O_{-\frac{d}{d+2}}$};
            \node[shape=rectangle, draw=black] (C3) at (7.5,-4) {$\mathcal{F}^{\circ}_\delta$, $\as_{-\frac{d}{d+2}}$};

            \draw[<-, very thick] (A2) edge (B2);
            \draw[<-, very thick] (A2) edge (B3);
            
            \draw[->, very thick] (B2) edge (C2);
            \draw[->, very thick] (B3) edge (C3);

            \node[shape=rectangle, draw=black] (B1) at (0,-2) {$\mathcal{F}^{\lambda}_\delta$, $\Omega_1^\lambda$};
            
            \node[shape=rectangle, draw=black] (C1) at (0,-4) {$\mathcal{F}_\delta$, $\as_1$};
            
            \draw[<-, very thick] (A1) edge (B1);
            
            \draw[->, very thick] (B1) edge (C1);
        \end{tikzpicture}
        \caption{Relation between the non-Euclidean (polar) floating body and its associated curvature measure and the Euclidean curvature measures that can be derived by scaling the curvature $\lambda$ to $0$. Remarkably, $\as_1$ is centro-affine and rigid-motion invariant and directly relates to its non-Euclidean analog $\Omega_1^{\diamond}$, while $\Omega_{-d/(d+2)}^\diamond$ can be related to two different Euclidean curvature measures, the rigid-motion invariant curvature measure $O_{-d/(d+2)}$ and the centro-affine invariant curvature measure $\as_{-d/(d+2)}$.}
        \label{graph:rels}
\end{figure}

\section{Background on convex geometry}\label{sec:background}

A convex body $K\subset \R^d$ is a compact convex subset with non-empty interior, such as, for example, the Euclidean ball $B_2^d(r)=\{\bx\in\R^d: \|\bx\|_2 \leq r\}$ of radius $r>0$ centered at the origin.
Here $\|\bx\|_2=\sqrt{\bx \cdot \bx}$ is the Euclidean norm induced by the Euclidean inner product $\cdot$ and the boundary sphere of $B_2^d:=B_2^d(1)$ is the Euclidean unit sphere $\S^{d-1}: = \{\bx\in \R^d : \|\bx\|_2 = 1\}$. For a general reference on convex geometry we refer to the books by Gardner \cite{Gardner:2006} and by Schneider \cite{Schneider:2014}.

A hyperplane $H(\bu,t):=\{\bx\in\R^d: \bx\cdot \bu = t\}$ is uniquely determined by the direction $\bu\in\S^{d-1}$ and the signed distance $t\in \R$ from the origin.
It is the boundary of the closed half-spaces  $H^+(\bu,t) := \{\bx\in\R^d : \bu\cdot \bx \geq t\}$ and $H^-(\bu,t) := H^+(-\bu,-t)$.

The support function $h_K:\R^d\to \R$ is defined by $h_K(\by) := \max \{ \by\cdot \bx: \bx\in K\}$ and,
if $K$ contains the origin in the interior,
then the radial function $\rho_K:\S^{d-1}\to(0,+\infty)$ is defined by
$\rho_K(\bu) := \max\{t>0 : t\bu \in K\}$ for all $\bu \in \S^{d-1}$. Note that $\bx\in \bd K$ has a uniquely determined outer unit normal $\bu = \bn_K(\bx)\in\S^{d-1}$ if and only if $h_K$ is differentiable in $\bu$ and $\bx=\nabla h_K(\bu)$ where $\nabla h_K$ is the gradient of $h_K$, see \cite[Cor.\ 1.7.3]{Schneider:2014}.

In the next lemma we state well known bounds for the volume (Lebesgue measure) $\Vol_d$ of ball caps, see, e.g., \cite[Hilfssatz 1]{LW1:1986}.
\begin{lemma}(volume of ball caps)\label{lem:ball_cap}
	Let $C^d(r, h):=\{(\by,z)\in\R^{d-1}\times\R : \|\by\|_2^2+z^2\leq r^2, r-h\leq z \leq r\}\subset B_2^d(r)$ be a cap of height $h$ of the Euclidean ball with radius $r$ in $\R^{d}$.
	Then
	\begin{equation*}
		\Vol_d(C^d(r,h)) =  a_d r^{\frac{d-1}{2}}  h^{\frac{d+1}{2}} (1+O(h)), \quad \text{for $h\to 0^+$},
	\end{equation*}
	where $a_d$ is the volume of the radially symmetric parabolic cap $P_1^d(1,1)=\{(\mathbf{y},z)\in\R^{d-1}\times\R : 1\geq z \geq \frac{1}{2}\|\mathbf{y}\|_2^2\}$ of height $1$ at the apex, that is,
	\begin{equation}\label{eqn:ad}
		a_d = \Vol_d(P_1^d(1,1))
		= (2\pi)^{\frac{d-1}{2}} \Big/ \Gamma\left(\frac{d+1}{2}+1\right).
	\end{equation}
    To be more precise, for $0< h< \frac{2r}{d}$, we have
    \begin{align*}
		\frac{a_d}{2} r^{\frac{d-1}{2}} h^{\frac{d+1}{2}} < \Vol_d(C^d(r,h)) < a_d r^{\frac{d-1}{2}} h^{\frac{d+1}{2}}.
	\end{align*}
\end{lemma}

\subsection{Convex bodies that admit a rolling ball}

We say that $K$ is of class $\cC^k$ if $\bd K$ is a $\cC^k$-smooth submanifold of $\R^d$. We call $K$ of class $\cC^k_+$, for $k\geq 2$, if the \emph{generalized Gauss--Kronecker curvature} $H_{d-1}(K,\bx)$ is strictly positive for all $\bx\in \bd K$. Note that if $K$ is of class $\cC^k$, then $K$ is of class $\cC^k_+$ if and only if the Gauss map $\bn_K:\bd K \to \S^{d-1}$ is a $\cC^{k-1}$-diffeomorphism.
We are interested in $\cC^{1,1}$-smooth convex bodies, that is, $K$ is of class $\cC^1$ and the Gauss map $\bn_K:\bd K \to \S^{d-1}$ is Lipschitz. We further denote by $\cC^{1,1}_+$ the class of convex bodies such that $\bn_K$ is bi-Lipschitz between $\bd K$ and $\S^{d-1}$.

Finally, we denote by $\cC^k_{\mathrm{sc}}\subset \cC^k$ the class of all strictly convex bodies in $\cC^k$. A convex body is of class $\cC^1_{\mathrm{sc}}$ if and only if the Gauss map $\bn_K$ is a homeomorphism. Note that $\cC^k_+\subset \cC^k_{\mathrm{sc}}$ and $\cC^{1,1}_+\subset \cC^{1,1}_{\mathrm{sc}}$ and both inclusions are strict. We also note if $K$ is of class $\cC^{1,1}_+$ and contains the origin in the interior, then $K^\circ$ is also of class $\cC^{1,1}_+$, that is, $\cC^{1,1}_+$ is closed with respect to polarity. This is not true for $\cC^{1,1}_{\mathrm{sc}}$ in general. For example, for $p\geq 1$, let $\|(x_1,\dotsc,x_d)\|_p=\left(\sum_{k=1}^d |x_k|^p\right)^{1/p}$ and $B_p^d=\{\bx\in\R^d: \|\bx\|_p \leq 1\}$. Then, for $p>2$, $B_p^d$ is of class $\cC^2_{\mathrm{sc}}\setminus \cC^{1,1}_+ \subset \cC^{1,1}_{\mathrm{sc}}$, but $(B_p^d)^\circ = B_q^d$, for $1/p+1/q=1$, is of class $\cC^1_{\mathrm{sc}}\setminus \cC^{1,1}_\mathrm{sc}$.

\medskip
The \emph{rolling function} $r_K:\bd K \to [0,\infty)$ was introduced by McMullen in \cite{MM:1974}, see also \cite{SW:1990}, and is defined by 
\begin{equation*}
	r_K(\bx) := \sup \{ r\geq 0 : \exists \by\in K \text{ such that } \bx\in B_2^d(r)+\by \subset K\},
\end{equation*}
i.e., $r_K(\bx)$ is the maximal radius of a Euclidean ball inside $K$ that contains $\bx$. Note that the rolling function may not be continuous in general, but is always upper semicontinuous, that is, the super level sets $[r_K\geq t] := \{\bx\in \bd K: r_K(\bx)\geq t\}$ are closed for all $t\geq 0$, see \cite[Lem.~2.1]{Hug1:1996}.

We say that $K$ \emph{admits a rolling ball} if $\roll(K) := \inf\{r_K(\bx) >0: \bx\in \bd K\} > 0$.
A convex body $K$ admits a rolling ball of radius $\roll(K) > 0$ if and only if $K$ is of class $\cC^{1,1}$ \cite[Hilfssatz~1]{LW2:1986}, see also \cite[Lem.~2.1]{Hug1:1996}. Furthermore, if $\roll(K)>0$, then $K$ is the outer parallel body of a convex body $L$, i.e., $K = L+\roll(K) B_2^d$, and conversely for any convex body $K$ the outer parallel body $L=K+rB_2^d$ admits a rolling ball with $\roll(L)\geq r$.

Note that there are $\cC^1$-smooth convex bodies that do not admit a rolling ball, see Example \ref{ex:Lp-ball}.
By Blaschke's rolling theorem any $\cC^2$-smooth convex body admits a rolling ball. However, not all bodies that admit a rolling ball are of class $\cC^2$, see \cite[p.\ 438]{LW1:1986}. For example, the outer parallel bodies of polytopes admit a rolling ball, but are not of class $\cC^2$. See Figure \ref{fig:rel_spaces} for a graph of the inclusion relation between the different classes of convex bodies.

\begin{figure}
\centering
\begin{tikzpicture}

    \node[rectangle] (C2)   at (-0.5,2) {$\cC^k$};
    \node[rectangle] (C2p)  at (-0.5,-2){$\cC^k_+$};
    \node[rectangle] (C2sc) at (-0.5,0) {$\cC^k_{\mathrm{sc}}$};

    \node[rectangle] (pC2dots)   at (-1.75,2) {$\dots$};
    \node[rectangle] (pC2pdots)  at (-1.75,-2){$\dots$};
    \node[rectangle] (pC2scdots) at (-1.75,0) {$\dots$};

    \node[rectangle] (C2dots)   at (0.75,2) {$\dots$};
    \node[rectangle] (C2pdots)  at (0.75,-2){$\dots$};
    \node[rectangle] (C2scdots) at (0.75,0) {$\dots$};

    \node[rectangle] (R)    at (2,2) {$\cC^{1,1}$};
    \node[rectangle] (Rp)   at (2,-2) {$\cC^{1,1}_{+}$};
    \node[rectangle] (Rsc)   at (2,0) {$\cC^{1,1}_{\mathrm{sc}}$};
    
    \node[rectangle] (C1) at (4,2) {$\cC^1$};
    \node[rectangle] (C1sc) at (4,0) {$\cC^1_{\mathrm{sc}}$};
    
    \path[->] (pC2dots) edge (C2);
    \path[->] (C2) edge (C2dots);
    \path[->] (C2dots) edge (R);
    \path[->] (R) edge (C1);
    
    \path[->] (C2p) edge (C2sc);
	
	\path[->] (pC2pdots) edge (C2p);
	\path[->] (C2p) edge (C2pdots);
	\path[->] (C2pdots) edge (Rp);
	
	\path[->] (Rp) edge (C1sc);
	\path[->] (Rp) edge (Rsc);
	
	\path[->] (C2sc) edge (C2);
	
	\path[->] (pC2scdots) edge (C2sc);
	\path[->] (C2sc) edge (C2scdots);
	\path[->] (C2scdots) edge (Rsc);
	
	\path[->] (Rsc) edge (C1sc);
	\path[->] (Rsc) edge (R);
	\path[->] (C1sc) edge (C1);
	
\end{tikzpicture}
\caption{Partial order induced by the inclusion between the different classes of convex bodies for $k\geq 2$. Note that among all the classes considered only $\cC^k_+$, $\cC^{1,1}_+$ and $\cC^1_{\mathrm{sc}}$ are closed with respect to polarity on convex bodies that contain the origin in the interior. Furthermore, a convex body is of class $\cC^{1,1}$ if and only if it admits a rolling ball from the inside.}
\label{fig:rel_spaces}
\end{figure}
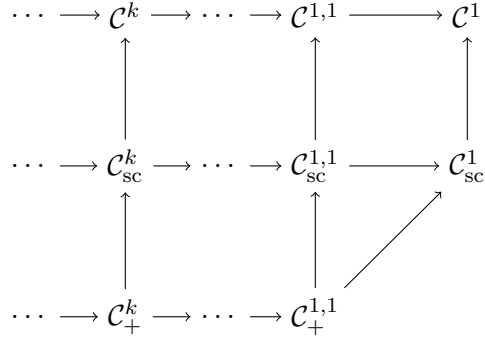

\medskip
We recall that for a convex body $K\subset \R^d$ a boundary point $\bx\in \bd K$ is called \emph{normal} if $\bx$ has a unique outer unit normal and the local representation function $f_{(\bx,K)}$ defined on an open neighborhood $U$ around $o$ of the tangent space $T_{\bx} K\cong \R^{d-1}$, $f_{(\bx,K)}: U\subset T_{\bx} K \to \R$ of $\bd K$ near $\bx=(o,f_{(\bx,K)}(o))$ is second order differentiable at the origin $o\in\R^{d-1}$, see, e.g., \cite[Sec.~2.6]{Schneider:2014}.
While for any full dimensional convex body $K$ almost all boundary points are normal, the corresponding set of unit normals to the normal boundary points might be discrete and may not give complete information about the shape of $K$. For example, for a convex polytope $P$ the \emph{support set} $F(P,\bu)=\{\bx\in P : \bu\cdot \bx = h_P(\bu)      \}$ is a singular boundary point for almost all directions $\bu\in\S^{d-1}$ and only for the facets of $P$ the relative interior of $F(P,\bu)$ are normal boundary points. However, if $K$ admits a rolling ball, then almost all normal directions $\bu\in \S^{d-1}$ determine a unique normal boundary point $\bx\in \bd K$. This motivates the next approximation result that follows by results of Leichtweiss \cite{LW1:1986} and Hug \cite{Hug1:1996}.

\pagebreak

\begin{theorem}[local approximation of convex bodies of class $\cC^{1,1}$]\label{thm:approx_parabola}
    Let $K\subset \R^d$ be a convex body of class $\cC^{1,1}$. Then for almost all $\bu\in\mathbb{S}^{d-1}$ we have:
    \begin{enumerate}
        \item[i)] The support function $h_K$ is second order differentiable at $\bu$ and $\bx:=\nabla h_K(\bu)\in\bd K$ is a normal boundary point with $H_{d-1}(K,\bx) >0$.
        
        \item[ii)] In addition, since $\bx$ is normal and $H_{d-1}(K,\bx)>0$, we find that $\bd K$ can be approximated near $\bx$ in following sense: For all small enough $\varepsilon\in (0,1)$, there exists $\delta=\delta(\varepsilon)>0$ such that the local representation of $\bd K$ at $\bx$ is given by a convex function $f_{(\bx,K)}: B_2^{d-1}(\delta) \subset T_\bx K \to [0,+\infty)$, where we identify the tangent hyperplane $T_\bx K\cong \R^{d-1}$ and assume that $\bx$ is at the origin, i.e., $(\by,f_{(\bx,K)}(\by))\in \bd K$ for all $\by\in B_2^{d-1}(\delta)$. Moreover, we have that
        \begin{equation}\label{eqn:taylor_approx}
            \frac{1}{(1+\varepsilon)} Q^2_{(\bx,K)}(\by) \leq f_{(\bx,K)}(\by) \leq \frac{1}{(1-\varepsilon)} Q^2_{(\bx,K)}(\by) \quad \text{for all $\by\in B_2^{d-1}(\delta)$},
        \end{equation}
        where the positive definite quadratic form $Q^2_{(\bx,K)}$ is defined by the Hessian of $f_{(\bx,K)}$ at $o\in\R^{d-1}$ as
        \begin{equation*}
            Q^2_{(\bx,K)}(\by) := \frac{1}{2} d^2 f_{(\bx,K)}(o)(\by,\by) = \frac{1}{2} \sum_{i=1}^{d-1} \frac{\partial^2 f}{\partial x_i \partial x_j}(o) y_i y_j,
        \end{equation*}
        where $\by=(y_1,\dotsc,y_{d-1})$.
     
        \item[iii)] Moreover, we may approximate $\bd K$ at $\bx$ by paraboloids, that is, there is $h=h(\varepsilon)>0$ such that
        \begin{equation*}
            P^d_{(\bx,K)}(1-\varepsilon,h) \subseteq K\cap H^+(\bu,h_K(\bu)-h) \subseteq P^d_{(\bx,K)}(1+\varepsilon,h),
        \end{equation*}
        where
        \begin{equation*}
            P^d_{(\bx,K)}(r,h) = \left\{(\by,z)\in T_\bx K \times \R : h\geq z \geq \frac{1}{r} Q^2_{(\bx,K)}(\by)\right\}
        \end{equation*}
        is an elliptic paraboloid that touches $\bd K$ at $\bx$ and 
        \begin{equation}\label{eqn:paraboloid_gauss}
            H_{d-1}(P^d_{(\bx,K)}(r,h),\bx) = r^{1-d} H_{d-1}(K,\bx).
        \end{equation}
    \end{enumerate}
    Finally, for an integrable function $g:\S^{d-1}\to \R$, we have that
	\begin{equation}\label{eqn:int_formula}
        \int_{\S^{d-1}} g(\bu)\, \mathcal{H}^{d-1}(\dint{\bu}) = \int_{\bd K} g(\bn_K(\bx))\, H_{d-1}(K,\bx)\, \mathcal{H}^{d-1}(\dint{\bx}),
	\end{equation}
	where $\mathcal{H}^{d-1}$ is the $(d-1)$-dimensional Hausdorff measure in $\R^d$ and $\bn_K$ denotes the Gauss map, i.e., $\bn_K(\bx)\in\S^{d-1}$ is the outer unit normal of $\bd K$ at $\bx$.
\end{theorem}
\begin{proof}
    Since $K$ admits a rolling ball, statement i) follows from \cite[Lem.\ 2.6]{Hug1:1996}. 
    
    For ii) we just recall that a boundary point is called normal if and only if the local representation function $f_{(\bx,K)}$ is second order differentiable at $o\in\R^{d-1}$, that is
    \begin{equation*}
        \left|f_{(\bx,K)}(\by)-\frac{1}{2} d^2f_{(\bx,K)}(o)(\by,\by)\right| = o(\|\by\|^2_2),
    \end{equation*}
    which yields \eqref{eqn:taylor_approx}. Compare also \cite[Lem.\ 6]{Reitzner:2002} and \cite[Sec.\ 1.6]{SW:2003}.
    
    Next, iii) is just a reformulation of \eqref{eqn:taylor_approx} and \eqref{eqn:paraboloid_gauss} follows by
    \begin{equation*}
        H_{d-1}(P^d_{(\bx,K)}(r,h),\bx) = \det\left(\frac{1}{r}\frac{\partial^2f_{(\bx,K)}}{\partial x_i\partial x_j}(o)\right) = r^{1-d} H_{d-1}(K,\bx).
    \end{equation*}
    
    Finally, \eqref{eqn:int_formula} follows by Federer's area formula since the Gauss map $\bn_K:\bd K\to \S^{d-1}$ is Lipschitzian, see \cite[Lem.\ 2.1]{Hug1:1996} respectively \cite[Hilfsatz 1]{LW1:1986}, and for almost all $\bx\in \bd K$ the approximate Jacobian is
    \begin{equation*}
        J_{d-1}^{\bd K} (\bn_K)(\bx) = H_{d-1}(K,\bx),
    \end{equation*}
    see \cite[Lem.\ 2.3]{Hug1:1996}.
\end{proof}

\begin{remark}
For a general convex body a version of \eqref{eqn:int_formula} holds if the measure on the right hand side is replaced by the Gaussian curvature measure $C_0(K,\cdot) = \mathcal{H}(\sigma_K(\cdot))$, i.e., for a Borel set $A\subset \R^d$ we have
\begin{align*}
    \int_{\S^{d-1}} \mathbf{1}_{\sigma_K(A)}(\bu)\, \mathcal{H}^{d-1}(\dint \bu) 
    &= \int_{\bd K} \mathbf{1}_A(\bx)\, C_0(K,\dint \bx) \\
    &= \int_{\bd K} \mathbf{1}_A(\bx) H_{d-1}(K,\bx) \, \mathcal{H}^{d-1}(\dint \bx) + C_0^s(K,A),
\end{align*}
where $\sigma_K(A) = \{\bu\in \S^{d-1} : \text{$\bu$ is an outer normal vector to some $\bx\in A$}\}$ and $C_0^s(K,\cdot)$ is the singular part of $C_0(K,\cdot)$ with respect to $\mathcal{H}^{d-1}$.

Hence, if $K$ is $\cC^1$ and $C_0(K,\cdot)$ is absolutely continuous then \eqref{eqn:int_formula} follows.
However, \eqref{eqn:int_formula} may fail even if $K\in\cC^1_{\mathrm{sc}}$. For example, for $R\in(0,1)$ a classic construction of Busemann and Feller \cite[Sec.\ 6]{BF:1936} yields a convex body $K$ of class $\cC^1_{\mathrm{sc}}$ that is contained in the unit ball with constant Gauss--Kronecker curvature $H_{d-1}(K,\cdot)=R$ for almost all boundary points. Thus, for $g\equiv 1$, the right hand side of \eqref{eqn:int_formula} yields
\begin{equation*}
    \int_{\bd K} H_{d-1}(K,\bx) \,\mathcal{H}^{d-1}(\bx) = R \mathcal{H}^{d-1}(\bd K) \leq R \mathcal{H}^{d-1}(\S^{d-1}) < \mathcal{H}^{d-1}(\S^{d-1}).
\end{equation*}
Clearly, in this case the Gaussian curvature measure $C_0(K,\cdot)$ has a singular part. Note that Hug showed in \cite[Thm.\ 2.3]{Hug:1999} and \cite[Lem.\ 2.7]{Hug1:1996} that $C_0(K,\cdot)$ is absolutely continuous if and only if $r_K(\nabla h_K(\bu))>0$ for almost all $\bu\in\S^{d-1}$. This also implies that \eqref{eqn:int_formula} holds if $K$ is of class $\cC^{1,1}$.
\end{remark}

\subsection{The affine surface area and its relatives} \label{sec:p_affine_surface}

For a compact convex $K\subset \R^d$ \emph{Blaschke's (equi-)affine curvature measure} is defined by
\begin{equation*}
    \Omega(K,\eta) = \int_{\eta\,\cap\, \bd K} H_{d-1}(K,\bx)^{\frac{1}{d+1}} \, \mathcal{H}^{d-1}(\dint\bx),
\end{equation*}
for any Borel subset $\eta\subset \R^d$, where $H_{d-1}(K,\cdot)$ denotes the (generalized) Gauss--Kronecker curvature and $\mathcal{H}^{d-1}$ is the $(d-1)$-dimensional Hausdorff measure in $\R^d$. $\Omega$ is equi-affine invariant, that is, for equi-affine invariant maps $\alpha(\bx) := A\bx + \bz$, $A\in \mathrm{SL}(d)$ and $\bz\in\R^d$, we have $\Omega(\alpha(K),\alpha(\eta)) = \Omega(K,\eta)$.
We note that if the boundary of $K$ is a smooth hypersurface, then $\Omega$ is the equi-affine curvature measure obtained by the volume form that is derived by equipping $\bd K$ with the Blaschke--Berwald metric in affine differential geometry, see \cite[Sec.\ 2.1]{LSZH:2015}. \emph{Blaschke's affine surface area} is the total measure $\as_1(K):=\Omega(K,\R^d)$.

\medskip
Lutwak \cite{Lutwak:1996} introduced the family of \emph{$L_p$-affine surface areas} as a natural centro-affine extension of Blaschke's affine surface area for $p\geq 1$. This notion has attracted considerable interest and has been extended for all $p\in\mathbb{R}$, $p\neq -d$, by Hug \cite{Hug1:1996} ($0< p <1$), Meyer \& Werner \cite{MW:2000} and Schütt \& Werner \cite{SW:2003, SW:2004} ($p<0$). See also \cite{HSW:2019, Ludwig:2010, WY:2008}.
For a convex body $K\subset\R^d$ that contains the origin in the interior we define the $L_p$-affine surface area, by
\begin{equation}\label{def:p_surf_area}
    \as_p(K) := \int_{\bd K} \kappa_o(K,\bx)^{\frac{p}{d+p}}  
        \,  C_K(\dint \bx) \quad \text{for $p>-d$},
\end{equation}
where $\kappa_o(K,\bx)$ is the \emph{centro-affine curvature} of $\bd K$ at $\bx$ and $C_K$ is the centro-affine invariant \emph{cone volume measure}, see Section \ref{sec:back_polar} for more details.
Thus $\as_p(K)$ is a centro-affine invariant, that is, for any $A\in\mathrm{SL}(d)$ we have $\as_p(AK)=\as_p(K)$. 

Contained within the familiy of $L_p$-affine affine surface area is the important case $p=d$, i.e., the classical \emph{centro-affine surface area} 
\begin{equation*}
    \as_d(K)=\as_d(K^\circ) = \int_{\bd K} \sqrt{\kappa_o(K,\bx)} \, C_K(\dint\bx).
\end{equation*}
The centro-affine surface area is $\mathrm{GL}(d)$ invariant. It appears in many contexts, see, for example, \cite{BBV:2010, Faifman:2020}.

For $p>0$, respectively $p<0$, $\as_p$ is upper semi-continuous, respectively lower semi-continuous, on convex bodies that contain the origin in the interior with respect to the Hausdorff metric, as shown in \cite{Ludwig:2001, Ludwig:2010, Lutwak:1991}. As shown in \cite{Ludwig:2001, Schutt:1993} $\as_p$ is a valuation on convex bodies that contain the origin in the interior, that is, if $K\cup L$ is convex, then $\as_p(K)+\as_p(L) = \as_p(K\cup L) +\as_p(K\cap L)$. Finally, for the $L_p$-affine isoperimetric inequalities associated with $\as_p$ we refer to \cite{HS:2009, Ivaki:2014, Ivaki:2022,IM:2023, Lutwak:1996, LYZ:2000, LYZ:2004, WY:2008}.

Ludwig \cite{Ludwig:2010} and Ludwig \& Reitzner \cite{LudRei:2010} introduced the family of Orlicz affine surface areas, which are semi-continuous centro-affine invariant valuations, as an extension of Lutwak's $L_p$-affine surface area. Orlicz affine surface areas were latter characterized as the natural centro-affine invariant semi-continuous valuations on convex bodies in the celebrated centro-affine Hadwiger theorem, see \cite{HabPar:2014,LudRei:2010}. For Orlicz-affine isoperimetric inequalites see \cite{HLYZ:2010, Ye:2015}.

\subsection{Polarity and centro-affine invariants} \label{sec:back_polar}

Let $K\subset\R^d$ be a convex body that contains the origin in the interior and admits a rolling ball, i.e., $K$ is of class $\cC^{1,1}$.
Since $K$ contains the origin in the interior, its polar body is
\begin{equation*}
    K^\circ = \left\{\by\in \R^d : h_K(\by) \leq 1\right\},
\end{equation*}
and we note that $\rho_{K^\circ}(\bu) = 1/h_K(\bu)$ for all $\bu\in\S^{d-1}$.
Following ideas of Hug \cite{Hug2:1996}, we define the \emph{polar point $\bx^\circ$} of $\bx\in\bd K$ by
\begin{equation}\label{def:polar_points}
    \bx^\circ := \frac{\bn_K(\bx)}{\bx\cdot \bn_K(\bx)}.
\end{equation}
Note that since $K$ admits a rolling ball the outer unit normal $\bn_K(\bx)$ is uniquely determined for all $\bx\in\bd K$ and $\bn_K(\bx):\bd K\to \S^{d-1}$ is continuous. Since $K$ contains the origin in the interior we have that $h_K(\bn_K(\bx)) = \bx\cdot \bn_K(\bx) >0$. Finally, we notice that since $K$ admits a rolling ball from the inside, its boundary is a hypersurface that has an embedding in $\R^d$ that is at least $\cC^1$ smooth. Thus $\bx^\circ:\bd K \to \bd K^\circ$ is continuous and surjective and $K^\circ$ is strictly convex, but in general $K^\circ$ is not $\cC^1$, for example see Figure \ref{fig:not_c2}.

The measure $C_K$ that is absolutely continuous with respect to $\mathcal{H}^{d-1}$ on $\bd K$ with density
\begin{equation*}
    \frac{\dint C_K}{\dint \mathcal{H}^{d-1}}(\bx) = \bx\cdot \bn_K(\bx),
\end{equation*}
is the \emph{cone-volume measure} of $K$. It is concentrated on $\bd K$ and
\begin{equation*}
    C_{AK}(A\omega) = \left|\det A\right| \, C_K(\omega)
\end{equation*}
for all Borel $\omega\subset \R^d$ and $A\in\mathrm{GL}(\R^d)$. In particular,
\begin{equation*}
    C_K(\R^d) = d\Vol_d(K).
\end{equation*}
Finally, the centro-affine curvature defined by
\begin{equation*}
    \kappa_o(\bx) := \frac{H_{d-1}(K,\bx)}{(\bx\cdot \bn_K(\bx))^{d+1}},
\end{equation*}
is related to the volume of the centered ellipsoid $\mathcal{E}_o(K,\bx)$ that oscillates $\bd K$ at $\bx$, i.e.,
\begin{equation*}
    \Vol_d(\mathcal{E}_o(K,\bx)) = \frac{\Vol_d(B_2^d)}{\sqrt{\kappa_o(K,\bx)}}.
\end{equation*}
Indeed, we may choose an orthogonal basis $\{\be_1,\dotsc,\be_d\}$ in $\R^d$ such that $\be_d =\bn_K(\bx)$ and $\be_1,\dotsc,\be_{d-1}$ are the principal directions of $\bd K$ at $\bx$, then
\begin{equation*}
    \mathcal{E}_o(K,\bx) = A(K,\bx) B_2^d,
\end{equation*}
where $A(K,\bx)\in\mathrm{GL}(\R^d)$ is given by
\begin{equation*}
    A(K,\bx) = \begin{pmatrix}
            \sqrt{\frac{x_d}{\kappa_1(K,\bx)}} & & & x_1\\
            & \ddots & & \vdots\\
            & & \sqrt{\frac{x_d}{\kappa_{d-1}(K,\bx)}} & x_{d-1}\\
            0& \hdots & 0 & x_d
        \end{pmatrix},
\end{equation*}
for $\bx = (x_1,\dotsc,x_d)$, and $\kappa_i(K,\bx)$ are the generalized principal curvatures of $\bd K$ at $\bx$, see for example \cite[Sec.\ 2.6]{Schneider:2014}.
Since by our choice of coordinates $x_d = \bx\cdot \bn_K(\bx)$, we derive that
\begin{equation*}
    \Vol_d(\mathcal{E}_o(K,\bx)) 
    = \left|\det A(K,\bx)\right| \Vol_d(B_2^d) 
    = \sqrt{\frac{x_d^{d+1}}{H_{d-1}(K,\bx)}} \, \Vol_d(B_2^d) 
    = \frac{\Vol_d(B_2^d)}{\sqrt{\kappa_o(K,\bx)}},
\end{equation*}
or equivalently
\begin{equation*}
    \kappa_o(K,\bx) 
    = \left(\frac{\Vol_d(B_2^d)}{\Vol_d(\mathcal{E}_o(K,\bx))}\right)^2 
    = \frac{1}{(\det A(K,\bx))^2} = \frac{H_{d-1}(K,\bx)}{(\bx\cdot \bn_K(\bx))^{d+1}}.
\end{equation*}
See also \cite[Sec.\ 2.2]{HW:2023} for coordinate free way to compute $\kappa_o$.

We recall that for a centro-affine transformation $AK$ of $K$, for $A\in\mathrm{GL}(\R^d)$, we have
\begin{align*}
    \bn_{AK}(A\bx) &= \frac{A^{-\top} \bn_K(\bx)}{\|A^{-\top} \bn_K(\bx)\|}, \quad
    J_{d-1}^{\bd K}(A)(\bx) = \left|\det A\right|\|A^{-\top} \bn_K(\bx)\|,\quad
    (A\bx)^\circ = A^{-\top} \bx^\circ,
\end{align*}
and $\bn_{(AK)^\circ}((A\bx)^\circ) = A \bx / \|A\bx\|$.
Furthermore, 
\begin{equation*}
    H_{d-1}(AK, A\bx) = \frac{\left|\det A\right|^{d-1}}{J_{d-1}^{\bd K}(A)(\bx)^{d+1}} H_{d-1}(K,\bx), \quad
    \kappa_o(AK,A\bx) = \frac{\kappa_o(K,\bx)}{(\det A)^2}.
\end{equation*}


\begin{figure}
    \centering
    \begin{tikzpicture}[scale=2]
        
        \fill[black, opacity=0.3, domain={-1}:{1},smooth,variable=\x] 
            plot ({(1-(\x)^2)/2}, {\x})
            plot ({-(1-(\x)^2)/2}, {-\x})--cycle;
        
        \draw[->,gray] (-2.3,0) -- (2.3,0) node[below] {$x$};
        \draw[->,gray] (0,-1.3) -- (0,1.3) node[left] {$y$};

         \draw[thick, domain=-1:1,smooth,variable=\x] 
            plot ({(1-(\x)^2)/2}, {\x});
        
        \draw[dotted, domain=-1.3:-1,smooth,variable=\x] 
            plot ({(1-(\x)^2)/2}, {\x});
        \draw[dotted, domain=1:1.3,smooth,variable=\x] 
            plot ({(1-(\x)^2)/2}, {\x});

        \draw[thick, domain=-1:1,smooth,variable=\x] 
            plot ({-(1-(\x)^2)/2}, {\x});
        
        \draw[dotted] (-1,-1) -- (-1,1);
        \draw[dotted] (1,1) -- (1,-1);
        \draw[thick] (-1,-1) -- (1,-1);
        \draw[thick] (-1,1) -- (1,1);
        \fill (0,1) circle(0.025);
        \fill (0,-1) circle(0.025);
        \draw[thick] (1,-1) arc(-90:90:1);
        \draw[dotted] (1,1) arc(90:270:1);
        \draw[thick] (-1,1) arc(90:270:1);
        
        \node[below,gray] at (1,0) {$1$};
        \node[left,gray] at (0,1) {$1$};
        \node[below left,gray] at (0,0) {$0$};
        
        \node[above] at (1.7,0) {$K$};
        \node[above] at (-0.6,0.3) {$K^\circ$};
        
        \node at (2.7,-0.5) {$(x-1)^2+y^2=1$};
        \node at (0.85,-0.65) {$2x+y^2=1$};
        
        \fill ({1+1/sqrt(2)},{1/sqrt(2)}) circle(0.025) node[right] {$\bx$};
        \draw[dashed,gray] ({1+1/sqrt(2)},{1/sqrt(2)}) --++({1/(1+sqrt(2))},{1/(1+sqrt(2))});
        \draw[thick, ->] ({1+1/sqrt(2)},{1/sqrt(2)}) -- ({1+1.3/sqrt(2)},{1.3/sqrt(2)}) node[right] {$\bn_K(\bx)$};
        \draw[dashed] (0,0) -- ({1+1/sqrt(2)},{1/sqrt(2)});
        
        \draw[dashed] (0,0) -- ({1/(1+sqrt(2))},{1/(1+sqrt(2))});
        \fill ({1/(1+sqrt(2))},{1/(1+sqrt(2))}) circle(0.025) 
            node[below right, xshift=0.1cm, yshift=0.2cm] {$\bx^\circ$};
            
        \draw[dashed,gray] ({1/(1+sqrt(2))},{1/(1+sqrt(2))}) --++ ({(1+1/sqrt(2))},{1/sqrt(2)});
        \draw[thick, ->] ({1/(1+sqrt(2))},{1/(1+sqrt(2))}) --++ ({0.3*(1+sqrt(2))/sqrt(4+2*sqrt(2))},{.3/sqrt(4+2*sqrt(2))}) 
            node[above right] {$\bn_{K^\circ}(\bx^\circ)$};
        
    \end{tikzpicture}
    \caption{Example for a convex body $K$ that admits a rolling ball, but is not $\cC^2$. The polar $K^\circ$ of $K$ is strictly convex, but not $\cC^1$.}
    \label{fig:not_c2}
\end{figure}
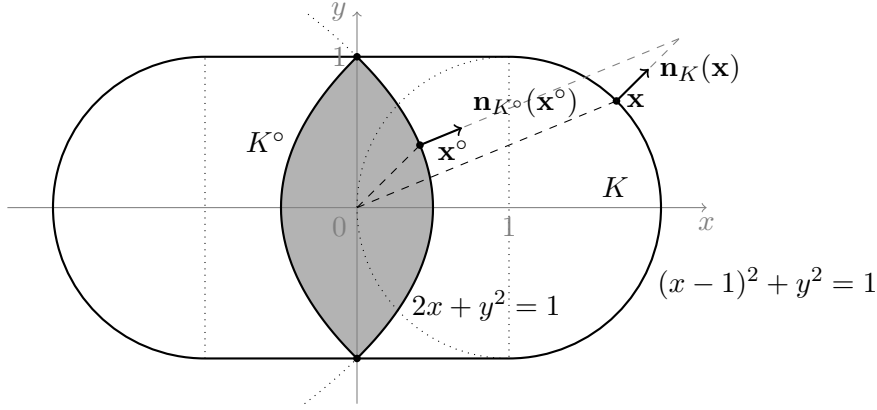

The next lemma follows from a result of Hug \cite{Hug2:1996}.
\begin{lemma}[integral transformation of the cone-volume measure by polarity]\label{lem:integral_formula}
    Let $K\subset \R^d$ be a convex body that admits a rolling ball and that contains the origin in the interior. Then for an integrable function $f:\bd K^\circ \to \R$, we have that
    \begin{align*}
        \int_{\bd K^\circ} f(\by)\, C_{K^\circ}(\dint{\by})
        &= \int_{\bd K} f(\bx^\circ) \kappa_o(K,\bx)\, C_K(\dint \bx).
    \end{align*}
\end{lemma}
\begin{proof}
    First, for the radial map $R:\S^{d-1} \to \bd K^\circ$ defined by $R(\bu) = \rho_{K^\circ}(\bu)\bu$, we have that 
    \begin{equation*}
        J(R)(\bu)=\frac{1}{h_K(\bu)^{d-1} (\bu\cdot \bn_{K^\circ}(R(\bu)))},
    \end{equation*}
    for almost all $\bu\in\S^{d-1}$, see \cite[Lemma 3.1]{Hug2:1996}.
    Hence
    \begin{equation*}
        \int_{\bd K^\circ} f(\by) \, C_{K^\circ}(\dint{\by}) = \int_{\bd K^\circ} f(\by) (\by\cdot \bn_{K^\circ}(\by))\, \mathcal{H}^{d-1}(\dint \by)
        = \int_{\S^{d-1}} \frac{f(R(\bu))}{h_K(\bu)^{d}} \, \mathcal{H}^{d-1}(\dint \bu).
    \end{equation*}
    Now by Theorem \ref{thm:approx_parabola}, we find that
    \begin{align*}
        &\int_{\S^{d-1}} \frac{f(R(\bu))}{h_K(\bu)^{d}} \, \mathcal{H}^{d-1}(\dint \bu)
        = \int_{\bd K} f\left(\frac{\bn_K(\bx)}{\bx\cdot \bn_K(\bx)}\right) \kappa_o(K,\bx) \, C_K(\dint \bx).
    \end{align*}
    Thus the statement of the lemma follows.
\end{proof}

By Lemma \ref{lem:integral_formula} we derive for any continuous $\Phi:[0,+\infty)\to [0,+\infty)$ 
\begin{multline}\label{eqn:polarity_formula}
        \as_{\Phi}(K^\circ) 
        := \!\int_{\bd K^\circ} \!\!\!\! \Phi\left(\kappa_o(K^\circ,\by)\right)\, C_{K^\circ}(\dint \by)
         = \!\int_{\bd K}       \!\!\!\! \Phi\left(\kappa_o(K,\bx)^{-1}\right) \kappa_o(K,\bx) \, C_K(\dint\bx)
        =\as_{\Phi^*}(K),
\end{multline}
where $\Phi^*(s) = s\Phi(1/s)$ and we used the fact that for almost all $\bx\in \bd K$ we have that
\begin{equation}\label{eqn:kappa_0-inv}
    \kappa_o(K,\bx)\kappa_o(K^\circ,\bx^\circ) = 1,
\end{equation}
see \cite[Thm.\ 2.8]{Hug2:1996}. 

Equation \eqref{eqn:polarity_formula} yields a direct proof for the polarity relation in the family of Orlicz-affine surface areas. In particular, for $\Phi(t)=t^{\frac{p}{n+p}}$ we find that $\as_p(K^\circ) = \as_{d^2/p}(K)$. This was previously established for the $L_p$-affine surface area, $p>0$, by Hug \cite[Thm.~3.2]{Hug2:1996} and for Orlicz-affine surface areas by Ludwig \cite[Thm.\ 4]{Ludwig:2010}, who gave an elegant proof using a very deep characterization result on centro-affine invariant valuations. 

\section{Weighted floating bodies and polarity}

The floating body of a convex body is a classical affine construction that can be traced back to Dupin in the 19th century. A generalization of Dupin's floating body was introduced in \cite{SW:1990} and a weighted notion of floating body in \cite{Werner:2002}, see also \cite{BLW:2018}. More recent generalizations and connections of the concept of floating body can be found in \cite{BW:2016, Fresen:2012, Klartag:2018, NSW:2019}.

\begin{definition}[weighted floating body]\label{def:WFB}
    Let $K\subset \R^d$ be a convex body and let $\phi:K\to [0,+\infty)$ be a continuous weight function. Then, for $\delta>0$, the $\phi$-weighted floating body $K_\delta^\phi$ is defined by
    \begin{equation*}
        K_\delta^\phi = \bigcap \{ H^- : \Vol_d^\phi(H^+\cap K)\leq \delta\},
    \end{equation*}
    where $(H^-,H^+)$ is any pair of closed half-spaces that share a boundary hyperplane, and $\Vol_d^\phi$ is an absolutely continuous measure with respect to the Lebesgue measure with continuous density function $\phi$.
\end{definition}

Definition \ref{def:WFB} can be seen as a construction of the floating body from the outside by intersecting all closed half-spaces that contain it, that is, as a special\emph{Wulff shape}, see, for example, \cite[Ch.\ 7.5]{Schneider:2014}. Note that weighted volume derivatives of Wulff shapes were considered recently in \cite[Lem.\ 2.7]{KL:2023}.

We may see the weighted floating body as a Wulff shape: let $h_\delta:\S^{d-1} \to \R$ be defined implicitly by
\begin{equation*}
    \delta = \Vol_d^\phi\left(K\cap H^+\left(\bu, h(K,\bu)-h_\delta(\bu)\right)\right).
\end{equation*}
Then
\begin{equation*}
    K_\delta^\phi = [h(K,\cdot)-h_\delta(\cdot)] := \{\bx \in \R^d : \text{for all $\bu \in \S^{d-1}$ we have } \bx\cdot \bu \leq h(K,\bu)-h_\delta(\bu)\}.
\end{equation*}
In particular, if $K_\delta^\phi\neq \emptyset$, then we have
\begin{equation*}
    h(K,\bu) - h(K_\delta^\phi,\bu) \geq h_\delta(\bu) \qquad \text{for all $\bu\in\S^{d-1}$}.
\end{equation*}
Note that $h(K,\bu)-h(K_\delta^\phi,\bu) = h_\delta(\bu)$ if and only if the cap $K\cap H^+(\bu,h(K_\delta^\phi,\bu))$ with normal $\bu$ tangent to $K_\delta^\phi$ has $\phi$-volume exactly $\delta$. 
However, in general a hyperplane that cuts off $\phi$-volume $\delta$ from $K$ is not necessary tangent to the convex floating body $K_\delta^\phi$.
Meyer and Reisner \cite[Thm.~3]{MR:1991} showed that if $K$ is symmetric and $\phi\equiv 1$, then $h(K_\delta,\bu) = h(K,\bu)-h_\delta(\bu)$ for all $\bu\in \S^{d-1}$ and all $\delta\in (0,\frac{1}{2}\Vol_d(K))$ and Leichtweiss showed the following theorem:

\begin{theorem}[regularity of the floating body of convex bodies of class $\cC^{1,1}$]\label{thm:leichtweiss}
	Let $K\subset \R^d$ be a convex body that admits a rolling ball of radius $r>0$. 
	Set $\delta_0 := \frac{1}{2} r^d \Vol_d(B_2^d)$.
	Then for all $\delta \in (0,\delta_0)$ we have
	\begin{enumerate}
		\item[i)] the convex floating body $K_\delta$ is of class $\cC^2_+$, and
		\item[ii)] every hyperplane that cuts off a cap of volume $\delta$ from $K$ is tangent to $K_\delta$, that is, $h(K_\delta,\bu)=h(K,\bu)-h_\delta(\bu)$ for all $\bu\in\S^{d-1}$.
	\end{enumerate}
\end{theorem}
\begin{proof}
    See \cite[Satz 1]{LW1:1986} and \cite[Satz 1 \& 2]{LW2:1986}.
\end{proof}

Instead of a construction from the outside, there is also an equivalent construction of $K_\delta^\phi$ from the inside which is described next and follows ideas of \cite{BL:1988}.

\begin{lemma}[minimal cap density]\label{lem:MCD}
    Let $K\subset \R^d$ be a convex body and let $\phi:K\to [0,+\infty)$ be a continuous and non-negative weight function. The \emph{minimal cap density} of $K$ with respect to $\phi$ is defined by 
    \begin{equation*}
        \mcd_{K,\phi}(\bx) = \min_{\bu\in\S^{d-1}} \Vol_d^\phi(K\cap H^+(\bu,\bx\cdot \bu)),\quad \bx\in K.
    \end{equation*}
    Then the interior of the weighted floating body $K_\delta^\phi$ is the strict $\delta$-superlevel set of $\mcd_{K,\phi}$, i.e.,
    \begin{equation*}
        \interior K_\delta^\phi = [\mcd_{K,\phi} > \delta] := \{\bx\in K: \mcd_{K,\phi}(\bx) > \delta\}.
    \end{equation*}
    In particular, if $\phi$ is strictly positive in an open neighborhood of $\bd K$, then $K_\delta^\phi = [\mcd_{K,\phi}\geq \delta]$.
\end{lemma}
\begin{proof}
    Let $\bx\in K_\delta^\phi$. Then for all half-spaces $H^+$ that contain $\bx$ in the interior, we have that $\Vol_d^\phi(K\cap H^+) > \delta$. Hence $\mcd_{K,\phi}(\bx)\geq \delta$, or equivalently $K_\delta^\phi\subseteq [\mcd_{K,\phi}\geq \delta]$. Furthermore, if $\mcd_{K,\phi}(\bx)=\delta$, then there exist $\bu\in\S^{d-1}$ such that $\Vol_d^\phi(K\cap H^+(\bu,\bx\cdot \bu)) = \delta$. Thus $K_\delta^\phi \subset H^-(\bu,\bx\cdot\bu)$ and $x\in \bd H^-(\bu,\bx\cdot \bu)$ which yields $\bx\in \bd K_\delta^\phi$. Thus, $\interior K_\delta^\phi \subseteq [\mcd_{K,\phi}>\delta]$.
    
    Conversely, if $\bx\not\in \interior K_\delta^\phi$, then there is a half-space $H^+(\bu,t)$ that contains $\bx$ such that $\Vol_d^\phi(K\cap H^+(\bu,t))\leq\delta$. 
    Since $\bx$ is contained in $H^+(\bu,t)$, we have that $t\leq \bu\cdot \bx$, and therefore
    \begin{equation*}
        \Vol_d^\phi(K\cap H^+(\bu,t) \setminus H^+(\bu,\bu\cdot \bx))\geq 0.
    \end{equation*}
    This yields
    \begin{equation*}
        \mcd_{K,\phi}(\bx) \leq \Vol_d^\phi(K\cap H^+(\bu,\bu\cdot \bx)) \leq \Vol_d^\phi(K\cap H^+(\bu,t)) \leq \delta
    \end{equation*}
    Thus $\bx\not\in[\mcd_{K,\phi} > \delta]$ and therefore $\interior K_\delta^\phi \supseteq [\mcd_{K,\phi} > \delta]$.
    
    Finally, assume that $\bx \not\in K_\delta^\phi$ and that $\varphi$ is continuous and strictly positive in a neighborhood of $\bd K$.
    Then there exists a half-space $H^+(\bu,t)$ that contains $\bx$ in the interior and such that $\Vol_d^\phi(K\cap H^+(\bu,t))\leq \delta$. Since $\bx$ is contained in the interior of $H^+(\bu,t)$, we conclude that $t< \bu\cdot \bx$, and since $\varphi$ is continuous and strictly positive we find
    \begin{equation*}
        \Vol_d^\phi(K\cap H^+(\bu,t) \setminus H^+(\bu,\bu\cdot \bx)) > 0.
    \end{equation*}
    Thus
    \begin{equation*}
        \mcd_{K,\phi}(\bx) \leq \Vol_d^\phi(K\cap H^+(\bu,\bu\cdot \bx)) < \Vol_d^\phi(K\cap H^+(\bu,t)) \leq \delta,
    \end{equation*}
    and therefore $\bx \not\in [\mcd_{K,\phi}\geq \delta]$, which yields $K_\delta^\phi\supseteq [\mcd_{K,\phi}\geq \delta]$.
\end{proof}

\begin{lemma}[monotonicity of weighted floating bodies]\label{lem:mono_float}
    Let $K\subset \R^d$ be a convex body and let $\phi: K \to [0,+\infty)$ be a continuous weight function. If $\psi: K\to [0,+\infty)$ is another continuous weight function with $\psi\leq \phi$ on $K$, then
    \begin{equation*}
        K_\delta^\psi \subseteq K_\delta^\phi \quad \text{for all $\delta\geq 0$}.
    \end{equation*}
    This implies, in particular, that if $L\subseteq K$ is another convex body, then
    \begin{equation*}
        L_\delta^\phi \subseteq K_\delta^\phi \quad \text{for all $\delta\geq 0$}.
    \end{equation*}
\end{lemma}
\begin{proof}
    Since $\psi \leq \phi$ implies that $\Vol_d^\psi(K\cap H^+) \leq \Vol_d^\phi(K\cap H^+)$ for any closed half-space $H^+$, we find that $\mcd_{K,\psi} \leq \mcd_{K,\phi}$. Thus
    $K_\delta^\psi = \closure [\mcd_{K,\psi}> \delta] \subseteq \closure [\mcd_{K,\phi}> \delta] = K_\delta^\phi.$
\end{proof}

\begin{remark}[floating body of a measure]
    The proof of Lemma \ref{lem:MCD} shows that even for a general measure $\mu$ we always have
    \begin{equation*}
        \closure [\mcd_\mu >\delta] \subseteq 
        [\mu]_\delta \subseteq [\mcd_\mu \geq  \delta],
    \end{equation*}
    where 
    $[\mu]_\delta:= \bigcap \{ H^-: \mu(H^+)\leq \delta\}$ and 
    $\mcd_\mu(x) := \inf \{\mu(H^+(\bu,\bu\cdot \bx)): \bu\in\S^{d-1}\}$ for $\bx\in \R^d$.
    
    Furthermore, if the support of $\mu$ is bounded and there exists $\varepsilon>0$ such that for every extremal point $\bx$ in the convex hull of the support of $\mu$ and for all $t\in(0,\varepsilon)$ the function $t\mapsto \mu(H^+(\bu,\bx\cdot\bu-t))$ is strictly increasing, then there exists $\delta_0>0$ such that $[\mu]_\delta = [\mcd_\mu\geq \delta]$ for all $\delta \in (0,\delta_0)$.
    
    In particular, for the construction of the weighted surface body we may consider a continuous function 
    $\varphi: \bd K \to (0,+\infty)$. 
    Then the measure $\dint \nu = \varphi\, \dint\mathcal{H}^{d-1}_{\bd K}$ gives rise to the \emph{surface body}, introduced by Schütt \& Werner in \cite{SW:2003, SW:2004}, by $[\bd K]_\delta^\phi:=[\nu]_\delta=[\mcd_\nu\geq \delta]$ for all $\delta >0$.
\end{remark}

The following lemma is an extension of \cite[p.\ 311]{MW:2000}, where $K$ is assumed to be of of class $\cC^2_+$ and $\phi$ is a uniform weight function. Compare also \cite[Lem.\ 6]{SW:1990}.
\begin{lemma}(uniform convergence rate)\label{lem:upper_bound_support}
	Let $K$ be a convex body of class $\cC^{1,1}$, that is, $K$ admits a rolling ball of radius $r:=\mathrm{roll}(K)>0$. Further, let $\phi:K\to (0,\infty)$ be a integrable function such that $\alpha:=\inf_{K} \phi >0$.
	Then there exists $\delta_0 = \delta_0(d,r,\alpha) >0$ and $C=C(d,r,\alpha)>0$ such that for all $\delta \in (0,\delta_0)$
	\begin{equation*}
        |h(K,\bu) - h(K_\delta^\phi,\bu)|\leq C \delta^{\frac{2}{d+1}} \qquad \text{for almost all $\bu\in\S^{d-1}$}.
	\end{equation*}
\end{lemma}
\begin{proof}
    The general idea of the proof is to use a ball of radius $r$ and a uniform density $\alpha$ to approximate the weighted volume of a cap of $K$, see Figure \ref{fig:upper_bound_support}. First we recall that almost every normal vector $\bu\in\S^{d-1}$ of $K$ is a regular normal vector, see \cite[Thm.\ 2.2.11]{Schneider:2014}, and therefore the face set $F(K,\bu)=\{\by\in K : \by\cdot \bu = h_K(\bu)\}=\{\bx\}$ is a singular boundary point $\bx$. This also implies that the support function $h_K$ is differentiable at $\bu$ and $\nabla h_K(\bu) = \bx$, see \cite[Cor.\ 1.7.3]{Schneider:2014}. 
    Furthermore, since $K$ is of class $\cC^1$, the boundary point is regular, which yields $\bn_K(\bx)=\bu$. 
    
    Since $K$ admits a rolling ball of radius $r$ it contains the ball $B(r):=(\bx-r\bu)+B_2^d(r)$ of radius $r$ that touches $\bd K$ from the inside at $\bx$, and therefore
    $K\supseteq B(r)$.
    Hence, by Lemma \ref{lem:mono_float}, we conclude
    \begin{equation*}
        K_\delta^\phi \supseteq B(r)_{\delta/\alpha} \neq \emptyset \quad \text{for all $\delta < \delta_1:= \frac{\alpha}{2} \Vol_d(B(r)) = \frac{\alpha}{2} r^d \Vol_d(B_2^d)$},
    \end{equation*}
    which yields
    \begin{equation*}
        h(K,\bu) - h(K_\delta^\phi,\bu) \leq h(B(r),\bu)-h(B(r)_{\delta/\alpha},\bu)=:\tilde{h}_\delta.
    \end{equation*}
    Now the floating body of $B(r)$ is again a Euclidean ball, and thus the cap 
    \begin{equation*}
    C^d(r,\tilde{h}_\delta):=B(r)\cap H^+(\bu,h_K(\bu)-\tilde{h}_\delta)
    \end{equation*}
    of $B(r)$ cuts off a volume of exactly $\delta/\alpha$ for all $\delta < \delta_1$. This yields, by Lemma \ref{lem:ball_cap}, that
    \begin{equation*}
        \frac{\delta}{\alpha} = \Vol_d(C^d(r,\tilde{h}_\delta)) \geq \frac{a_d}{2} r^{\frac{d-1}{2}} \tilde{h}_\delta^{\frac{d+1}{2}},
    \end{equation*}
    for all $\delta$ small enough so that $\tilde{h}_\delta < \frac{2r}{d}$. So by choosing $\delta < \delta_2:=\frac{\alpha}{2} (1/d)^{\frac{d+1}{2}} a_d r^d$, we find that
    \begin{equation*}
        \Vol_d(C^d(r,\tilde{h}_\delta)) 
        = \frac{\delta}{\alpha} < \frac{\delta_2}{\alpha}= \frac{a_d}{2} r^{\frac{d-1}{2}} \left(\frac{r}{d}\right)^{\frac{d+1}{2}} 
        < \Vol_d\left(C^d\left(r,\frac{r}{d}\right)\right),
    \end{equation*}
    which yields $\tilde{h}_\delta < \frac{2r}{d}$.
    Thus
    \begin{equation*}
        h(K,\bu) - h(K_\delta^\phi,\bu) \leq \tilde{h}_\delta \leq r^{-\frac{d-1}{d+1}} \left(\frac{2\delta}{a_d\alpha}\right)^{\frac{2}{d+1}},
    \end{equation*}
    for all 
    \begin{equation*}
        \delta < \delta_0:=\min\{\delta_1,\delta_2\} = \frac{\alpha}{2} r^d \min\{\Vol_d(B_2^d), (1/d)^{(d+1)/2}a_d\}.
    \end{equation*}
    This concludes the proof.
\end{proof}
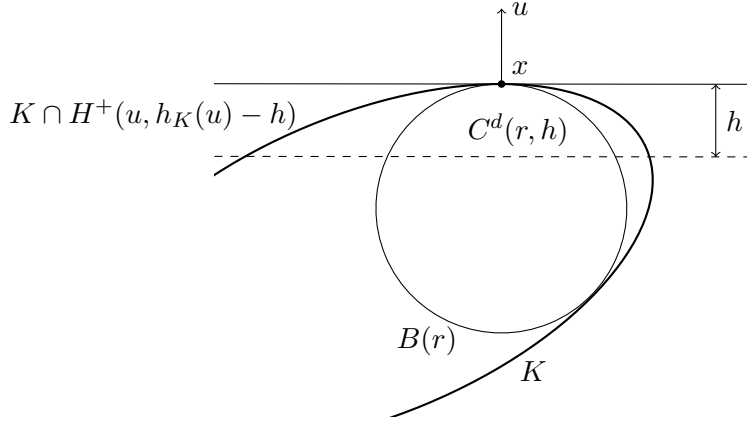
\begin{figure}
    \centering
    \begin{tikzpicture}[scale=2]
        \def\alp{50};
        \def\R{0.825};
        \def\t1{2};
        \def\h{0.6};
        
        \begin{scope}
        \clip (-1,-1) rectangle (2.5,2);
        
        \begin{scope}[rotate=22.7]
        
            \draw[thick] (0,0) circle (2 and 1);
            \draw[->] ({2*cos(\alp)},{sin(\alp)}) --++ ({0.5*1/sqrt(4-3*cos(\alp)^2)*cos(\alp)},{0.5*2/sqrt(4-3*cos(\alp)^2)*sin(\alp)})
                node[right] {$u$};
        
            \fill ({2*cos(\alp)},{sin(\alp)}) circle(0.025) node[above right] {$x$};
        
            \draw ({2*cos(\alp)-\R/sqrt(4-3*cos(\alp)^2)*cos(\alp)},{sin(\alp)-\R*2/sqrt(4-3*cos(\alp)^2)*sin(\alp)}) circle(\R);
        
            \draw ({2*cos(\alp)+\t1*2/sqrt(4-3*cos(\alp)^2)*sin(\alp)},{sin(\alp)-\t1/sqrt(4-3*cos(\alp)^2)*cos(\alp)}) 
                --++ ({-2*\t1*2/sqrt(4-3*cos(\alp)^2)*sin(\alp)},{2*\t1/sqrt(4-3*cos(\alp)^2)*cos(\alp)});
        
            \draw[dashed] ({\h*2*cos(\alp)+\t1*2/sqrt(4-3*cos(\alp)^2)*sin(\alp)},{\h*sin(\alp)-\t1/sqrt(4-3*cos(\alp)^2)*cos(\alp)}) 
                --++ ({-2*\t1*2/sqrt(4-3*cos(\alp)^2)*sin(\alp)},{2*\t1/sqrt(4-3*cos(\alp)^2)*cos(\alp)});
        \end{scope}
        \end{scope}
        
        \draw[<->] (2.3,1.2) -- (2.3,0.72) node[midway,right] {$h$};
        
        \node at (0.4,-0.5) {$B(r)$};
        \node at (1.1,-0.7) {$K$};
        
        \node at (-1.4,1) {$K\cap H^+(u,h_K(u)-h)$};
        \node at (1,0.9) {$C^d(r,h)$};
    \end{tikzpicture}
    \caption{Sketch for the proof of Lemma \ref{lem:upper_bound_support}.}
    \label{fig:upper_bound_support}
    \end{figure}

\subsection{Weighted Polar Volume of the Weighted Floating Body}\label{sec:weighted_polar_volume}

We aim to show the following main theorem:

\begin{theorem}[weighted volume derivative of the polar weighted floating body] \label{thm:main_weighted}
	Let $K\subset \R^d$ be a convex body of class $\cC^{1,1}$, and assume that $K$ contains the origin in the interior. Furthermore, let $\phi:K\to (0,\infty)$ be a strictly positive and continuous function, and let $\psi : \R^d \to [0,\infty)$ be a non-negative continuous function.
	Then
	\begin{align}\label{eqn:main1}
		a_d^{\frac{2}{d+1}}\lim_{\delta\to 0^+} \frac{\Vol_d^\psi\big((K_\delta^\phi)^\circ\big)-\Vol_d^\psi(K^\circ)}{\delta^{\frac{2}{d+1}}} 
		&=  \int_{\bd K} \kappa_o(K,\bx)^{\frac{d+2}{d+1}} \phi(\bx)^{-\frac{2}{d+1}} \psi(\bx^\circ)\, C_K(\dint \bx).
	\end{align}
\end{theorem}

Theorem \ref{thm:main_weighted} extends \cite[Thm.~8]{MW:2000} where $K$ is assumed to be of class $\cC^2_+$ and $\phi=\psi=1$ are uniform weights. 
Note that the theorem is still true if one considers $\phi$ to be defined near $\bd K$, say $\phi:K\setminus K_{\delta_0}\to (0,+\infty)$ for some fixed $\delta_0>0$, and $\psi$ to be defined near $\bd K^\circ$, say $\psi: (K_{\delta_0})^\circ\setminus (\interior K^\circ) \to [0,+\infty)$.
We also note that $\cC^1_{\mathrm{sc}}$-regularity is in general not enough for the limit to be finite as the following example shows.

\begin{figure}[h]
    \centering
    \begin{tikzpicture}[scale=3]
        
        \def\alp{20};
        \def\h{0.6};

        \def\x0{\h/cos(\alp)};
        \def\y{(1-\x0^(1.2))^(1/1.2)/(((1-\x0^(1.2))/\x0)^(0.2/1.2)-tan(\alp))}

        \fill[black!20] ({\x0},0) -- ({\x0}, {(1-(\x0)^(1.2))^(1/1.2)}) -- ({\x0-tan(\alp)*\y},{\y}) -- cycle;
        \fill[black!10, domain={-(1-(\x0)^(1.2))^(1/1.2)}:{(1-(\x0)^(1.2))^(1/1.2)}, variable=\x, smooth] plot ({(1-abs(\x)^(1.2))^(1/1.2)}, {\x});
        
        \draw[->,gray] (-1.2,0) -- (1.2,0) node[below] {$\be_1$};
        \draw[->,gray] (0,-1.2) -- (0,1.2) node[left] {$\be_2$};
        
        \node[below, gray] at (1,0) {$1$};
        \node[left, gray] at (0,1) {$1$};
        \node[below left, gray] at (0,0) {$o$};

        \fill ({\x0-tan(\alp)*\y},{\y}) circle(0.02) node[above right] {$(x,y)$};

        \draw[thick, domain=-1.0:1.0, variable=\x, smooth] plot ({\x}, {(1-abs(\x)^(1.2))^(1/1.2)});
        \draw[thick, domain=-1.0:1.0, variable=\x, smooth] plot ({\x}, {-(1-abs(\x)^(1.2))^(1/1.2)});
        
        \draw[dotted, domain={2*\x0-1}:{\x0}, variable=\x, smooth] plot ({\x}, { (1-abs(2*\h/cos(\alp)-\x)^(1.2))^(1/1.2)});
        \draw[dotted, domain={2*\x0-1}:{\x0}, variable=\x, smooth] plot ({\x}, {-(1-abs(2*\h/cos(\alp)-\x)^(1.2))^(1/1.2)});

        \draw[->] (0,0) -- (\alp:1) node[right] {$\bu$};
        
        \draw (0.25,0) arc (0:\alp:0.25) node[midway, left,xshift=0.12cm, yshift=-0.03cm] {$\alpha$};
        
        \node at (-0.8,-0.8) {$x^p+y^p=1$};
        \node at (-0.8,0.8) {$B_p^2$};
        
        \draw ({\h*cos(\alp)-sin(\alp)},{\h*sin(\alp)+cos(\alp)}) -- ({\h*cos(\alp)+sin(\alp)},{\h*sin(\alp)-cos(\alp)}) node[right]{$H(\bu,h(B_p^2,\bu)-h)$};
        
        \draw ({\h/cos(\alp)},0.7) -- ({\h/cos(\alp)},-0.7) node[below] {$H(\be_1,x_0)$};
        \fill ({\x0},{(1-(\x0)^(1.2))^(1/1.2)}) circle(0.02) node[above right] {$(x_0,y_0)$};
        
        \fill ({\x0},0) circle(0.02);
        
        \draw ({(\h-0.07)*cos(\alp)},{(\h-0.07)*sin(\alp)}) -- ({(\h-0.07)*cos(\alp)+0.07*sin(\alp)},{(\h-0.07)*sin(\alp)-0.07*cos(\alp)})
                                                          -- ({(\h)*cos(\alp)+0.07*sin(\alp)},{(\h)*sin(\alp)-0.07*cos(\alp)});
        
        \node at (0.3,0.4) {$A(\bu)$};
    \end{tikzpicture}
    \caption{Sketch for Example \ref{ex:Lp-ball}.}
    \label{fig:Lp-ball}
\end{figure}
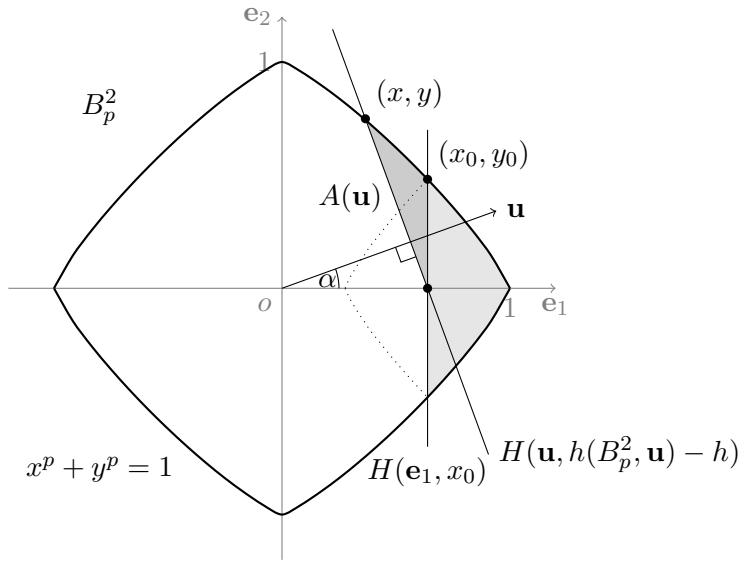

\begin{example}\label{ex:Lp-ball}
    We consider the $2$-dimensional convex body $B_p^2 := \{(x,y)\in\R^2: |x|^p+|y|^p\leq~1\}$ for $p\in(1,2)$ which is of class $\cC^1_{\mathrm{sc}}$. Then the area of a  cap of height $h$ in direction $\bu=(\cos\alpha,\sin\alpha)\in\S^1$ of $B_p^2$ can be estimated by
    \begin{equation}\label{eqn:example1}
        0\leq \Vol_2(B_p^2 \cap H^+(\bu, h(B_p^2,\bu)-h)) - \Vol_2(B_p^2 \cap H^+(\be_1, x_0)) \leq \Vol_2(A(\bu)),
    \end{equation}
    where $x_0:=(h(B_p^2,\bu)-h)\sec \alpha$, which is the distance of the cap $H^+(\be_1,x_0)$ from the origin. Further $y_0=f(x_0):=(1-x_0^p)^{1/p}$ is the $y$-coordinate where $(x_0,y_0)\in\bd B_p^2\cap H(\be_1,x_0)$. Moreover, $A(\bu)$ is the triangle spanned by $(x_0,0)$, $(x_0,y_0)$ and $(x,y)$, where $(x,y)$ are the coordinates of the intersection of $H(\bu,h(B_p^2,\bu)-h)$ with the tangent line of $\bd B_p^2$ at $(x_0,y_0)$. 
    See Figure \ref{fig:Lp-ball}. Thus, with an absolute constants $C_0,C_1>0$,
    \begin{align*}
        1+C_0\alpha^{\frac{p}{p-1}}-h\geq x_0 &= (h(B_p^2,\bu) - h)\sec \alpha = (1+(\tan \alpha)^{\frac{p}{p-1}})^{\frac{p-1}{p}} - h \sec \alpha\geq  1-h\sec\alpha,\\
        0\leq y_0 &= f(x_0) = (1-x_0^p)^{1/p} \leq (ph\sec\alpha)^{1/p},\\
        x_0-x &= y \tan \alpha = f'(x_0)(y_0-y) = -\left(\frac{x_0}{y_0}\right)^{p-1} (y_0-y), \text{ and }\\
        y &= \frac{y_0}{1-(y_0/x_0)^{p-1}\tan\alpha} \leq y_0 \left(1+C_1 \alpha h^{\frac{p-1}{p}}\right).
    \end{align*}
    We further estimate, with an absolute constant $C_2,C_3>0$,
    \begin{align*}
        \Vol_2(B_p^2\cap H^+(\be_1,x_0)) &= 2 \int_{x_0}^1 (1-x^p)^{1/p}\, \dint x,
            \leq 2(1-x_0^p)^{\frac{p+1}{p}},
            \leq 2 (ph\sec\alpha)^{\frac{p+1}{p}},\\
        \Vol_2(B_p^2\cap H^+(\be_1,x_0)) &\geq \frac{2}{p+1} (1-x_0^p)^{\frac{p+1}{p}} \geq C_2 \left(h-C_0\alpha^{\frac{p}{p-1}}\right)^{\frac{p+1}{p}}, \text{ and }\\
        \Vol_2(A(\bu)) &\leq \frac{\tan \alpha}{2} yy_0 \leq \frac{\tan\alpha}{2} (ph\sec\alpha)^{2/p} \left(1+C_1\alpha h^{\frac{p-1}{p}}\right) \leq C_3 \alpha h^{2/p}.
    \end{align*}
    We derive by \eqref{eqn:example1}
    \begin{equation*}
        C_2 \left(h-C_0\alpha^{\frac{p}{p-1}}\right)^{\frac{p+1}{p}} 
        \leq \Vol_2(B_p^2 \cap H^+(\bu, h(B_p^2,\bu)-h)) 
        \leq 2(ph\sec \alpha)^{\frac{p+1}{p}} + C_3 \alpha h^{\frac{2}{p}}.
    \end{equation*}
    For $\delta\in(0,1)$ we implicitly define $h_\delta(\bu)$ by $\delta = \Vol_2(B_p^2\cap H^+(\bu, h(B_p^2,\bu)- h_\delta(\bu))$. Then, for all $\alpha \leq \delta^{\frac{p-1}{p+1}}$ and some absolute constants $C_4,C_5>0$, we have
    \begin{align*}
        \delta &\leq h_\delta(\bu)^{2/p} \left(2 (p\sec\alpha)^{\frac{p+1}{p}} h_\delta(\bu)^{\frac{p-1}{p}}+C_3\alpha\right) \\
        &\leq \delta^{\frac{p-1}{p+1}} h_\delta(\bu)^{2/p} \left(2(p\sec\alpha)^{\frac{p+1}{p}} (C_2^{-\frac{p}{p+1}}+C_0)^{\frac{p-1}{p}}+C_3\right) 
        \leq C_4 \delta^{\frac{p-1}{p+1}} h_\delta(\bu)^{2/p} 
    \end{align*}
    which yields
    \begin{equation*}
        h(B_p^2,\bu)-h((B_p^2)_\delta,\bu) \geq h_\delta(\bu) 
        \geq C_5 \delta^{\frac{p}{p+1}},
    \end{equation*}
    for all $\delta>0$ small enough.
    Thus, since $h((B_2^p)_\delta,\bu)\leq h(B_2^p,\bu) \leq \sqrt{2}$, we find for some constant $C_6>0$,
    \begin{align*}
        \Vol_2((B_p^2)_\delta^\circ) - \Vol_2((B_p^2)^\circ) 
        &= \int_{\S^1} \frac{h(B_p^2,\bu)^2 - h((B_p^2)_\delta,\bu)^2}{h(B_p^2,\bu)^2 h((B_p^2)_\delta,\bu)^2}\, \dint\bu
        \geq \frac{1}{\sqrt{2}}\int_{\S^1} h_\delta(\bu) \, \dint \bu\\
        &\geq \frac{C_5}{\sqrt{2}} \int_{0}^{\delta^{\frac{p-1}{p+1}}}  \delta^{\frac{p}{p+1}} \dint \alpha
         = C_6 \delta^{\frac{2p-1}{p+1}}
    \end{align*}
    We conclude, that for $1<p<\frac{5}{4}$, we have
    \begin{equation*}
        \liminf_{\delta\to 0^+} \frac{\Vol_2((B_p^2)_\delta^\circ) - \Vol_2((B_p^2)^\circ) }{\delta^{2/3}} \geq  C_6 \liminf_{\delta\to 0^+} \delta^{\frac{2p-1}{p+1}-\frac{2}{3}} = +\infty.
    \end{equation*}
    The bound $p>\frac{5}{4}$ appears to be sharp, since by \cite[Ex.\ 8]{SW:2004}, for $d=2$ we have
    \begin{equation*}
        \as_{-d(d+2)}(B_p^d) = \as_{-\frac{d}{d+2}}(B_q^d)= \begin{cases} 
                                +\infty, &\text{if $p\leq \frac{5}{4}$},\\
                                \frac{4}{p} (p-1)^{4/3}\, \Gamma(\frac{4p-5}{3p})^2 / \Gamma(\frac{2(4p-5)}{3p}), &\text{if $p> \frac{5}{4}$}.
                               \end{cases}
    \end{equation*}
\end{example}

The following lemma and its proof are a generalization of \cite[Claim on page 311]{MW:2000}.
\begin{lemma}[uniform integrable upper bound] \label{lem:interchange}
    Let $K\subset \R^d$ be a convex body of class $\cC^{1,1}$ and let $\phi:K\to (0,\infty)$ be a integrable function such that $\inf_K \phi>0$.
    Furthermore, let $\psi: \R^d \to (0,\infty)$,
    be an integrable function such that $\sup \psi < +\infty$.
    Then there exists $C = C(K,\phi,\psi) > 0$ and $\delta_0=\delta_0(K,\phi,\psi)>0$ such that for all $\delta \in (0,\delta_0)$
    \begin{equation*}
        \int_{h(K_\delta^\phi,\bu)}^{h(K,\bu)} \psi\left(\bu/s\right) s^{-(d+1)}\, \dint{s} \leq C \delta^{\frac{2}{d+1}}, \qquad \text{for almost all $\bu\in \S^{d-1}$.}
    \end{equation*}
\end{lemma}
\begin{proof}
	We set $\alpha := \inf_K \phi$ and $\beta:= \sup \psi$.
	We may assume w.l.o.g.\ that $K$ contains the origin in the interior. Then there is $\rho > 0$ such that
	\begin{equation*}
		2\rho B_2^d \subset K \subset \frac{1}{2\rho}B_2^d.
	\end{equation*}
	We choose $\delta_1$ so small that 
	\begin{equation}\label{eqn:delta0_cond1}
		\rho B_2^d \subset K_\delta^\phi,
	\end{equation}
	for all $\delta \in (0,\delta_1)$. This yields
	\begin{equation*}
		h(K,\bu) \geq h(K_\delta^\phi,\bu) \geq \rho,
	\end{equation*}
	for all $\delta<\delta_1$ and $\bu\in \S^{d-1}$.
	Thus
	\begin{equation*}
		\int_{h(K_\delta^\phi,\bu)}^{h(K,\bu)} \psi\left(\bu/s\right) s^{-(d+1)}  \, \dint{s}  
		\leq \frac{\beta}{\rho^{d+1}} \left(h(K,\bu)-h(K_\delta^\phi,\bu)\right).
	\end{equation*} 
	By Lemma \ref{lem:upper_bound_support} there exists $C_0>0$ and $\delta_0 \in (0,\delta_1)$ such that
	\begin{equation*}
		\frac{\beta}{\rho^{d+1}}\frac{h(K,\bu)-h(K_\delta^\phi,\bu)}{\delta^{\frac{2}{d+1}}}
		\leq \frac{\beta}{\rho^{d+1}} C_0 =: C,
	\end{equation*}
	for all $\delta \leq \delta_0$ and almost all $\bu\in\S^{d-1}$.
\end{proof}

Meyer and Werner \cite[p.\ 311]{MW:2000} showed the following theorem for convex bodies of class $\cC^2_+$ and the uniform weight $\phi\equiv 1$.

\begin{theorem}[pointwise limit] \label{thm:limit_weights}
    Let $K\subset \R^d$ be a convex body of class $\cC^{1,1}$ and let $\phi:K\to (0,\infty)$ be a continuous function. Then
	\begin{equation}
        \lim_{\delta\to 0^+} \frac{h(K,\bu)-h(K_\delta^\phi,\bu)}{\delta^{\frac{2}{d+1}}} 
            = \left(\frac{H_{d-1}(K,\nabla h_K(\bu))}{a_d^2 \,\phi(\nabla h_K(\bu))^2}\right)^{\frac{1}{d+1}},
	 \end{equation}
	 for almost all $\bu\in\S^{d-1}$.
\end{theorem}
\begin{proof}
    The general idea is to use an approximating paraboloid in any normal boundary point and apply the result of Meyer \& Werner \cite[p.\ 311]{MW:2000} to the floating body of the paraboloid.
    
    By Theorem \ref{thm:approx_parabola} we may consider $\bu\in\S^{d-1}$ such that $\bx:=\nabla h_K(\bu)\in \bd K$ exists and is a normal boundary point with $H_{d-1}(K,\bx)>0$.
    Thus, we may use an equi-affine transformation $\alpha_\bx$ that maps $K$ to a position 
    $\tilde{K}=\alpha_\bx(K)$ where $\bx$ is mapped to the origin $o$, the direction $\bu$ to $-\be_d$ and the tangent hyperplane of $\bd K$ at $\bx$ to the hyperplane $\be_d^\bot\cong \R^{d-1}$. 
    We may also assume that the approximation paraboloid of $\bd K$ at $\bx$ is mapped to a radially symmetric paraboloid where the height in direction $\bu=\bn_K(\bx)$ is not changed, i.e.,
    \begin{equation*}
        P_\kappa^d(r,h) = \alpha_\bx(P^d_{(\bx,K)}(r,h)) = \left\{(\by,z)\in \R^{d-1} \times \R : h\geq z \geq \frac{\kappa}{2r} \|\by\|_2^2\right\},
    \end{equation*}
    for $\kappa := H_{d-1}(K,\bx)^{1/(d-1)}$. Note that $H_{d-1}(P_\kappa^d(r,h),o) = (\kappa/r)^{d-1} = r^{1-d} H_{d-1}(K,\bx)$. Also note that if $\be_1,\dotsc,\be_{d-1}$ are the principal directions of $T_\bx K$ and $\be_d = \bn_K(\bx) = \bu$, then $\alpha_\bx$ is determined by
    $\alpha_\bx(\by) = A(\by-\bx)$ where $A = \mathrm{diag}(\sqrt{\kappa/\kappa_1},\dotsc,\sqrt{\kappa/\kappa_{d-1}}, 1)$,
    and $\kappa_i=\kappa_i(K,x)>0$ are the (generalized) principal curvatures of $\bd K$ at $\bx$.
    
    Next, we set $\tilde{\phi}:=\phi\circ\alpha_{\bx}^{-1}$ and find that
    \begin{equation*}
        \Vol_d^\phi(\alpha_{\bx}^{-1}(B)) = \int_{B} \phi\circ\alpha_{\bx}^{-1}(\by) \left|\det \alpha_\bx\right|\, \dint \by = \Vol_d^{\tilde{\phi}}(B),
    \end{equation*}
    for all Borel $B\subset \tilde{K}$.
    Thus 
    \begin{equation*}
        \tilde{h}_\delta:=h(K,\bu) - h(K_\delta^\phi,\bu) = -h(\tilde{K}_\delta^{\tilde{\phi}}, -\be_d)>0.
    \end{equation*}
    
    By Theorem \ref{thm:approx_parabola} iii) we may approximate $\tilde{K}$ near $o$ by paraboloids. So let $\varepsilon>0$. Then there exists $h_0 = h_0(\varepsilon)>0$ such that
    \begin{equation*}
        P_\kappa^d(1-\varepsilon,h) \subset \tilde{K} \cap H^-(\be_d,h) \subset P_\kappa^d(1+\varepsilon,h) \quad \text{for all $h\in (0,h_0)$.}
    \end{equation*}
    Furthermore, since $\tilde{\phi}$ is continuous there is an open neighborhood $U$ of the origin such that for all $\by\in U\cap \tilde{K}$ we have that
    \begin{equation*}
        (1-\varepsilon) \phi(\bx) \leq \tilde{\phi}(\by) \leq (1+\varepsilon) \phi(\bx).
    \end{equation*}
    Thus there is also $h_1 \in (0,h_0)$ such that $P_\kappa^d(1+\varepsilon,h_1) \subset U$.
    Since $\tilde{h}_\delta = -h(\tilde{K}_\delta^{\tilde{\phi}},-\be_d)$ converges monotone to $0$ as $\delta \to 0^+$, there exists $\delta_0=\delta_0(\varepsilon,K,\phi)>0$, such that $\tilde{h}_\delta < h_1$ for all $\delta < \delta_0$. 
    
    Now, since a cap tangent to the floating body cuts off at least $\delta$, we find that for all $h\in(0,h_1)$
    \begin{align*}
        \delta 
        &\leq \Vol_d^{\tilde{\phi}}\left(\tilde{K}\cap H^-(\be_d,h)\right)
        \leq (1 + \varepsilon) \phi(\bx) \Vol_d\left(P_\kappa^d(1+\varepsilon,h)\right) \\
        &\leq (1 + \varepsilon) \phi(\bx) \int_{0}^h \left(\frac{2(1+\varepsilon)z}{\kappa}\right)^{\frac{d-1}{2}} \Vol_d(B_2^{d-1}) \dint z
        = (1+\varepsilon)^{\frac{d+1}{2}} \phi(\bx) a_d h^{\frac{d+1}{2}} \kappa^{-\frac{d-1}{2}},
    \end{align*}
    which yields
    \begin{equation*}
        \tilde{h}_\delta \geq \frac{1}{1+\varepsilon} \left(\frac{\delta}{a_d \phi(\bx)}\right)^{\frac{2}{d+1}} \kappa^{\frac{d-1}{d+1}} \quad 
        \text{for all $\delta < \delta_0$}.
    \end{equation*}
    Thus
    \begin{equation*}
        \liminf_{\delta\to 0^+} \frac{h(K,\bu) - h(K_\delta^\phi,\bu)}{\delta^{\frac{2}{d+1}}} 
        = \liminf_{\delta\to 0^+} \frac{\tilde{h}_\delta}{\delta^{\frac{2}{d+1}}} 
        \geq \frac{1}{1+\varepsilon} (a_d\phi(\bx))^{-\frac{2}{d+1}} \kappa^{\frac{d-1}{d+1}}.
    \end{equation*}

    For the other direction we note that
    \begin{equation*}
        \tilde{K} \supset \tilde{K} \cap H^-(\be_d,h_1) \supset P_\kappa^d(1-\varepsilon,h_1),
    \end{equation*}
    implies by the monotonicity of the weighted floating body that
    \begin{equation*}
        \tilde{K}_\delta^{\tilde{\phi}} \supset P_\kappa^d(1-\varepsilon,h_1)_{\xi} \quad \text{for all $\delta<\delta_0$,}
    \end{equation*}
    where $\xi := \delta/[(1-\varepsilon)\phi(\bx)]$. We also used the fact that $\tilde{\phi}(\by) \geq (1-\varepsilon)\phi(\bx)$ for all $\by\in P_\kappa^d(1-\varepsilon,h_1)\subset \tilde{K}\cap U$.
    This yields
    \begin{equation*}
        \tilde{h}_\delta = -h(\tilde{K}_\delta^{\tilde{\phi}}, -\be_d) \leq -h(P_\kappa^d(1-\varepsilon,h_1)_{\xi},-\be_d).
    \end{equation*}
    By the results of Meyer \& Werner \cite[p.\ 311]{MW:2000}, we conclude
    \begin{align*}
        \limsup_{\delta\to 0^+} \frac{h(K,\bu) - h(K_\delta^\phi,\bu)}{\delta^{\frac{2}{d+1}}} 
        &\leq ((1-\varepsilon)\phi(\bx))^{-\frac{2}{d+1}} \lim\sup_{\xi\to 0^+} \frac{-h(P_\kappa^d(1-\varepsilon,h_1)_\xi,-\be_d)}{\xi^{\frac{2}{d+1}}}\\
        &= \frac{1}{1-\varepsilon} (a_d\phi(\bx))^{-\frac{2}{d+1}} \kappa^{\frac{d-1}{d+1}}.
    \end{align*}
    Since $\varepsilon>0$ was chosen arbitrarily and since $\kappa^{d-1}=H_{d-1}(K,\bx)$ and $\bx=\nabla h_K(\bu)$ we conclude
    \begin{equation*}
        \lim_{\delta\to 0^+} \frac{h(K,\bu)-h(K_\delta^\phi,\bu)}{\delta^{\frac{2}{d+1}}} = \left(\frac{H_{d-1}(K,\nabla h_K(\bu))}{a_d^2 \phi(\nabla h_K(\bu))^2 }\right)^{\frac{1}{d+1}},
    \end{equation*}
    for almost all $\bu\in\S^{d-1}$.
\end{proof}

\begin{proof}[Proof of Theorem \ref{thm:main_weighted}]
    Using polar coordinates, we find that
	\begin{align*}
		\Vol_d^\psi\left(\big(K_\delta^\phi\big)^\circ\right) - \Vol_d^\psi\left( K^\circ \right) 
		&= \int_{\S^{d-1}} \int_{\rho(K^\circ,\bu)}^{\rho((K_\delta^\phi)^\circ,\bu)}  \psi(t \bu)t^{d-1} \, \dint{t} \, \mathcal{H}^{d-1}(\dint \bu)\\
		&= \int_{\S^{d-1}} \int_{h(K_\delta^\phi,\bu)}^{h(K,\bu)} \psi\left(\bu/s\right) s^{-(d+1)} \, \dint {s}\, \mathcal{H}^{d-1}(\dint \bu).
	\end{align*}
	Lemma \ref{lem:interchange}, the continuity of $\psi$, and the Dominated Convergence Theorem allow to interchange integration and limit. Hence
	\begin{align*}
		\lim_{\delta\to 0^+} \frac{\Vol_d^\psi\big((K_\delta^\phi)^\circ\big) - \Vol_d^\psi(K^\circ)}{\delta^{\frac{2}{d+1}}} 
		&= \int_{\S^{d-1}} \lim_{\delta\to 0^+} \frac{1}{\delta^{\frac{2}{d+1}}}
            \int_{h(K_\delta^\phi,\bu)}^{h(K,\bu)} \psi\left(\bu/s\right) s^{-(d+1)} \,\dint{s}\, \mathcal{H}^{d-1}(\dint \bu) \\
		&= \int_{\S^{d-1}} \frac{\psi(\bu/h(K,\bu))}{h(K,\bu)^{d+1}} 
            \lim_{\delta\to 0^+} \frac{h(K,\bu)-h(K_\delta^\phi,\bu)}{\delta^{\frac{2}{d+1}}}\, \mathcal{H}^{d-1}(\dint \bu).
	\end{align*}
	Thus, by Theorem \ref{thm:limit_weights} and by applying Federer's area formula to the Gauss map $\bu=\bn_K(\bx)$, see Theorem \ref{thm:approx_parabola}, we derive that
	\begin{align*}
		&\lim_{\delta\to 0^+} \frac{\Vol_d^\psi\big((K_\delta^\phi)^\circ\big)-\Vol_d^\psi(K^\circ)}{\delta^{\frac{2}{d+1}}}\\
		&\qquad = a_d^{-\frac{2}{d+1}} \int_{\S^{d-1}} \!\!\frac{\psi(\bu/h(K,\bu))}{h(K,\bu)^{d+1}} \phi(\nabla h_K(\bu))^{-\frac{2}{d+1}} H_{d-1}(K,\nabla h_K(\bu))^{1/(d+1)} \mathcal{H}^{d-1}(\dint \bu)\\
        &\qquad = a_d^{-\frac{2}{d+1}} \int_{\bd K} \kappa_o(K,\bx)^{\frac{d+2}{d+1}} \phi(\bx)^{-\frac{2}{d+1}} \psi(\bx^\circ)\, C_K(\dint \bx).\qedhere
	\end{align*}
\end{proof}

As an immediate application of Theorem \ref{thm:main_weighted} we derive the following proof of Theorem \ref{thm:V1_illumination_body} (see Section \ref{sec:real_analytic} for another proof).

\begin{proof}[Proof of Theorem \ref{thm:V1_illumination_body}]
    We first note that $V_1(K)$ can be expressed, using polar coordinates, by
    \begin{align*}
        V_1(K) &=  \frac{1}{\Vol_{d-1}(B_2^{d-1})} \int_{\S^{d-1}} h_K(\bu) \, \mathcal{H}^{d-1}(\dint \bu) \\
        &= \frac{1}{\Vol_{d-1}(B_2^{d-1})} \int_{\S^{d-1}} \int_{\rho_{K^\circ}(\bu)}^{\infty} \frac{1}{t^2}\, \dint t \, \dint\mathcal{H}^{d-1}(\dint \bu)\\
        &= \frac{1}{\Vol_{d-1}(B_2^{d-1})} \int_{\R^d\setminus K^\circ} \|\bx\|^{-(d+1)}\, \dint\bx.
    \end{align*}
    This was apparently first observed by Glasauer in \cite[Rmk.\ 3]{GG:1997}, see also \cite{BHK:2021, Ludwig:1999}.
    Hence, for $\tilde{\psi}(\bx) = \Vol_{d-1}(B_2^{d-1})^{-1}\|\bx\|_2^{-(d+1)}$, we have
    \begin{align*}
        \Delta_1([K,\bx],K) 
        &= V_1([K,\bx])-V_1(K) =\int_{K^\circ \setminus [K,\bx]^\circ} \tilde{\psi}(\bx)\, \dint \bx
        = \Vol_d^{\tilde{\psi}}\left(K^\circ \cap H^+\left(\frac{\bx}{\|\bx\|_2}, \frac{1}{\|\bx\|_2}\right)\right),
    \end{align*}
    where we used the fact that
    \begin{align*}
        K^\circ \cap H^-\left(\frac{\bx}{\|\bx\|_2}, \frac{1}{\|\bx\|_2}\right) 
        &= \{\by \in K^\circ : \by\cdot \bx \leq 1\} \\
        &= \{\by\in\R^d : \by\cdot [(1-t) \bz+ t\bx] \leq 1 \text{ for all $t\in[0,1]$ and $\bz\in K$}\}\\
        &= [K,\bx]^\circ.
    \end{align*}
    This yields
    \begin{align*}
        \mathcal{I}^{V_1}_\delta(K)^\circ
        &= \{\by\in\R^d : \text{$\bx\cdot \by \leq 1$ for all $\bx\in \R^d$ such that $\Delta_1([K,\bx],K)\leq \delta$}\}\\
        &= \left\{\by : 
            \text{$\bx\cdot \by \leq 1$ for all $\bx \in \R^d$ such that $\Vol_d^{\tilde{\psi}}\left(K^\circ\cap H^+\left(\frac{\bx}{\|\bx\|_2}, \frac{1}{\|\bx\|_2}\right)\right) \leq \delta$}
            \right\}\\
        &= \left\{\by : \text{$\bu\cdot \by\leq t$ for all $\bu \in \S^{d-1}$ and $t\in\R$ such that } \Vol_d^{\tilde{\psi}}\left(K^\circ\cap H^+\left(\bu, t\right)\right) \leq \delta\right\}\\
        &= \bigcap \{ H^- : \Vol_d^{\tilde{\psi}}(K^\circ \cap H^+)\leq \delta\} = (K^\circ)_\delta^{\tilde{\psi}}.
    \end{align*}
    Thus, since $\mathcal{I}_\delta^{V_1}(K)$ is a convex body by \cite[Lem.\ 1]{Schneider:2020}, we conclude that 
    \begin{equation}
        \mathcal{I}_\delta^{V_1}(K) = ((K^\circ)_\delta^{\tilde{\psi}})^\circ.
    \end{equation}
    Applying Theorem \ref{thm:main_weighted} and Lemma \ref{lem:integral_formula} we conclude that
    \begin{align*}
        \lim_{\delta\to 0^+} \frac{\Vol_d(\mathcal{I}_\delta^{V_1}(K)) - \Vol_d(K)}{\delta^{\frac{2}{d+1}}} 
        &=\lim_{\delta\to 0^+} \frac{\Vol_d(((K^\circ)_\delta^{\tilde{\psi}})^\circ) - \Vol_d(K)}{\delta^{\frac{2}{d+1}}} \\
        &= \left(\frac{\Vol_{d-1}(B_2^{d-1})}{a_d}\right)^{\frac{2}{d+1}} \int_{\bd K^\circ} \kappa_o(K^\circ,\by)^{\frac{d+2}{d+1}} \|\by\|^2_2 
            \, C_{K^\circ}(\dint \by) \\
        &= c_d \int_{\bd K} \kappa_o(K,\bx)^{-\frac{1}{d+1}} \frac{1}{(\bx\cdot \bn_K(\bx))^2} \, C_K(\dint \bx)\\
        &= c_d \int_{\bd K} H_{d-1}(K,\bx)^{-\frac{1}{d+1}} \,\mathcal{H}^{d-1}(\dint \bx),
    \end{align*}
    where we used that
    \begin{equation*}
        \frac{\Vol_{d-1}(B_2^{d-1})}{a_d} = \frac{\pi^{\frac{d-1}{2}}}{(2\pi)^{\frac{d-1}{2}}} \frac{\Gamma(\frac{d+1}{2}+1)}{\Gamma(\frac{d+1}{2})} 
        = \frac{d+1}{2^{\frac{d+1}{2}}} = c_d^{\frac{d+1}{2}}.
    \end{equation*}
    This finishes the proof.
\end{proof}

\section{Floating bodies and duality between convex bodies in spaces of constant curvature}

In this section we discuss applications of our previous results to the setting of spherical, hyperbolic and de Sitter space, as well as general real space forms of constant curvature $\lambda\in \R$. We will see that the duality mapping $^*$ (see \eqref{eqn:def_*_sphere} and \eqref{eqn:def_*_hyperbolic}) between convex bodies in spaces of constant curvature is related to the natural duality mapping on closed convex cones in $\R^{d+1}$, or the Lorentz--Minkowski space $\R^{d,1}$. This duality between cones on $\R^{d+1}$, for $d\geq 2$, was characterized by Schneider \cite[Cor.~1]{Schneider:2008} as being the only order-reversing involution up to self-adjoint linear transformations, see also \cite{AAM:2009, AASW:2023, BS:2008}. 
We also refer to \cite[Sec.\ 2]{FS:2019} where non-Euclidean convex bodies and the duality mapping are presented for spherical, hyperbolic and de Sitter space in the model spaces.

\subsection{Convex bodies in spherical space} \label{sec:spherical}

We consider the unit sphere $\S^d := \{\bx\in \R^{d+1}: \|x\|_2=1\}$ as model for the spherical space, with its natural geodesic distance $d_s$ determined by $\cos d_s(\bu,\bv) = \bu\cdot \bv$ for $\bu,\bv\in\S^d$. A spherical convex body $K\subset \S^d$ is a closed subset with non-empty interior such that for any two points $\bu,\bv\in K$ with $d_s(\bu,\bv)<\pi$ the uniquely determined geodesic arc between $\bu$ and $\bv$ is contained in $K$.
Equivalently, $K$ is a spherically convex body if and only if the radial extension $\operatorname{rad} K := \{r\bu : r\geq 0, \bu\in K\}$ is a closed convex cone in $\R^{d+1}$ with non-empty interior. We call $K$ proper, if $K$ is contained in the interior of a half-sphere $H^+(\be) = \{\bu \in\S^d : \bu\cdot \be \geq 0\}$ for some $\be\in\S^d$.

The dual $K^*$ of a spherical convex body $K$ is defined by 
\begin{equation}\label{eqn:def_*_sphere}
    K^* = \left\{\bv \in \S^d: d_s(\bu,\bv)\leq \frac{\pi}{2} \text{ for all $\bu\in K$}\right\} = \bigcap_{\bu \in K} H^+(\bu).
\end{equation}
This duality is related to the usual duality on convex cones by
\begin{equation*}
    \operatorname{rad} K^* = (\operatorname{rad} K)^* = \{\by\in \R^{d+1} : \bx\cdot \by \geq 0 \text{ for all $\bx\in \operatorname{rad} K$}\}.
\end{equation*}

A spherical convex body $K\subset \S^d$ is of class $\cC^{k}_+$, for $k\geq 2$, if $\bd K$ is a $\cC^k$-smooth embedded hypersurface of $\S^d$ and the spherical Gauss--Kronecker curvature $H_{d-1}^s(K,\bu)$ is strictly positive for all $\bu\in\bd K$. Note that if $K$ is of class $\cC^2_+$, then $K$ is strictly convex, i.e., for all boundary points $\bu\in\bd K$ the great subsphere tangent to $\bd K$ at $\bu$ only contains $\bu$. Further note, that if $K\subset \S^d$ is a strictly convex spherical convex body, then $K$ is contained in an open half-sphere.

The gnomonic projection $g_{\be}$ of $\interior H^+(\be)$ to $\R^d$ is the radial projection from the origin $\bo\in\R^{d+1}$ to the tangent plane of $\S^d$ at $\be$, that is, if $\be = \be_{d+1}$, then
\begin{equation*}
    g(\bu) := \left(\frac{u_1}{u_{d+1}},\dotsc,\frac{u_d}{u_{d+1}}\right),
\end{equation*}
where $\bu=(u_1,\dotsc,u_{d+1}) \in \S^d_+ := \interior H^+(\be_{d+1}) =\{\bu\in\S^d : u_{d+1}>0\}$.
The gnomonic projection maps geodesic arcs of $\S^d_+$ to straight lines in $\R^d$ and therefore a spherical convex body $K\subset \S^d_+$ is mapped to a Euclidean convex body $\overline{K}:=g(K)$.
The natural spherical volume $\Vol_d^s$ is mapped by $g$ to the radially symmetric measure $\Vol_d^{\phi_s}$ with density
\begin{equation*}
    \phi_s(\bx) := (1+\|\bx\|^2)^{-(d+1)/2} \qquad \text{for all $\bx\in\R^d$}.
\end{equation*}

If $K\subset \S^+$ is a spherical convex body that contains $\be_{d+1}$ in the interior, then the dual $K^*$ is contained in $\S^+$ and the gnomonic projection of $K^*$ is related to the the polar body of $\overline{K}$ by
\begin{equation}\label{eqn:spherical_polar_projective}
    g(K^*) = -\overline{K}^\circ.
\end{equation}

The spherical floating body $\mathcal{F}_\delta^s K$ can either be defined as an intersection of half-spaces that cut off a spherical volume of at least $\delta$ or as the $\delta$-super-level set of the minimal cap density function. Both definitions are equivalent and connected to the weighted floating body of $\overline{K}$ by 
\begin{equation*}
    g(\mathcal{F}^s_\delta K) = \overline{K}_\delta^{\phi_s}. 
\end{equation*}

We conjugate the spherical floating body with the duality mapping to obtain the operator $\mathcal{F}_\delta^{s,*}$ defined on spherical convex bodies by
\begin{equation*}
    \mathcal{F}_\delta^{s,*} K = (\mathcal{F}_\delta^s K^*)^*.
\end{equation*}
Note that $\mathcal{F}_\delta^{s,*} K$ converges in the Hausdorff metric to $\mathcal{F}_0^{s,*} K = K$ as $\delta\to 0^+$.

\medskip
A geodesic ball $B_{\bu}(\alpha) = \{\bv\in\S^d: d_s(\bu,\bv)\leq \alpha\}$ for $\alpha \in (0,\alpha)$ is a proper spherical convex body of class $\cC^2_+$. The two extremal cases $\alpha=0$ and $\alpha=\frac{\pi}{2}$ are treated in the following
\begin{example}[floating body of a half-sphere / dual floating body of a point]
    For $\bu_1,\bu_2\in\S^d$ the intersection $L(\bu_1,\bu_2) := H^+(\bu_1)\cap H^-(\bu_2)$ is a spherical wedge and
    \begin{equation*}
        \Vol_d^s(L(\bu_1,\bu_2)) = \Vol^s_d(\S^d) \frac{d_s(\bu_1,\bu_2)}{2\pi}.
    \end{equation*}
    Thus, for $\bu\in\S^d$ and $\delta\in[0,\Vol_d^s(\S^d)/4]$, the floating body of $H^+(\bu)$ is 
    \begin{equation*}
        \cF_\delta^s H^+(\bu) = \bigcap\{H^+(\bv) : \bv\in \S^d \text{ such that $\Vol_d^s(\S^d) d_s(\bv,\bu)\leq 2\pi \delta$}\} 
        = B_{\bu}\left(\frac{\pi}{2}-\frac{2\pi\delta}{\Vol_d^s(\S^d)}\right),
    \end{equation*}
    where $B_{\bu}(\alpha)=\{\bv\in\S^d: d_s(\bu,\bv)\leq \alpha\}$ is a geodesic ball of radius $\alpha\in[0,\frac{\pi}{2}]$ with center $\bu$.
    Moreover, for the dual floating body we conclude $\cF_\delta^{s,*} \{\bu\} = B_{\bu}(\pi/2-2\pi\delta/\Vol_d^s(\S^d))^* = B_{\bu}(2\pi\delta/\Vol_d^s(\S^d))$.
    Thus
    \begin{align}\notag
        \Vol_d^s(\cF_{\delta}^s H^+(\bu)) &= \frac{\Vol_d^s(\S^d)}{2} - \Vol_{d-1}^s(\S^{d-1}) \, C_{d}\left(\frac{2\pi \delta}{\Vol_d^s(\S^d)}\right) \\\notag
            &= \frac{\Vol_d^s(\S^d)}{2} - \frac{2\pi \Vol_{d-1}^s(\S^{d-1})}{\Vol_d^s(\S^d)} \delta + o(\delta), \text{ and }\\
        \Vol_d^s(\cF_{\delta}^{s,*} \{\bu\}) &= \Vol_{d-1}^s(\S^{d-1}) \, S_{d}\left(\frac{2\pi\delta}{\Vol_d^s(\S^d)}\right) 
            = \frac{(2\pi)^d\Vol_{d-1}^s(\S^{d-1})}{(\Vol_d^s(\S^d))^d} \delta^{d-1} + o(\delta^{d-1}),\label{ex:sphere_point}
    \end{align}
    for $\delta\to 0^+$, where $C_d(\alpha):= \int_{0}^\alpha (\cos s)^{d-1}\, \dint s$ and $S_d(\alpha):= \int_{0}^\alpha (\sin s)^{d-1}\, \dint s$.
\end{example}

\begin{theorem}[volume derivative of the spherical floating body conjugate by duality]\label{thm:sphere_main}
    Let $K\subset\S^d$ be a spherically convex body of class $\cC^2_+$. Then
    \begin{equation*}
        a_d^{\frac{2}{d+1}} \lim_{\delta\to 0^+} \frac{\Vol_d^s(\mathcal{F}_\delta^{s,*} K)-\Vol_d^s(K)}{\delta^{\frac{2}{d+1}}} 
        = \Omega^s_{-d/(d+2)}(K).
    \end{equation*}
\end{theorem}
\begin{proof}
    Since $K$ is of class $\cC^2_+$, it is strictly convex and therefore contained in an open hemisphere.
    Thus we may assume w.l.o.g.\ that $\be_{d+1}$ is an interior point of $K$ and that $K\subset \S^d_+$.
    The gnomonic projection, Lemma \ref{lem:integral_formula} and Theorem \ref{thm:main_weighted} yield
    \begin{align*}
        &a_d^{\frac{2}{d+1}}\lim_{\delta\to 0^+} \frac{\Vol_d^s(\mathcal{F}_\delta^{s,*} K)-\Vol_d^s(K)}{\delta^{\frac{2}{d+1}}}
        = a_d^{\frac{2}{d+1}}\lim_{\delta\to 0^+} 
            \frac{\Vol_d^{\phi_s}(((\overline{K}^\circ)_\delta^{\phi_s})^\circ)-\Vol_d^{\phi_s}(\overline{K})}{\delta^{\frac{2}{d+1}}}\\
        &\quad =  \int_{\bd \overline{K}^\circ} H_{d-1}(\overline{K}^\circ,\by)^{\frac{d+2}{d+1}} 
            (\by\cdot \bn_{\overline{K}^\circ}(\by))^{-(d+1)} 
                \phi_s(\by^\circ) \phi_s(\by)^{-\frac{2}{d+1}} 
                    \, \Vol_{\bd \overline{K}^\circ}(\dint \by)\\
        &\quad =  \int_{\bd \overline{K}} H_{d-1}(\overline{K},\bx)^{-\frac{1}{d+1}} (\bx \cdot \bn_{\overline{K}}(\bx))^2 
                \phi_s(\bx) \phi_s(\bx^\circ)^{-\frac{2}{d+1}} \, \Vol_{\bd \overline{K}}(\dint \bx)\\
        &\quad =  \int_{\bd \overline{K}} \sqrt{\frac{1+(\bx\cdot \bn_{\overline{K}}(\bx))^2}{1+\|\bx\|_2^2}}  H_{d-1}(\overline{K},x)^{-\frac{1}{d+1}}\,
            \frac{\sqrt{1+(\bx\cdot \bn_{\overline{K}}(\bx))^2}}{(1+\|\bx\|_2^2)^{\frac{d}{2}}} \, 
                \Vol_{\bd \overline{K}}(\dint \bx)\\
        &\quad = \int_{\bd \overline{K}} H_{d-1}^s(\overline{K},\bx)^{-\frac{1}{d+1}} \, \Vol_{\bd \overline{K}}^s(\dint \bx).
    \end{align*}
    In the last equality we used that for $\bx=g(\bu)\in \bd \overline{K}$ we have that 
    \begin{equation}\label{eqn:spherical_gauss_projective}
        H_{d-1}^s(K,\bu) = H_{d-1}^s(\overline{K},\bx) 
        =\left(\frac{1+\|\bx\|_2^2}{1+(\bx\cdot \bn_{\overline{K}}(\bx))^2}\right)^{\frac{d+1}{2}}  H_{d-1}(\overline{K},\bx)
    \end{equation}
    and for any Borel $A\subset \bd K$ we have that
    \begin{equation}\label{eqn:spherical_surf_projective}
        \Vol_{\bd K}^s(A) = \Vol_{\bd \overline{K}}^s(g(A)) = \int_{g(A)} \frac{\sqrt{1+(\bx\cdot \bn_{\overline{K}}(\bx))^2}}{(1+\|\bx\|_2^2)^{\frac{d}{2}}} \, 
                \Vol_{\bd \overline{K}}(\dint \bx),
    \end{equation}
    see \cite[Eq.\ 4.13 and  Eq.\ 4.11]{BW:2016}.
\end{proof}

\begin{example}
    For $\alpha\in(0,\pi/2)$ we have
        \begin{equation*}
            \Omega_{-d/(d+2)}^s(B_{\bu}(\alpha)) = (\cos\alpha)^{-\frac{d-1}{d+1}} (\sin\alpha)^{\frac{(d-1)(d+2)}{d+1}} \Vol_{d-1}^s(\S^{d-1}).
        \end{equation*}
    Note that $\lim_{\alpha\to (\pi/2)^-} \Omega_{-d/(d+2)}^s(B_{\bu}(\alpha)) = +\infty$ and $\lim_{\alpha\to 0^+} \Omega_{-d/(d+2)}^s(B_{\bu}(\alpha)) = 0$, which also follows from \eqref{ex:sphere_point}.
    
\end{example}

\begin{remark}
    We note that $\mathcal{F}_\delta^{s,*} K$ can also be seen as a spherical illumination body $\mathcal{I}^{s,V_{-1}}_\delta K$, see \cite{Werner:1994}, or separation body, see \cite{Schneider:2020}, with respect to the dual volume deviation $\Delta_{-1}(K,L) = V_{-1}(K)+V_{-1}(L) - 2V_{-1}(K\cap L)$ where $V_{-1}(K) = \Vol_d^s(K^*)$. To see this we set
    \begin{equation*}
        [K,{\bv}] = \mathrm{conv}(K\cup \{\bv\}).
    \end{equation*}
    Then
    \begin{equation*}
        K^*\cap H^+(\bv) = K^*\setminus [K,{\bv}]^*,
    \end{equation*}
    which yields
    \begin{align*}
        \mathcal{F}_\delta^{s,*} K 
        = \{\bv : \Vol_{d}^s(K^*\cap H^+(\bv)) \leq \delta\}
        = \{\bv : \Delta_{-1}(K,[K,\bv]) \leq \delta\} = \mathcal{I}^{s,V_{-1}}_\delta K.
    \end{align*}
\end{remark}

\subsection{Convex bodies in hyperbolic and de Sitter space} \label{sec:hyperbolic}

We consider the hyperboloid model of hyperbolic space $\H^d:=\{\bu\in\R^{d+1}_1 : \bu\circ \bu = -1, u_{d+1}>0\}$ as the upper part of the elliptic hyperboloid in the Lorentz--Minkowski space $\R^{d,1} =(\R^{d+1},\circ)$ with the indefinite product $\circ$  defined by $\bu\circ \bu = (u_1^2 + \cdots + u_d^2) - u_{d+1}^2$. The hyperbolic distance $d_h$ between two points $\bu,\bv\in \H^d$ is determined by $\cosh d_h(\bu,\bv) =  -\bu \circ \bv$.

\medskip
In the Lorentz--Minkowski space we distinguish between vectors $\bv \in \R^{d,1}$ that are
\begin{enumerate}
    \item[a)] \emph{space-like}, if $\bv\circ \bv >0$,
    \item[b)] \emph{light-like}, if $\bv\circ \bv =0$, and
    \item[c)] \emph{time-like}, if $\bv\circ \bv < 0$. A time-like vector $\bv=(v_1,\dotsc,v_{d+1})\in\R^{d,1}$ is called \emph{future-directed} if $v_{d+1} = -\bv\circ \be_{d+1} >0$.
\end{enumerate}
Similarly, a linear subspace $L\subset \R^{d,1}$ is called
\begin{enumerate}
    \item[a)] \emph{space-like}, if all vectors in $L$ are space-like,
    \item[b)] \emph{light-like}, if $L$ contains no time-like vectors and at least one light-like vector, and
    \item[c)] \emph{time-like}, if $L$ contains at least one time-like vector.
\end{enumerate}
The closure of all future-, respectively past-, directed vectors is the positive light-cone $L^d_+ = \{ \bv \in \R^{d,1} : \bv \circ \bv \leq 0 \text{ and } v_{d+1}\geq 0\}$, respectively the negative light-cone $L^d_-=-L^d_+$. $L_+^d$ and $L_-^+$ are closed convex cones in $\R^{d,1}$. The light double-cone $L^d=L^d_+\cup L^d_-$ is the set of all light-like and time-like vectors in $\R^{d,1}$.

\medskip
A hyperbolic convex body $K\subset \H^d$ is a compact subset with non-empty interior and such that for any two point $\bu,\bv\in K$ the geodesic arc connecting them is contained in $K$. Equivalently, $K$ is hyperbolic convex body if and only if the radial extension $\operatorname{rad} K$ is a closed convex cone in $\R^{d,1}$ of future-directed vectors, that is, $\operatorname{rad} K$ is contained in the interior of $L^d_+$.

\medskip
The projective dual of $\H^d$, the space of all oriented hyperplanes in $\H^d$, can be identified with the de Sitter space, which has the hyperboloid model $\dS^d_1:=\{\bu\in\R^{d,1} : \bu\circ \bu = 1\}$. Indeed, any vector $\bv\in\dS^d_1$ determines an oriented hyperplane $H(\bv):=\{\bu\in \H^d : \bu\circ \bv=0\}$, i.e.\ a totally geodesic co-dimension $1$ submanifold of $\H^d$, and closed half-spaces by $H^+(\bv):=\{\bu\in\H^d : \bu\circ \bv \leq 0\}$ and $H^-(\bv) := H^+(-\bv)$.
Conversely, any point $\bu \in \H^d$ determines an oriented hyperplane $H(\bu):=\{\bv\in \dS^d_1: \bu\circ\bv=0\}$ and closed half-spaces $H^+(\bu) := \{\bv \in \dS^d_1: \bu\circ \bv \geq 0\}$ and $H^-(\bu) := H^+(-\bu)$.

Note that for $\bu\in \H^d$ the hyperplane $H(\bu)\subset \dS^d_1$ contains only space-like vectors and $H^+(\bu) = \tilde{H}^+(\bu) \cap \dS^d_1$ where $\tilde{H}^+(\bu) := \{\bv \in\R^{d,1} : \bv\circ \bu \leq 0\}\supset L^d_+$. Thus $H^+(\bu)\subset\dS^d_1$ is a \emph{future-directed} closed half-space with compact space-like boundary.

\medskip
Unless empty, the intersection $K\subset \dS_1^d$ of a arbitrary family of closed space-like half-spaces in $\dS_1^d$ is called a \emph{de Sitter convex set} and if $K$ has a compact boundary, then $K$ is called a $\emph{de Sitter convex body}$.
We call $K\subset \dS_1^d$ future-, respectively past-, directed if $K$ can be written as an exclusive intersection of half-spaces $H^+(\bu)$, respectively half-spaces $H^-(\bu)$, for $\bu\in \H^d$. 

We call a de Sitter convex body $K\subset \dS_1^d$ \emph{proper}, if it does not contain a pair of antipodal points. Note that $K$ is proper, if and only if it is contained in an open de Sitter half-space.

\begin{example} Some examples of de Sitter convex bodies are the following:
    \begin{enumerate}
        \item A closed half-space $H^+(\bu)$, for $\bu\in \H^d$, is a future-directed de Sitter convex body. It is not proper.
        
        \item \emph{Space-like subspaces:} hyperplanes $H(\bu)$, for $\bu\in\H^d$, as well as any non-empty intersection of a collection of hyperplanes, are de Sitter convex bodies that are neither future- nor past-directed. In particular, the set of antipodal points $\{\pm \bv\}$ for $\bv\in \dS_1^d$ is a de Sitter convex body. These sets are exactly the intersection of $\dS^d_1$ with space-like linear subspaces in $\R^{d,1}$, that is, compact great-spheres of $\dS^d_1$ that are congruent to $\S^{k-1} = \{(v_1,\dotsc,v_k,0,\dotsc,0) : \sum_{i=1}^k v_i^2 = 1\} \subset \R^{d,1}$ if the dimension of the linear subspace is $k$. 
        
        \item Given a point $\bu \in \H^d$ and a point $\bv\in\dS^d_1$ the signed geodesic distance $d_h(\bu,\bv)$ between the oriented hyperplane $H(\bu)$ and $\bv$ is given by
        \begin{equation*}
            \sinh d_h(\bu,\bv) = -\bu\circ \bv.
        \end{equation*}
        Then for $\alpha\in \R$ and $\bu\in\H^d$, the \emph{de Sitter balls} $C(\bu, \alpha) = \{\bv \in \S^d_1: d_h(\bu,\bv) \geq \alpha\}$ are proper future-directed de Sitter convex bodies for $\alpha >0$. For $\alpha=0$ we have $C(\bu,0) = H^+(\bu)$ and for $\alpha<0$ we have that $C(-\bu,-\alpha)=\dS^d_1\setminus C(\bu,\alpha)$ is a proper past-directed de Sitter convex body.
        
        \item Random $\beta^*$-polytopes introduced by Godland, Kabluchko and Thäle \cite{GKT:2022} are examples of polyhedral proper future-directed de Sitter convex bodies.
    \end{enumerate}
\end{example}

The hyperbolic dual $K^*\subset\dS^d_1$ of a hyperbolic convex body $K$, respectively the dual $L^*\subset \H^d$ of a de Sitter convex body $L$, is defined by
\begin{align}
    K^* &= \bigcap \{H^+(\bu) : \bu \in K\} = \{\bv \in \dS^d_1: \bu\circ \bv \leq 0 \text{ for all $\bu\in K$}\}, \notag\\
    L^* &=  \bigcap \{H^+(\bv) : \bv \in L\} = \{\bu \in \H^d: \bu\circ \bv \leq 0 \text{ for all $\bv\in L$}\}. \label{eqn:def_*_hyperbolic}
\end{align}
Similar to the spherical setting, this duality is related to the duality of the convex cones in $\R^{d,1}$ with respect to the inner product $\circ$, i.e., if $C\subset \R^{d,1}$ is a closed convex cone that is contained in $\interior L^d_+$, then for $K=C\cap \H^d$ we have that $K^* = C^*\cap \dS^d_1$, where
\begin{equation*}
    C^* = \{\bv \in \R^{d,1}: \bu\circ \bv \leq 0 \text{ for all $\bu\in C$}\}.
\end{equation*}
Similarly, if $C$ is a closed convex cone such that $L^d_+\subset \interior C$, then for $L=C\cap \dS^d_1$ we have that $L^* = C^*\cap \H^d$.

\begin{example}[hyperbolic balls and de Sitter balls]
    The hyperbolic dual of a closed hyperbolic ball $B_{\be}(\alpha) = \{\bu\in\H^d: d_h(\bu,\be)\leq \alpha\}$, with center $\be\in\H^d$ and radius $\alpha>0$, is the de Sitter ball $C(\be,\alpha) = \{\bv\in\dS_1^d : d_h(\be,\bv)\geq \alpha\}$. Note that the boundary of $B_{\be}(\alpha)$ is a round sphere (compact totally umbilical hypersurface) with geodesic curvature $\lambda = \coth \alpha > 1$ and the boundary of $C(\be,\alpha)$ is a Riemannian round sphere with geodesic curvature $\lambda^{-1} = \tanh \alpha <1$.
    
    The ``limit'' of geodesic balls $B_{\be}(\alpha)$ when the center $\be$ and radius $\alpha$ move to infinity is a \emph{horoball} $B_{\infty}$ which has an unbounded boundary that is a totally umbilical hypersurface of geodesic curvature $\lambda =1$. The hyperbolic dual of a horoball $B_{\infty}$ is a de Sitter horoball $C_{\infty}=B_{\infty}^*$. $C_\infty$ is a future-directed proper de Sitter convex set and has an unbounded space-like boundary that is a totally umbilical hypersurface of $\dS^d_1$ with geodesic curvature $\lambda^{-1}=1$. Note that we do not consider $B_\infty$ and $C_\infty$ as convex bodies since their boundary is not compact.
\end{example}

\begin{lemma}[Properties of hyperbolic duality on convex bodies]\label{lem:prop_deSitter}
    Let $K\subset \H^d$ be a hyperbolic convex body and $L\subset \dS^d_1$ a proper future-directed de Sitter convex body. Then for the hyperbolic dual body $K^*$, respectively $L^*$, the following holds:
    \begin{enumerate}
     \item[i)] 
     $K^*$ is a proper future-directed de Sitter convex body and $L^*$ is a hyperbolic convex body. In particular, $\be_{d+1}\in \H^d$ is an interior point of $K$, if and only if $K^*$ is contained in the open de Sitter half space $\dS_1^+:=\interior H^+(\be_{d+1})=\{\bu\in\dS^d_1 : u_{d+1} >0\}$. Conversely, $L$ is contained in $\interior H^+(\bu)$, if and only if $\bu\in\H^d$ is an interior point of $L^*$.
     
     \item[ii)] The boundary $\bd K^*\subset \dS_{1}^d$ is space-like and any outer unit normal is negative time directed. Thus the induced metric on the tangent hyperplane at any boundary point of $K^*$ is Riemannian.
     
     \item[iii)] For hyperbolic convex bodies $K_1\subset K_2\subset \H^d$ we have that $K_2^*\subset K_1^*\subset \dS^d_1$ and $K_1^*\setminus \interior K_2^*$ is compact.
     
     \item[iv)] We have $(K^*)^*=K$ and $(L^*)^*=L$.
    \end{enumerate}
\end{lemma}
\begin{proof}
    First, i) follows since if $\bu\in \H^d$ is an interior point of $K$, then for all $\bv\in K^*$ we have that $\bu\circ \bv < 0$ and therefore $K^*\subset \interior H^+(\bu)$. In particular, if $\be_{d+1}\in\interior K$, then $K^*\subset \interior H^+(\be_{d+1})= \dS^{+}_1$. The statement for $L$ follows analogously.
    
    Statement ii) follows from the observation that $\bu\in \bd K\subset \H^d$, if and only if $H(\bu)$ is tangent to the cone $(\operatorname{rad} K)^*$ and $\circ$ restricted to $H(\bu)$ is positive definite since $\bu$ is future-directed.
    
    The first statement of iii) follows easily from the definition. To prove the second statement of iii) we assume w.l.o.g.\ that $\be_{d+1}$ is an interior point of $K_1$. Then $K_2^*$ as well as $K_1^*$ are contained in $\dS_1^+ = \interior C(\be_{d+1},0)$. 
    Since $K_2$ is compact, there is $\alpha>0$ such that $K_2$ is contained in the interior of the geodesic ball $B_{\be_{d+1}}(\alpha) = \{\bu \in \H^d: d_h(\be_{d+1},\bu) \leq \alpha\}$ and $B_{\be_{d+1}}(\alpha)^* = C(\be_{d+1},\alpha)$. Thus
    \begin{equation*}
        C(\be_{d+1},0) \supset K_1^*\supset \interior K_2^* \supset \interior C(\be_{d+1},\alpha),
    \end{equation*}
    which yields
    \begin{equation*}
        K_1^*\setminus \interior K_2^* \subset C(\be_{d+1},0)\setminus \interior C(\be_{d+1},\alpha) = \{\bv \in \dS_1^d: 0 \leq d_h(\be_{d+1},\bv) \leq \alpha\}.
    \end{equation*}
    Since $\{\bv \in \dS_1^d: 0 \leq d_h(\be_{d+1},\bv) \leq \alpha\}$ is compact, this shows that $K_1^*\setminus K_2^*$ is compact.

    Finally, for iv) let first $\bu\in K$. Then for all $\bv\in K^*$ we have that $\bu\circ \bv\leq 0$, which yields $\bu\in (K^*)^*$. Thus $K\subseteq (K^*)^*$. 
    
    Conversely, if $\bu\not\in K$, then there exist $\bv\in \dS^d_1$ such that $H(\bv)$ strictly separates $K$ and $\bu$ and we may assume w.l.o.g.\ that $\bu\circ \bv > 0 \geq  \bw \circ \bv$ for all $\bw\in K$. Thus $\bv\in K^*$, and $\bu\circ \bv >0$, which yields $\bu\not\in (K^*)^*$. Similarly it follows that $(L^*)^*=L$.
\end{proof}

\begin{figure}
    \centering
    \begin{tikzpicture}[scale=2.3]
        \draw[->,gray] (-1.5,0) -- (1.5,0) node[below] {$\R^d$};
        \draw[->,gray] (0,-0.5) -- (0,1.5);
        \draw[dotted,gray] (-0.,-0.) -- (2.3,2.3);
        \draw[dotted,gray] (0.,-0.) -- (-2.3,2.3);
        
        \def\alp{-0.7}
        \def\beta{0.9}
        
        \draw[domain=-1.5:1.5, variable=\x, smooth] plot ({sinh(\x)}, {cosh(\x)});
        \draw[thick, red, domain=\alp:\beta, variable=\x, smooth] plot ({sinh(\x)}, {cosh(\x)});
        \node[above, red] at ({sinh(0.3)}, {cosh(0.3)}) {$K$};

        \draw[domain=-0.5:1.5, variable=\x, smooth] plot ({cosh(\x)}, {sinh(\x)});
        \draw[thick, blue, domain=\beta:1.5, variable=\x, smooth] plot ({cosh(\x)}, {sinh(\x)});
        \node[right, blue] at ({cosh(1)},{sinh(1)}) {$K^*$};

        \draw[domain=-0.5:1.5, variable=\x, smooth] plot ({-cosh(\x)}, {sinh(\x)});
        \draw[thick, blue, domain=-\alp:1.5, variable=\x, smooth] plot ({-cosh(\x)}, {sinh(\x)});
        
        \draw[gray] (0,0) -- ({-cosh(\alp)},{-sinh(\alp)}) node[midway, below, yshift=-0.1cm] {$H(\bu)$};
        \draw[gray] (0,0) -- ({cosh(\beta)},{sinh(\beta)});

        \draw[gray] (0,0) -- ({sinh(\alp)},{cosh(\alp)}) node[midway,left, yshift=0.2cm] {$H(\bv)$};
        \draw[gray] (0,0) -- ({sinh(\beta)},{cosh(\beta)});
        
        \node[above] at ({sinh(1)},{cosh(1)}) {$\H^d$};
        \node[right] at ({cosh(-0.5)},{sinh(-0.5)}) {$\dS^d_1$};
        \node[right] at (1.5,1.5) {$L^d_+$};

        \draw[red,->] ({-cosh(\alp)},{-sinh(\alp)}) --++ ({0.5*sinh(\alp)},{0.5*cosh(\alp)}) 
            node[right] {$\bu$};
        \draw[blue,->] ({sinh(\alp)},{cosh(\alp)}) --++ ({-0.5*cosh(\alp)},{-0.5*sinh(\alp)}) 
            node[below] {$\bv$};

        \fill[red] ({sinh(\alp)},{cosh(\alp)}) circle(0.025) node[above] {$\bu$};
        \fill[blue] ({-cosh(\alp)},{-sinh(\alp)}) circle(0.025) node[left] {$\bv$};
        
        \fill[red] ({sinh(\beta)},{cosh(\beta)}) circle(0.025);
        \fill[blue] ({cosh(\beta)},{sinh(\beta)}) circle(0.025);
        
        \fill[red] (0,1) circle (0.025) node[below right] {$\be_{d+1}$};
    \end{tikzpicture}
    \caption{The hyperbolic dual $K^*$ of a hyperbolic convex body $K\subset \H^d$ is a proper future-directed de Sitter convex body in $\dS^d_1$. Both, $\H^d$ and $\dS^d_1$, can be expressed as projective dual hypersurfaces in the Lorentz--Minkowski space $\R^{d,1}$. In this case, the hyperbolic duality reduces to the duality on convex cones that are either contained in the interior of the positive light-cone $L^d_+$ or contain $L^d_+$ in the interior.}
\end{figure}
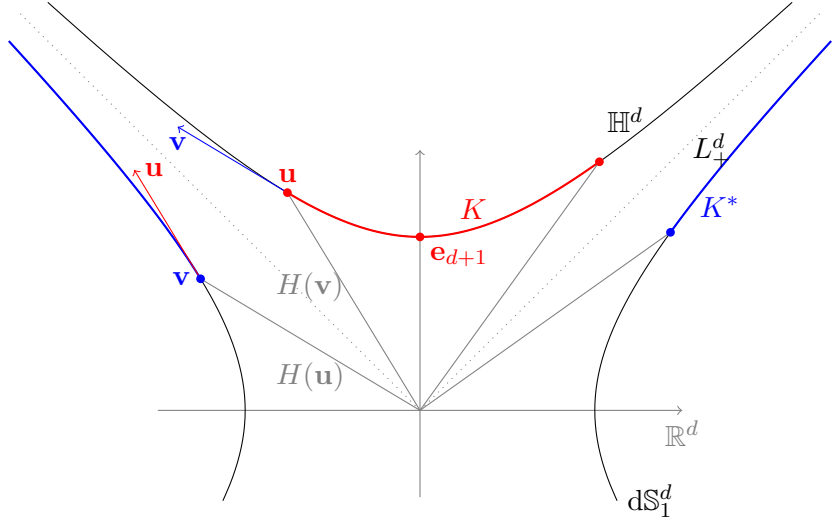

The product $\circ$ induces on the hyperboloid $\H^d$, respectively on $\dS^d_1$, a (pseudo-)Riemannian metric that turns $\H^d$ into a Riemannian, respectively $\dS^d_1$ into a Lorentz, manifold. We denote the natural (pseudo-)Riemannian volume measure on $\H^d$ and $\dS^d_1$ by $\Vol_d^h$ and note that it is invariant with respect to the Lorentz group $O(d,1)$. We also remark that topologically $\dS^d_1$ is homeomorphic to $\S^{d-1}\times \R$ and therefore simply connected for $d\geq 3$.

The gnomonic projection with respect to $\be_{d+1}$ is given by $g(\bu) = (\frac{u_1}{u_{d+1}},\dotsc,\frac{u_d}{u_{d+1}})$. It maps a point in $\H^d$ to the open Euclidean unit ball $\interior B_2^d$ (giving rise to the classical projective model of $\H^d$) and the open de Sitter half space $\dS^+_1 := \interior H^+(\be_{d+1}) = \{\bv \in \dS^d_1: v_{d+1}>0\}$ is mapped to $\R^d\setminus B_2^d$.
The inverse of $g$ is given by
\begin{equation*}
    g^{-1} (\bx) = \begin{cases}
                    \frac{(x_1,\dotsc,x_d,1)}{\sqrt{1-\|\bx\|_2^2}}\in \H^d & \text{if $\|\bx\|_2<1$},\\
                    \frac{(x_1,\dotsc,x_d,1)}{\sqrt{\|\bx\|_2^2-1}}\in \dS_{1}^+ & \text{if $\|\bx\|_2>1$}.
                 \end{cases}
\end{equation*}

Geodesic arcs in $\H^d$ and $\dS_1^+$ are mapped to affine line segments and hyperbolic convex bodies in $\H^d$, respectively future-directed de Sitter convex bodies in $\dS_1^+$, are mapped to Euclidean convex bodies that are contained in the open Euclidean unit ball, respectively to Euclidean convex bodies that contain the Euclidean unit ball in the interior. If $\be_{d+1}$ is an interior point of a hyperbolic convex body $K\subset \H^d$, then the dual body $K^*\subset \dS^+_1$ is mapped to
\begin{equation}\label{eqn:hyperbolic_polar_projective}
    g(K^*) = \left\{\by \in \R^d : (\by,1)\circ (\bx,1) \leq 0 \text{ for all $\bx \in \overline{K}$}\right\} = \overline{K}^\circ.
\end{equation}
Also, since $\overline{K}$ is contained in the interior of $B_2^d$, $\overline{K}^\circ$ contains $B_2^d$ in the interior.
Conversely, if $L\subset \dS^+_1$ is a future-directed de Sitter convex body, then 
\begin{equation*}
    g(L^*) = \overline{L}^\circ.
\end{equation*}

The volume $\Vol_d^h$ on $\H^d$ or $\dS_1^+$ is mapped by the gnomonic projection to the radially symmetric measure $\Vol_d^{\phi_h}$ with density
\begin{equation*}
    \phi_h(\bx) := \left|1-\|\bx\|^2\right|^{-(d+1)/2} \quad \text{for $\|\bx\|_2\neq 1$,}
\end{equation*}
see \cite[Eq.~3.7]{BW:2018} for the hyperbolic case and \cite[Prop.\ 3.17]{GKT:2022} for the de Sitter case.

The hyperbolic floating body $\mathcal{F}_\delta^h$ was introduced in \cite{BW:2018}.
For a hyperbolic convex body $K\subset \H^d$ or a future-directed de Sitter convex body $K\subset \dS^d_1$, the hyperbolic floating body $\mathcal{F}_\delta^h K$ can either be defined as an intersection of half-spaces that cut off a volume of at least $\delta$ or as the $\delta$-super-level set of the minimal cap density function. As in the Euclidean and spherical case, both definitions are also equivalent in the hyperbolic setting to the to the weighted floating body of $\overline{K}=g(K)$, that is,
\begin{equation*}
    g(\mathcal{F}^h_\delta K) = \overline{K}^{\phi_h}_\delta.
\end{equation*}

\begin{example}
    For $\bu_1,\bu_2\in \H^d$, $\bu_1\neq \bu_2$, the intersection $L(\bu_1,\bu_2) = H^+(\bu_1)\cap H^-(\bu_2)$ is a bounded and de Sitter convex body that is neither future- nor past-directed. The hyperbolic volume of $L(\bu_1,\bu_2)$ can be calculated using the gnomonic projecton in $\bu_1$ by
        \begin{align*}
            \Vol_d^h(L(\bu_1,\bu_2)) 
            &= \Vol_d^{\phi_h}( \{ (\bx,t)\in \R^{d-1}\times \R : t\geq (\tanh d_h(\bu_1,\bu_2))^{-1}\} )\\
                &= \int_{\R^{d-1}} \int_{(\tanh d_h(\bu_1,\bu_2))^{-1}}^{\infty} \frac{1}{(\|\bx\|^2+ t^2 -1)^{\frac{d+1}{2}}} \,\dint t\, \dint \bx\\
                &= \int_0^{\tanh d_h(\bu_1,\bu_2)} \frac{1}{1-t^2} \int_{\R^{d-1}} 
                    \frac{1}{(1+\|\bx\|^2)^{\frac{d+1}{2}}}\, \dint \bx \,\dint t\\
                &= \Vol_d^s(\S^d)\frac{d_h(\bu_1,\bu_2)}{2\pi}
        \end{align*}
        This yields in particular, that for $\bu\in \H^d$ and $\delta\geq 0$, the floating body of $H^+(\bu)$ is 
        \begin{align*}
            \mathcal{F}_\delta^h H^+(\bu) 
                &= \bigcap \{H^+(\bv): \bv \in \H^d \text{ such that } \Vol_d^s(\S^d) d_h(\bv,\bu) \leq 2\pi\delta\}
                = C\left(\bu,\frac{2\pi\delta}{\Vol_d^s(\S^d)}\right),
        \end{align*}
        and the dual floating body $\mathcal{F}^{h,*}_\delta  \{\bu\}$ is a geodesic ball in $\H^d$ centered in $\bu$ with radius $2\pi\delta/\Vol_d^s(\S^d)$.
        Thus
        \begin{align*}
            \Vol_d^h(H^+(\bu)\setminus \cF_\delta^h H^+(\bu)) = \Vol_{d-1}^s(\S^{d-1})\, \tilde{C}_d\left(\frac{2\pi\delta}{\Vol_d^s(\S^d)}\right)
            = \frac{2\pi\Vol_{d-1}^s(\S^{d-1})}{\Vol_d^s(\S^d)} \delta + o(\delta)\\
            \Vol_d^h(\cF_\delta^h \{\bu\}) = \Vol_{d-1}^s(\S^{d-1})\, \tilde{S}_d\left(\frac{2\pi\delta}{\Vol_d^s(\S^d)}\right)
            = \frac{(2\pi)^d\Vol_{d-1}^s(\S^{d-1})}{\Vol_d^s(\S^d)^d} \delta^{d-1} + o(\delta^{d-1}),
        \end{align*}
        where $\tilde{C}_d(t) = \int_{0}^t (\cosh t)^{d-1}\, \dint t$ and $\tilde{S}_d(t)=\int_{0}^t (\sinh t)^{d-1}\, \dint t$.
\end{example}

In our previous work \cite{BW:2018}, we already observed, that using the gnomonic projection we can relate the intrinsic surface area measure on $\bd K$ and the generalized Gauss--Kronecker curvature with the $\phi_h$-weighted surface area measure and generalized Gauss--Kronecker curvature of the Euclidean convex body $\overline{K}=g(K)$.

\begin{theorem}\label{thm:local_hyperbolic}
    Let $K\subset \H^{d}$ be hyperbolic convex body that contains $\be_{d+1}$ in the interior, or let $K\subset \dS_1^+$ be future-directed de Sitter convex body. Then $\overline{K}:=g(K)$ is a Euclidean convex body and for any Borel $A\subset \bd K$, $\overline{A}:=g(A)\subset$, we have
    \begin{equation}\label{eqn:hyperbolic_gauss_projective}
        \Vol_{\bd K}^h(A) = \Vol_{\bd \overline{K}}^h(\overline{A}) = \int_{\overline{A}} \frac{\sqrt{\left|1-(\bx\cdot n_{\overline{K}}(\bx))^2\right|}}{|1-\|\bx\|^2|^{\frac{d}{2}}}\, \Vol_{\bd \overline{K}}(\dint \bx),
    \end{equation}
    and
    \begin{equation}\label{eqn:hyperbolic_surf_projective}
        H_{d-1}^h(K,\bu) = H_{d-1}^h(\overline{K},\bx) = H_{d-1}^e(\overline{K},\bx) \left|\frac{1-\|\bx\|^2}{1-(\bx\cdot \bn_{\overline{K}}(\bx))^2}\right|^{\frac{d+1}{2}},
    \end{equation}
    for almost all $\bu\in\bd K$, where $\bx=g(\bu)\in\bd \overline{K}$.
\end{theorem}
\begin{proof}
    For $K\subset \H^d$ this was shown in \cite[Eq.~3.16 and Eq.~3.12]{BW:2018}. For $K\subset \dS_1^+$ the proof is completely analogous, since the tangent hyperplanes to $\bd K$ are space-like and therefore $\circ$ induces a Riemannian metric on $\bd K$. Also see \cite[Prop.\ 3.17]{GKT:2022} for expressions for the pseudo-Riemannian metric in the projective model of $\dS_1^+$ after gnomonic projection.
\end{proof}

\begin{theorem}\label{thm:hyperbolic_float_limit}
    Let $K\subset \H^d$ be a hyperbolic convex body, or let $K\subset \dS_1^d$ be proper future-directed de Sitter convex body. Then
    \begin{equation*}
        a_d^{\frac{2}{d+1}} \lim_{\delta\to 0^+} \frac{\Vol_d^h(K\setminus \mathcal{F}_\delta^h K)}{\delta^{\frac{2}{d+1}}} 
        = \Omega_1^h(K).
    \end{equation*}
\end{theorem}
\begin{proof}
    For $K\subset \H^d$ this was previously obtained in \cite[Thm.\ 1.2]{BW:2018}. W.l.o.g.\ we may assume that $K$ is in a position such that the gnomonic projection can be applied in $\be_{d+1}$. Using the results on the weighted floating body $\overline{K}^{\phi_h}_\delta$ from \cite[Thm.\ 1.1]{BLW:2018} and Theorem \ref{thm:local_hyperbolic} we derive
    \begin{align*}
        a_d^{\frac{2}{d+1}} \lim_{\delta\to 0^+} \frac{\Vol_d^h(K\setminus \mathcal{F}_\delta^h K)}{\delta^{\frac{2}{d+1}}} 
        &= a_d^{\frac{2}{d+1}} \lim_{\delta\to 0^+} \frac{\Vol_d^{\phi_h}(\overline{K}\setminus \overline{K}_\delta^{\phi_h})}{\delta^{\frac{2}{d+1}}} \\
        &= \int_{\bd \overline{K}} H_{d-1}(\overline{K},\bx)^{\frac{1}{d+1}} \left|1-\|\bx\|^2\right|^{-\frac{d-1}{2}} \,
            \Vol_{\bd \overline{K}}(\dint \bx)\\
        &= \int_{\bd \overline{K}} H_{d-1}^h(\overline{K},\bx)^{\frac{1}{d+1}} \, \Vol_{\bd \overline{K}}^h(\dint \bx). \qedhere
    \end{align*}
\end{proof}

We use the gnomonic projection, our results on the weighted volume of the weighted polar floating body in Theorem \ref{thm:main_weighted} and the tools developed in \cite{BW:2018} to relate the hyperbolic and Euclidean curvature and surface area element on $\bd \overline{K}$. We derive the following:
\begin{theorem}\label{thm:hyperbolic_main}
    Let $K\subset \H^d$ be a hyperbolic convex body, or $K\subset \dS_1^d$ be a future-directed de Sitter convex body, that is of class $\cC^2_+$. Then
    \begin{equation*}
        a_d^{\frac{2}{d+1}} \lim_{\delta\to 0^+} \frac{\Vol_d^h( \mathcal{F}^{h,*}_\delta K \setminus K)}{\delta^{\frac{2}{d+1}}} 
        = \Omega^h_{-d/(d+2)}(K).
    \end{equation*}
\end{theorem}
\begin{proof}
    For $K\subset \H^d$, we may assume w.l.o.g.\ that $\be_{d+1}$ is an interior point of $K$.
    Furthermore, by Lemma \ref{lem:prop_deSitter} iii) we notice that $\Vol_d^h((K_\delta)^*\setminus K^*)$ is finite for all $\delta>0$ small enough.
    
    If $K\subset \dS_1^d$, then it is proper since it is of class $\cC^2_+$. Hence, we may assume w.l.o.g.\ that $K\subset \dS_1^+$.
    
    Thus in either case we may apply the gnomonic projection and by calculations analogous to the proof of Theorem \ref{thm:sphere_main} we find
    \begin{multline*}
        \lim_{\delta\to 0^+} \frac{\Vol_d^h(\mathcal{F}_\delta^{h,*} K \setminus K)}{\delta^{\frac{2}{d+1}}}\\
        = a_d^{-\frac{2}{d+1}} \int_{\bd \overline{K}} H_{d-1}(\overline{K},\bx)^{-\frac{1}{d+1}}
            \left|\frac{1-\|\bx\|_2^2}{1-(\bx\cdot \bn_{\overline{K}}(\bx))^2}\right|^{\frac{d+2}{2}}
            \frac{\sqrt{\left|1-(\bx\cdot \bn_{\overline{K}}(\bx))^2\right|}}{|1-\|\bx\|_2^2|^\frac{d}{2}}
            \,\Vol_{\bd \overline{K}}(\dint \bx).
    \end{multline*}
    Then the theorem follows by Theorem \ref{thm:local_hyperbolic}.
\end{proof}

We are now ready to establish the basic properties of our derived curvature measures $\Omega^\diamond_{-d/(d+2)}$ as stated in Theorem \ref{thm:dual_formula_non-euclidean}. 

\begin{proof}[Proof of Theorem \ref{thm:dual_formula_non-euclidean}]
    For the invariance with respect to isometries we just note that by definition the surface area and Gauss--Kronecker curvature are intrinsic notions and therefore also $\Omega_{-d/(d+2)}^\diamond$ is invariant with respect to isometries.
    
    For i) and ii), i.e., the valuation property and the lower semi-continuity, we refer to Ludwig \cite[Thm.\ 6]{Ludwig:2010}, who follows Schütt \cite{Schutt:1993} for the valuation property and her own work \cite{Ludwig:2001} for the lower semi-continuity. Ludwig's arguments are again easily adapted to our situation, compare also \cite[Sec.\ 5.1]{BW:2018}.
    
    Finally, iii) follows from Lemma \ref{lem:integral_formula} by using a projective model: let $\overline{K}\subset \R^d$ be a Euclidean convex body that is the projective model of $K$. Since $\overline{K}$ is of class $\cC^2_+$, we derive by Lemma \ref{lem:integral_formula} and equations \eqref{eqn:spherical_gauss_projective}, \eqref{eqn:spherical_surf_projective}, \eqref{eqn:hyperbolic_gauss_projective}, and \eqref{eqn:hyperbolic_surf_projective},
    \begin{align*}
        \Omega_{-d/(d+2)}^\diamond(K) 
            &= \int_{\bd \overline{K}} H^\diamond_{d-1}(\overline{K},\bx)^{-\frac{1}{d+1}} \, \Vol_{\bd \overline{K}}^\diamond(\dint \bx)\\
            &= \int_{\bd \overline{K}} \kappa_o(\overline{K},\bx)^{-\frac{1}{d+1}} \phi^\diamond(\|\bx\|_2)^{-(d+1)} \phi^\diamond(\|\bx^\circ\|_2)^2 
                    \,C_{\overline{K}}(\dint \bx)\\
            &= \int_{\bd \overline{K}^\circ} \kappa_o(\overline{K}^\circ,\by)^{\frac{d+2}{d+1}} \phi^\diamond(\|\by^\circ\|_2)^{-(d+1)} \phi^\diamond(\|\by\|_2)^2 \|\by\|_2
                    \,C_{\overline{K}^\circ}(\dint \by)\\
            &= \int_{\bd \overline{K}^\circ} H^\diamond_{d-1}(\overline{K}^\circ,\by)^{\frac{d+2}{d+1}} \, \Vol_{\bd \overline{K}^\circ}^\diamond(\dint \by),
    \end{align*}
    where 
    \begin{equation*}
        \phi^\diamond(t) := \begin{cases} 
                            \sqrt{1+t^2} & \text{if $\diamond =s$,}\\
                            \sqrt{|1-t^2|} &\text{if $\diamond=h$.}
                         \end{cases}
    \end{equation*}
    This concludes the proof since $\overline{K}^\circ$ is exactly the same as $K^*$ in the projective model for $\diamond = h$, see \eqref{eqn:hyperbolic_polar_projective}, respectively $-K^*$ in the projective model for $\diamond = s$, see \eqref{eqn:spherical_polar_projective}.
\end{proof}

\subsection{Convex bodies in real space forms}\label{sec:real_analytic}

Let $\Sp^d(\lambda)$ be the real space form of dimension $d$ and curvature $\lambda$.
We identify $\Sp^d(\lambda)$ with either a Euclidean sphere in $\R^{d+1}$ if $\lambda>0$, with $\R^d$ if $\lambda=0$, or with a hyperboloid in $\R^{d,1}$ if $\lambda<0$. Indeed, we have
\begin{equation*}
    \Sp^d(\lambda) \cong \begin{cases}
                            \S^d(\lambda) := \{\bx\in \R^{d+1}: \bx\cdot \bx = 1/\lambda\}
                                &\text{if $\lambda>0$},\\
                            \R^d &\text{if $\lambda=0$},\\
                            \H^d(|\lambda|) := \{\bx\in\R^{d,1}: \bx\circ \bx = -1/|\lambda|\}
                                &\text{if $\lambda<0$}.
                         \end{cases}
\end{equation*}
In addition, for $\lambda<0$, we consider the Lorentz space form $\Sp^d_1(\lambda)$ of dimension $d$ and constant curvature $\lambda$, which we identify with the hyperboloid
\begin{equation*}
    \Sp^d_1(\lambda) \cong \dS^d_1(|\lambda|) := \{\bx\in\R^{d,1} : \bx\circ\bx = 1/|\lambda|\}.
\end{equation*}

A convex body $K\subset \Sp^d(\lambda)$ is a compact geodesically convex subset such that for any two points $\bx,\by\in K$, if there is a unique geodesic segment connecting them, then it is contained in $K$. Equivalently, for $\lambda\neq 0$ we have that $K\subset\Sp^d(\lambda)$ is convex if and only if $\operatorname{rad} K$ is a closed convex cone in $\R^{d+1}$ if $\lambda>0$, or a cone contained in the interior of the light-cone in $\R^{d+1}_1$ if $\lambda<0$.

The dual body $K^*\subset\Sp^d(1/\lambda)$ for $\lambda \geq 0$, respectively $K^*\subset \Sp^d_1(1/\lambda)$ for $\lambda <0$, of a convex body $K\subset \Sp^d(\lambda)$ is defined by
\begin{equation*}
    K^* = \begin{cases}
            \{\by\in \Sp^d(1/\lambda) : \bx\cdot \by \geq 0 \text{ for all $\bx\in K$}\} &
                \text{if $\lambda>0$},\\
            K^\circ &\text{if $\lambda=0$},\\
            \{\by\in \Sp^d_1(1/\lambda) : \bx\circ \by \leq 0 \text{ for all $\bx\in K$}\} &
                \text{if $\lambda<0$}.
          \end{cases}
\end{equation*}
Note that
\begin{equation*}
    \operatorname{rad} K^* = (\operatorname{rad} K)^* = \begin{cases}
                                                            \{\bv\in\R^{d+1}: \bu\cdot \bv \geq 0\} &
                                                                \text{if $\lambda>0$},\\
                                                            \{\bv\in\R^{d,1}: \bu\circ \bv \leq 0\} &
                                                                \text{if $\lambda<0$}.
                                                        \end{cases}
\end{equation*}

\begin{theorem}
    Let $\lambda\neq 0$ and $K\subset \Sp^d(\lambda)$ be a convex body of class $\cC^2_+$. Then
    \begin{equation*}
        H^\lambda_{d-1}(K,\bu) \, H^{1/\lambda}_{d-1}(K^*,\bu^*) = 1 \quad \text{for all $\bu\in\bd K$},
    \end{equation*}
    where $\bu^* := \bn_K(\bu)$ is the uniquely determined normal boundary point of $K^*$ such that $\bu\cdot \bu^*=0$ if $\lambda>0$, respectively $\bu\circ \bu^*=0$ if $\lambda<0$.
\end{theorem}
\begin{proof}
    Since $K$ is of class $\cC^2_+$ we may assume that, for $\lambda >0$, $K$ is contained in an open hemisphere.
    Thus we assume w.l.o.g.\ that $K$ contains $\be_{d+1}$ in the interior and if $\lambda>0$, then we also assume that $K$ is contained in the open half-sphere $\S^d(\lambda) \cap H^+(\be)$.
    Thus, using the gnomonic projection, we may identify $K$ with a Euclidean convex body $\overline{K}:=g^\lambda(K)\subset \R^d$ that contains the origin in the interior in the Euclidean model of $\Sp^d(\lambda)$. 
    Here we consider the gnomonic projection $g^\lambda$ defined by
    \begin{equation*}
        \bx = g^\lambda(\bu) := \frac{1}{\sqrt{|\lambda|}} \left(\frac{u_1}{u_{d+1}},\dotsc,\frac{u_d}{u_{d+1}}\right).
    \end{equation*}
    Then
    \begin{equation*}
        \by = g^{1/\lambda}(\bu^*) = \sqrt{|\lambda|} \left(\frac{u_1^*}{u_{d+1}^*},\dotsc,\frac{u_d^*}{u_{d+1}^*}\right)
    \end{equation*}
    and since $\sum_{i=1}^d u_i u_i^*+\mathrm{sign}(\lambda)u_{d+1}u_{d+1}^* = 0$, we conclude
    \begin{equation*}
        \bx\cdot \by 
        = \frac{\sum_{i=1}^d u_i u_i^*}{u_{d+1} u_{d+1}^*} 
        = -\mathrm{sign}(\lambda).
    \end{equation*}
    Hence, for $\lambda>0$, $\bx^\circ = -\by$, respectively for $\lambda<0$, $\bx^\circ=\by$. This yields
    \begin{equation}\label{eqn:lambda_polar}
        g^{1/\lambda}(K^*) = \begin{cases}
                                -\overline{K}^\circ & \text{for $\lambda>0$},\\
                                \overline{K}^\circ & \text{for $\lambda<0$}.
                             \end{cases}
    \end{equation}
    Thus
    \begin{align*}
        H^\lambda_{d-1}(K,\bu) H_{d-1}^{1/\lambda}(K^*,\bu^*)
        &=H^\lambda_{d-1}(\overline{K},\bx) H_{d-1}^{1/\lambda}(\overline{K}^\circ,\bx^\circ)\\
        &=H_{d-1}(\overline{K},\bx) H_{d-1}(\overline{K}^\circ, \bx^\circ)
            \left(\frac{1+\lambda\|\bx\|^2_2}{1+\lambda/\|\bx^\circ\|^2_2}
            \cdot \frac{1+\|\bx^\circ\|_2^2/\lambda}{1+1/(\lambda\|\bx\|^2_2)}
                \right)^{\frac{d+1}{2}}\\
        &= \kappa_o(\overline{K},\bx) 
            \, \kappa_o(\overline{K}^\circ,\bx^\circ)
        =1.
    \end{align*}
    Here we used $\|\bx^\circ\|_2 = (\bx\cdot \bn_{\overline{K}}(\bx))^{-1}$ and 
    $\|\bx\|_2 = (\bx^\circ\cdot \bn_{\overline{K}^\circ}(\bx^\circ))^{-1}$, as well as,
    \begin{equation}\label{eqn:lambda_gauss}
        H_{d-1}^{\lambda}(\overline{K},\bx) 
        = H_{d-1}(\overline{K},\bx) \left(\frac{1+\lambda\|\bx\|^2_2}{1+\lambda(\bx\cdot \bn_{\overline{K}}(\bx))^2}\right)^{\frac{d+1}{2}}
        = H_{d-1}(\overline{K},\bx)\left(\frac{1+\lambda\|\bx\|^2_2}{1+\lambda/\|\bx^\circ\|^2_2}\right)^{\frac{d+1}{2}},
    \end{equation}
    see \cite[Eq.~3.24]{BW:2018}, and \eqref{eqn:kappa_0-inv}.
\end{proof}

Similar to Theorems \ref{thm:sphere_main} and \ref{thm:hyperbolic_main} we may derive the following theorem. We use Theorem \ref{thm:main_weighted} and tools developed in \cite{BW:2018}.
\begin{theorem}\label{thm:main_lambda}
    Let $\lambda\neq 0$ and let $K\subset\Sp^d(\lambda)$ be a convex body of class $\cC^2_+$. Then
    \begin{equation*}
        a_d^{\frac{2}{d+1}} \lim_{\delta\to 0^+} \frac{\Vol_d^{\lambda}(\mathcal{F}_\delta^{\lambda,*} K) - \Vol_d^\lambda(K)}{\delta^{\frac{2}{d+1}}}
        = \frac{1}{|\lambda|} O^\lambda_{-d/(d+2)}(K).
    \end{equation*}
\end{theorem}
\begin{proof}
    If $\lambda>0$, then we may assume that $K$ is contained in an open hemisphere since $K$ is of class $\cC^2_+$. Thus, for all $\lambda\neq 0$, using the gnomonic projection, we may identify $K$ with a Euclidean convex body $\overline{K}\subset \R^d$ that contains the origin in the interior in the Euclidean model of $\Sp^d(\lambda)$. 
    We calculate, for $\phi_{\lambda}(\bx) = |1+\lambda\|\bx\|^2_2|^{-(d+1)/2}$, that
    \begin{align*}
        a_d^{\frac{2}{d+1}} \lim_{\delta\to 0^+} \frac{\Vol_d^{\lambda}(\mathcal{F}_\delta^{\lambda,*} K) - \Vol_d^\lambda(K)}{\delta^{\frac{2}{d+1}}}
        &= a_d^{\frac{2}{d+1}} \lim_{\delta\to 0^+} \frac{\Vol_d^{\phi_{\lambda}}(((\overline{K}^\circ)_\delta^{\phi_{1/\lambda}})^\circ) -\Vol_d^{\phi_\lambda}(\overline{K})}
            {\delta^{\frac{2}{d+1}}}\\
        &= \int_{\bd \overline{K}} \kappa_o(K,\bx)^{-\frac{1}{d+1}} \phi_{\lambda}(\bx) 
            \phi_{1/\lambda}(\bx^\circ)^{-\frac{2}{d+1}} \, C_{\overline{K}}(\dint x)\\
        &= \frac{1}{|\lambda|} \int_{\bd \overline{K}} H_{d-1}^\lambda(\overline{K},\bx)^{-\frac{1}{d+1}} \, \Vol^\lambda_{\bd \overline{K}}(\dint \bx),
    \end{align*}
    where in the last equation we used \eqref{eqn:lambda_gauss} and the fact that for any Borel $A\subset \bd K$ we have that
    \begin{equation*}
         \Vol_{\bd K}^\lambda(A) = \Vol_{\bd\overline{K}}^\lambda(g^{\lambda}(A)) 
         = \int_{g^\lambda(A)} \frac{\sqrt{\lambda + \|\bx^\circ\|^2}}{(1+\lambda \|\bx\|_2^2)^{d/2}} \, C_{\overline{K}}(\dint x),
    \end{equation*}
    see \cite[Eq.\ 3.23]{BW:2018}. 
\end{proof}

\begin{remark}
    The proof of Theorem \ref{thm:main_lambda} can easily be adapted to show that for $\lambda <0$ the statement holds true for a de Sitter convex body $K\subset \Sp_1^d(\lambda)$ of class $\cC^2_+$, that is,
    \begin{equation}
        a_d^{\frac{2}{d+1}} \lim_{\delta\to 0^+} \frac{\Vol_d^{\lambda}(\mathcal{F}_\delta^{\lambda,*} K\setminus K)}{\delta^{\frac{2}{d+1}}}
        = \frac{1}{|\lambda|} O^\lambda_{-d/(d+2)}(K).
    \end{equation}
\end{remark}

By fixing $\overline{K}\subset\R^d$ in the projective model of $\Sp^d(\lambda)$ and rescaling $\delta$ with respect to $\lambda$ for $\lambda\to 0$ we derive another proof of Theorem \ref{thm:V1_illumination_body} (see Section \ref{sec:weighted_polar_volume} for the first proof).
\begin{proof}[Proof of Theorem \ref{thm:V1_illumination_body}]
    Let $\overline{K}\subset \R^d$ be a convex body of class $\cC_2^+$ that contains the origin in the interior.
    We observe, for $t\in\R$ and $\bu\in\S^{d-1}$, that
    \begin{align*}
    \lim_{\lambda \to 0} \lambda^{-\frac{d+1}{2}} \Vol_d^{\phi_{1/\lambda}}(\overline{K}^\circ \cap H^+(\bu, t)) 
        &=\lim_{\lambda\to 0^+} \int_{\overline{K}^\circ \cap H^+(\bu, t)} \frac{1}{(\lambda+\|\by\|_2^2)^{\frac{d+1}{2}}} \, \dint \by \\
        &= \int_{\overline{K}^\circ \cap H^+(\bu, t)} \|\by\|^{-(d+1)} \, \dint \by\\
        &= \Vol_{d-1}(B_2^{d-1}) \Delta_1([\overline{K},\bx],\overline{K}),
    \end{align*}
    where $\bx = \frac{\bu}{t}$.
    Thus, for $\delta_\lambda = \Vol_{d-1}(B_2^{d-1}) \lambda^{\frac{d+1}{2}} \delta$, and since in the projective model $\cF^{\lambda,*}_{\delta_\lambda} \overline{K} = \left[(\overline{K}^\circ)_{\delta_\lambda}^{\phi_{1/\lambda}}\right]^\circ$, we derive
    \begin{align*}
        \lim_{\lambda\to 0} \rho_{\cF^{\lambda,*}_{\delta_\lambda} \overline{K}}(\bu)
        = \lim_{\lambda\to 0} \frac{1}{ h_{(\overline{K}^\circ)_{\delta_\lambda}^{\phi_{1/\lambda}}}(\bu)} 
        = \rho_{\mathcal{I}_\delta^{\Delta_1}\overline{K}}(\bu),
    \end{align*}
    for all $\bu\in\S^{d-1}$.
    So $\mathcal{F}_{\delta_\lambda}^{\lambda,*} \overline{K} \to \mathcal{I}^{V_1}_\delta \overline{K}$ for $\lambda \to 0$ and \eqref{eqn:olaf} follows by Theorem \ref{thm:main_lambda}.
\end{proof}

\subsection{Real-analytic extension} \label{sec:star_polarity}

We notice that for $\lambda\to 0$ Theorem \ref{thm:main_lambda} does not yield Lutwak's centro-affine curvature measure $\as_{-d/(d+2)}$, that is, Theorem \ref{thm:main_weighted} with uniform weights. One reason is, that the polarity $^\circ$ on Euclidean convex bodies in $\R^d$ depends on the position of the origin $o\in\R^d$ and thus Theorem \ref{thm:main_weighted} with uniform weights is not translation invariant. However, Theorem \ref{thm:main_lambda} is invariant with respect to all isometries of $\Sp^d(\lambda)$ and the duality mapping $^*$ does not depend on the position of a fixed point. Thus, by fixing a point $\be\in\Sp^d(\lambda)$ and rescaling the convex body $K^*$ with respect to this point we can define a \emph{$\be$-polarity mapping} on convex bodies in $\Sp^d(\lambda)$ that contain $\be$ in the interior.
For $\lambda\in \R$, $\lambda\neq 0$, we set
\begin{equation*}
        \tan_{\lambda}(\alpha) :=\begin{cases}
                                 (\tan \sqrt{\lambda} \alpha)/\sqrt{\lambda} & \text{if $\lambda>0$},\\
                                 (\tanh \sqrt{|\lambda|} \alpha)/\sqrt{|\lambda|} & \text{if $\lambda<0$}.
                                \end{cases}
\end{equation*}

\begin{definition}[$\be$-polarity]
    Let $\lambda>-1$, $\lambda\neq 0$, and fix $\be\in\Sp^d(\lambda)$. Then the $\be$-polar body $K^\be\subset \Sp^d(\lambda)$ of a convex body $K\subset \Sp^d(\lambda)$ that contains $\be$ in the interior and
    \begin{enumerate}
     \item[i)] if $\lambda>0$, is contained in the interior of the open half-sphere $\Sp^d_{+}(\lambda)$ with center in $\be$, or,
     \item[ii)] if $-1<\lambda<0$, contains the geodesic ball $B_{\be}\left(R(\lambda)\right)$ in the interior, where $\tan_{\lambda} R(\lambda)=\sqrt{|\lambda|}$,
    \end{enumerate}
    is defined by
    \begin{equation*}
        K^{\be} = (g^{\lambda}_{\be})^{-1} (g_{\be}^\lambda(K)^\circ).
    \end{equation*}
    The gnomonic projection $g^\lambda_{\be}: \Sp^d_+(\lambda)\to \R^d$ for $\lambda>0$, respectively $g^\lambda_{\be}:\Sp^d(\lambda)\to \frac{1}{\sqrt{|\lambda|}} \interior B_2^d\subset \R^d$ for $\lambda<0$, is the diffeomorphism defined by
    \begin{equation*}
        g^\lambda_{\be}(\bu) = \begin{cases}
                            \frac{1}{\lambda} \frac{\bu}{\bu\cdot \be} - \be & \text{if $\lambda>0$},\\
                            \frac{1}{\lambda} \frac{\bu}{\bu\circ \be} - \be & \text{if $\lambda<0$},
                         \end{cases}
    \end{equation*}
    see \cite[Sec.\ 3.2]{BW:2018}.
\end{definition}

For $\be = \frac{1}{\sqrt{|\lambda|}}\be_{d+1} \in \Sp^d(\lambda)$ we have that
\begin{equation*}
    g^\lambda_{\be}(\bu) = \frac{1}{\sqrt{|\lambda|}} \left(\frac{u_1}{u_{d+1}},\dotsc,\frac{u_d}{u_{d+1}}\right).
\end{equation*}
Note that for $\lambda<0$ and a convex body $K\subset \Sp^d(\lambda)$ that contains $B_{\be}(R(\lambda))$ we have 
\begin{equation*}
    \interior g_{\be}^\lambda(K) \supset g_{\be}^\lambda(B_{\be}(R(\lambda))) = (\tan_{\lambda}R(\lambda)) B_2^d = \sqrt{\left|\lambda\right|} B_2^d.
\end{equation*}
Thus $g_{\be}^\lambda(K)^\circ \subset \frac{1}{\sqrt{|\lambda|}} \interior B_2^d$ and therefore $K^{\be} = (g^{\lambda}_{\be})^{-1}(g_{\be}^\lambda(K)^\circ)$ is well-defined.

Like the usual polarity $^\circ$ on convex bodies in $\R^d$ containing the origin in the interior, $\be$-polarity is a order-reversing involution. 

\begin{remark}
    The $\be$-polarity can also be defined via the duality on convex cones as follows: For $\lambda>0$ consider $\S^d(\lambda)\subset \R^{d+1}$ as model for $\Sp^d(\lambda)$. A proper convex body $K\subset \S^d(\lambda)$ determines the closed convex cone $\rad K = \{r\bx:\bx\in K \text{ and $r\geq 0$}\}$ such that $K = (\rad K)\cap \S^d(\lambda)$. If $K$ contains $\be\in \S^d(\lambda)$ in the interior and is contained in the open half-sphere $\S^d(\lambda)\cap \interior H^+(\be)$, then
    \begin{equation}\label{eqn:e-polar_sphere}
        K^\be = R((\rad K)^*) \cap \S^d(\lambda),
    \end{equation}
    where $C^*=\{\by\in\R^{d+1} : \bx\cdot \by \geq 0 \text{ for all $\bx \in C$} \}$ is the dual cone and $R\in \mathrm{GL}(d+1)$ is the linear map determined by $R(\be)=\frac{1}{\lambda}\be$ and $R(\bv)=\bv$ for all $\bv$ orthogonal to $\be$.
    For the proof of this fact we observe that the gnomonic projection $g^\lambda_{\be}(K)$ is determined by
    \begin{equation*}
        (\rad K)\cap (\be + \be^\bot) = g^\lambda_{\be}(K)+\be,
    \end{equation*}
    where we identify $\R^d$ with the hyperplane $\be^\bot = \{\by\in \R^{d+1}: \by\cdot \be = 0\}$, see also Figure \ref{fig:e-polar_sphere}.
    Note that $\lambda\be \in \S^d(1/\lambda)$ and by \eqref{eqn:lambda_polar} we have $g^{1/\lambda}_{\lambda\be}(K^*) = -g^\lambda_{\be}(K)^\circ$. Thus
    \begin{align*}
        (\rad K^\be)\cap (\be+\be^\bot) 
        &= g^\lambda_{\be}(K^\be)+\be
        = (g^\lambda_{\be}(K))^\circ+\be\\
        &= -g^{1/\lambda}_{\lambda\be}(K^*)+\be
        = R(-g^{1/\lambda}_{\lambda\be}(K^*)+\lambda\be)\\
        &= R((\rad K)^*\cap (\lambda\be+\be^\bot))
        = R((\rad K)^*) \cap (\be+\be^\bot).
    \end{align*}
    Thus $\rad K^\be = R((\rad K)^*)$ and \eqref{eqn:e-polar_sphere} follows.
    
    Similarly, for $-1<\lambda<0$, we consider $\H^d(|\lambda|)\subset \R^{d,1}$ as model for $\Sp^d(\lambda)$. A convex body $K\subset \H^d(|\lambda|)$ determines a closed convex cone $(\rad K)=\{r\bx: \bx\in K \text{ and $r\geq 0$}\}\subset L^d_+$ in $\R^{d,1}$ such that $K=(\rad K)\cap \H^d(|\lambda|)$. If $K\subset \H^d(|\lambda|)$ is a convex body that contains the geodesic ball $B_{\be}(R(\lambda))$ in the interior, then
    \begin{equation}\label{eqn:e-polar_hyperbolic}
        K^\be = R((\rad K)^*) \cap \H^d(|\lambda|),
    \end{equation}
    where $C^*=\{\by\in\R^{d,1} : \bx\circ \by \geq 0\}$ is the dual cone in $\R^{d,1}$ and $R\in\mathrm{GL}(d+1)$ is the linear map determined by $R(\be)=\frac{1}{|\lambda|}\be$ and $R(\bv)=\bv$ for all $\bv\in\R^{d,1}$ such that $\bv\circ\be=0$. Then \eqref{eqn:e-polar_hyperbolic} follows similar to \eqref{eqn:e-polar_sphere} where we identify $\R^d$ with $\be^\bot=\{\by\in\R^{d,1}:\by\circ\be=0\}$ and note that $(\rad K)\cap (\be+\be^\bot) = g^\lambda_{\be}(K)+\be$ and $g^{1/\lambda}_{|\lambda|\be}(K^*) = g^\lambda_\be(K)^\circ$.
\end{remark}

\begin{figure}[t]
    \centering
    \begin{tikzpicture}
        \def\al{38};
        \def\bl{115};
    
        \draw (0,0) circle(2);
        \draw (0,0) circle(0.5);
        
        \draw[gray,->] (0,0) -- (0,3.5);
        \draw[gray] (2.5,0) -- (-2.5,0) node[left] {$\be^\bot$};
        
        \draw[thin, gray] (\al:0.95) arc (\al:\al-90:0.95);
        \draw[thin, gray] (\al:1) arc (\al:\al-90:1);
        
        \draw[thin, gray] (\bl:0.95) arc (\bl:\bl+90:0.95);
        \draw[thin, gray] (\bl:1) arc (\bl:\bl+90:1);
        
        \draw ({3.3*cot(\bl-90)},3.3) -- (0,0) -- ({3.3*cot(90+\al)},3.3) node[left] {$(\rad K)^*$};
        \fill[opacity=0.2] ({3.3*cot(\bl-90)},3.3) -- (0,0) -- ({3.3*cot(90+\al)},3.3);

        \draw[very thick, red] (\bl:2) arc (\bl:\al:2) node[above] {$K$};
        \draw[red] ({3*cot(\bl)},3) -- (0,0) -- ({3*cot(\al)},3) node[right] {$\rad K$};
        \fill[red, opacity=0.3] ({3*cot(\bl)},3) -- (0,0) -- ({3*cot(\al)},3)--cycle;
        \draw[very thick, red] ({2*cot(\al)},2)-- ({2*cot(\bl)},2);

        \draw ({-2.2*cot(\al-90)},-2.2) -- (0,0) -- ({-2.2*cot(\bl-270)},-2.2) node[left] {$-(\rad K)^*$};
        \fill[opacity=0.2] ({-2.2*cot(\al-90)},-2.2) -- (0,0) -- ({-2.2*cot(\bl-270)},-2.2)--cycle;
        
        \draw[very thick] ({\bl-90}:0.5) arc ({\bl-90}:{90+\al}:0.5) node[below left, yshift=0.2cm, xshift=0.1cm] {$K^*$};
        \draw[very thick] ({0.5*cot(90+\al)},0.5) -- ({0.5*cot(\bl-90)},0.5);
        
        \draw[dashed] ({0.5*cot(90+\al)},0.5) -- ({0.5*cot(90+\al)},2);
        \draw[dashed] ({0.5*cot(\bl-90)},0.5) -- ({0.5*cot(\bl-90)},2);
        
        \draw[blue] ({2.7*0.25*cot(90+\al)},2.7) -- (0,0) -- ({2.7*0.25*cot(\bl-90))},2.7) node[right] {$\rad K^\be$};
        \fill[blue, opacity=0.3] ({2.7*0.25*cot(\bl-90))},2.7) -- (0,0) -- ({2.7*0.25*cot(90+\al)},2.7);
        \draw[very thick, blue] ({0.5*cot(90+\al)},2) -- ({0.5*cot(\bl-90)},2);
        \draw[very thick, blue] ({180+atan(4*tan(90+\al))}:2) arc ({180+atan(4*tan(90+\al))}:{atan(4*tan(\bl-90))}:2)
            node[below left] {$K^\be$};
        
        \node[red] at (4.2,2)   {$(\rad K)\cap (\be+\be^\bot)$};
        \node at (3,0.4) {$(\rad K)^*\cap (\lambda\be+\be^\bot)$};
        
        \node[right] at (-45:2) {$\S^d(\lambda)$};
        \node[below] at (-100:0.5) {$\S^d(1/\lambda)$};
        
        \fill (0,2) circle(0.05) node[above right] {$\be$};

    \end{tikzpicture}
    \caption{Sketch for the definition of $\be$-polar body $K^\be$ of a convex body $K\subset \S^d(\lambda)$ such that $\be$ is in the interior of $K$ and $K$ is contained in the open half-sphere with center in $\be$.}
    \label{fig:e-polar_sphere}
\end{figure}
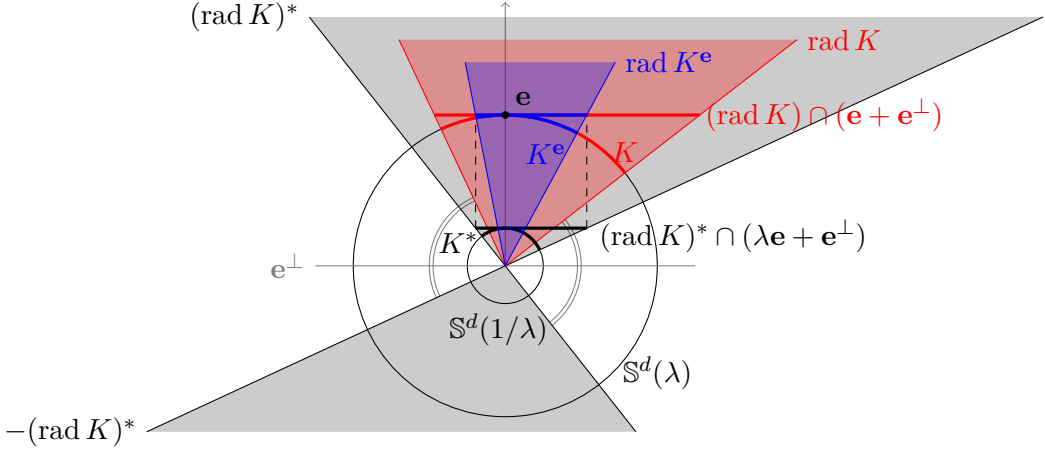

We now consider the floating body conjugated by $\be$-polarity, that is, $\cF_\delta^{\lambda,\be} K := (\cF_{\delta}^\lambda K^\be)^\be$.

\begin{theorem}\label{thm:main_analytic}
     Let $\lambda\neq 0$, $\be\in\Sp^d(\lambda)$ and let $K\subset \Sp^d(\lambda)$ be a convex body of class $\cC^2$ that contains $\be\in\Sp^d(\lambda)$ in the interior. We further assume, that if $\lambda>0$, $K$ is contained in the open half-space centered at $\be$, and, if $-1 < \lambda < 0$, then it contains the closed geodesic ball around $\be$ of radius $R(\lambda)$ in the interior, where $\tan_\lambda R(\lambda) = \sqrt{|\lambda|}$. Then
    \begin{equation*}
        a_d^{\frac{2}{d+1}} \lim_{\delta\to 0^+} \frac{\Vol_d^\lambda(\mathcal{F}_\delta^{\lambda,\be}K) - \Vol_d^\lambda(K)}{\delta^{\frac{2}{d+1}}} 
        = \Omega_{-d/(d+2)}^{\lambda,\be}(K),
    \end{equation*}
    where
    \begin{equation*}
        \Omega_{-d/(d+2)}^{\lambda,\be}(K) :=\int_{\bd K} 
            \left(\frac{H^{\lambda}_{d-1}(K,\bu)}{f^\lambda_\be(K,\bu)^{d+1}}\right)^{-\frac{1}{d+1}}
                f^\lambda_\be(K,\bu) \, \Vol_{\bd K}^\lambda(\dint \bu)
    \end{equation*}
    for
    \begin{equation*}
        f^\lambda_{\be}(K,\bu) := \sqrt{\left|\frac{\lambda + (\tan_{\lambda} d_\lambda(\be,H(K,\bu)))^2}{1+\lambda (\tan_{\lambda} d_\lambda(\be,H(K,\bu)))^2}\right|}.
    \end{equation*}
    Here $H(K,\bu)$ denotes the tangent hyperplane to $K$ at $\bu$, $d_{\lambda}(\be,H(K,\bu))$ is the minimal geodesic distance in $\Sp^d(\lambda)$ of $\be$ to the points in $H(K,\bu)$.
\end{theorem}
\begin{proof}
    We may assume w.l.o.g.\ that $\be=\frac{1}{\sqrt{|\lambda|}}\be_{d+1}$ so that $g^\lambda(\be)=o\in\R^d$ for all $\lambda$.
    Note that if $-1 < \lambda<0$, then, since $\tan_\lambda R(\lambda)=\sqrt{|\lambda|}$, we have $\interior g^\lambda(K)=\overline{K}\supset \sqrt{|\lambda|} B_2^d$. This yields $g^\lambda(K)^\circ \subset \frac{1}{\sqrt{|\lambda|}} \interior B_2^d$ and therefore we may apply $(g^\lambda)^{-1}$.
    Using the gnomonic projection $g^\lambda_{\be}$, Lemma \ref{lem:integral_formula} and Theorem \ref{thm:main_weighted} we derive that
    \begin{align*}
        a_d^{\frac{2}{d+1}} \lim_{\delta\to 0^+} \frac{\Vol_d^\lambda(\mathcal{F}^{\lambda,\be}_\delta K) - \Vol_d^\lambda(K)}{\delta^{\frac{2}{d+1}}}
        &=a_d^{\frac{2}{d+1}} \lim_{\delta\to 0^+} 
            \frac{\Vol_d^{\phi_\lambda}(((\overline{K}^\circ)_\delta^{\phi_\lambda})^\circ) - \Vol_d^{\phi_\lambda}(\overline{K})}{\delta^{\frac{2}{d+1}}}\\
        & =  \int_{\bd \overline{K}} \kappa_o(\overline{K},\bx)^{-\frac{1}{d+1}}
             \phi_{\lambda}(\bx) \phi_{\lambda}(\bx^\circ)^{-\frac{2}{d+1}}
            \, C_{\overline{K}}(\dint \bx)\\
        & =  \int_{\bd \overline{K}} H_{d-1}^\lambda(\overline{K},\bx)^{-\frac{1}{d+1}} \left| \frac{1+\lambda\|\bx^\circ\|_2^2}{\lambda+\|\bx^\circ\|_2^2} \right|
            \, \Vol_{\bd \overline{K}}^\lambda(\dint \bx).
    \end{align*}
    Now $\tan_{\lambda} d_\lambda(\be,H(K,\bu)) = \bx\cdot \bn_{\overline{K}}(\bx) = \|\bx^\circ\|^{-1}_2$, and therefore, for $g^\lambda_{\be}(\bu)=\bx$, we have that
    \begin{equation}\label{eqn:f_lambda}
        f^\lambda_{\be}(K,\bu) = \sqrt{\left|\frac{1+\lambda\|\bx^\circ\|_2^2}{\lambda + \|\bx^\circ\|_2^2}\right|}. \qedhere
    \end{equation}
\end{proof}

Fixing $\overline{K}\subset\R^d$ in the projective model of $\Sp^d(\lambda)$ such that $\be = o$, we note that \eqref{eqn:f_lambda} yields
\begin{equation*}
    \lim_{\lambda\to 0} f^\lambda_{o}(\overline{K},\bx) = \frac{1}{\|\bx^\circ\|_2} = \bx \cdot \bn_{\overline{K}}(\bx),
\end{equation*}
and
\begin{equation*}
    \lim_{\lambda \to 0} \Omega_{-d/(d+2)}^{\lambda,o}(\overline{K}) = \as_{-d/(d+2)}(K).
\end{equation*}
Hence Theorem \ref{thm:main_analytic} gives a real-analytic extension of Theorem \ref{thm:main_weighted} for uniform weights. Notice also, that for $\lambda=\pm 1$ we have $f^\lambda_{\be}(K,\cdot) \equiv 1$ and therefore there is no dependence on $\be$ and Theorem \ref{thm:main_analytic} gives the same result as Theorem \ref{thm:sphere_main}.

\medskip
\textbf{Acknowledgments.} Elisabeth Werner was supported by NSF grant \texttt{DMS-2103482}.

\medskip
\parindent=0pt

\bigskip
\begin{samepage}
	Florian Besau\\
	Institute of Discrete Mathematics and Geometry\\
	Technische Universität Wien (TU Wien)\\
	Wiedner Hauptstrasse 8--10, 1040 Vienna, Austria\\
	e-mail: florian.besau@tuwien.ac.at
\end{samepage}

\medskip 
\begin{samepage}
	Elisabeth M.\ Werner\\
	Department of Mathematics, Applied Mathematics and Statistics\\
	Case Western Reserve University (CWRU)\\
	10900 Euclid Avenue, Cleveland, Ohio 44106, USA\\
	e-mail: elisabeth.werner@case.edu
\end{samepage}

\end{document}